\documentclass[11pt]{amsart}
\usepackage{amstext,amssymb,amsmath,amsbsy,dsfont,tikz}

 \usepackage[left=2.55cm,right=2.55cm,top=3.2cm,bottom=3.2cm]{geometry}
 
\usepackage{amsmath,amssymb,latexsym,dsfont}
\usepackage[small]{caption}
\usepackage{graphicx,color,mathrsfs,tikz}
\usepackage{subfigure,color}
\usepackage{cite}
\usepackage[colorlinks=true,urlcolor=blue,
citecolor=red,linkcolor=blue,linktocpage,pdfpagelabels,
bookmarksnumbered,bookmarksopen]{hyperref}
\usepackage[italian,english]{babel}
\usepackage{units}
\usepackage{enumitem}
  
\usepackage{hyperref}
\usepackage{cleverref}

\usepackage{tikz}
\usetikzlibrary{intersections}

\usepackage{amscd}
\usepackage{amsfonts}
\usepackage{indentfirst}
\usepackage{verbatim}
\usepackage{amsmath}
\usepackage{amsthm}
\usepackage{enumerate}
\usepackage{graphicx}
\usepackage{color}
\usepackage[OT1]{fontenc}
\usepackage[latin1]{inputenc}
\usepackage[english]{babel}
\usepackage{amssymb,esint}

\newtheorem{theorem}{Theorem}[section]
\newtheorem{lemma}{Lemma}[section]
\newtheorem{proposition}{Proposition}[section]
\newtheorem{corollary}{Corollary}[section]
\newtheorem{remark}{Remark}[section]

\newtheorem{definition}{Definition}[section]
\newcommand{\tr}{^\mathsf{T}}

\newcommand{\ds}{\displaystyle}
  \newcommand{\beq}{\begin{equation}}
\newcommand{\eeq}[1]{\label{#1}\end{equation}}

      \newcommand{\R}{{\mathbb{R}}}

      \newcommand{\mN}{{\mathbb{N}}}

      \newcommand{\hgamma}{\gamma}

      \newcommand{\hL}{\hat L}
      
       \newcommand{\wSigma}{\widetilde \Sigma}

      \newcommand{\curl}{\operatorname{curl}}
      \newcommand{\dive}{\operatorname{div}}
      \newcommand{\Div}{\operatorname{div}}
      \newcommand{\N}{\mathbb{N}}

      \newcommand{\loc}{\operatorname{loc}}
      
      \newcommand{\eps}{\varepsilon}
      \newcommand{\mR}{\mathbb{R}}

\newcommand{\pB}{\partial B}

\newcommand{\testfn}{e^{-\beta r^{-p}}}
\newcommand{\testfnn}{e^{2\beta r^{-p}}}

\newcommand{\Y}{{\mathcal O}}

\newcommand{\pY}{{\partial \Y}}
      \newcommand{\T}{{\mathcal T}}

      \newcommand{\bH}{\operatorname{\textbf{H}}}

      \newcommand{\wtE}{\widetilde{E}}
      \newcommand{\wtH}{\widetilde{H}}

     \newcommand{\supp}{\mbox{supp }}
           \newcommand{\dsp}{\displaystyle}

      \makeatletter
      \def\@setcopyright{}
      \def\serieslogo@{}
      \makeatother

\numberwithin{equation}{section}

\newcommand{\cH}{{\mathcal H}}
\newcommand{\cE}{{\mathcal E}}

\newcommand{\cO}{{\mathcal O}}

\newcommand{\hH}{{\hat H}}
\newcommand{\hE}{{\hat E}}

\newcommand{\tH}{\widetilde H}
\newcommand{\tE}{\widetilde E}

\newcommand{\tK}{\widetilde K}

\newcommand{\tx}{\widetilde x}
\newcommand{\ty}{\widetilde y}
\newcommand{\tz}{\widetilde z}

\newcommand{\tZ}{\widetilde Z}

\newcommand{\dist}{\mbox{dist}}

\newcommand{\hy}{\hat y}
\newcommand{\hj}{\hat j}

\newcommand{\hM}{\hat M}

\newcommand{\hB}{\hat B}

\newcommand{\hU}{\hat U}

\newcommand{\hg}{\hat g}

\newcommand{\wtD}{\widetilde D}

\newcommand{\wtSigma}{\widetilde \Sigma}
 
\newcommand{\rr}{\hat r}

\newcommand{\M}{{\mathcal M}}

\newcommand{\hLambda}{\hat \Lambda}

\newcommand{\hhgamma}{\hat \gamma}

\newcommand{\teps}{\widetilde \eps}
\newcommand{\tmu}{\widetilde \mu}

          \newcommand{\hmu}{{\hat \mu}}
      \newcommand{\heps}{{\hat \eps }}

      \newcommand{\I}{{\mathcal I}}

           \newcommand{\mup}{\mu^+}
           \newcommand{\mun}{\mu^-}

           \newcommand{\ep}{\eps^+}
           \newcommand{\en}{\eps^-}

   \newcommand{\cF}{{\mathcal F}}

   \newcommand{\cG}{{\mathcal G}}

   \newcommand{\bF}{{\bf F}}

   \newcommand{\D}{\mbox{Data}}

\usepackage{rotating}

\begin{document}

 \title[Cloaking via anomalous localized resonance]{
Cloaking property of a plasmonic structure in  doubly complementary media and three-sphere inequalities with partial data}

\author[H.-M. Nguyen]{Hoai-Minh Nguyen}

\address[Hoai-Minh Nguyen]{Department of Mathematics, EPFL SB CAMA, Station 8,  \newline\indent
	 CH-1015 Lausanne, Switzerland.}
\email{hoai-minh.nguyen@epfl.ch}

\begin{abstract} 

We investigate cloaking property of negative-index metamaterials in the time-harmonic electromagnetic setting for the so-called doubly complementary media. These are media consisting   of negative-index metamaterials in a shell  (plasmonic structure) and positive-index materials in its complement for which  the shell is complementary to a part of the core and a part of the exterior of the core-shell structure.  We show that an {\it arbitrary} object is invisible when it  is  placed close  to a plasmonic structure of a doubly complementary medium as long as its cross section is smaller than a threshold given by the property of the plasmonic structure. To handle the loss of the compactness and of the ellipticity of the modeling Maxwell equations with sign-changing coefficients, we first obtain  Cauchy's problems associated with two Maxwell systems using reflections. We then derive information from them,  and combine it with the removing localized singularity technique to deal with the localized resonance.   A central part of the analysis on the Cauchy's problems is to establish three-sphere inequalities with partial data for  general elliptic systems,  which are interesting in themselves.  The proof of these inequalities first relies on an appropriate change of variables,  inspired by conformal maps,  and  is then based on Carleman's estimates for a class of degenerate elliptic systems. 

\end{abstract}


\maketitle


\medskip 


\tableofcontents


\section{Introduction}
Negative-index metamaterials  are artificial structures whose  refractive index has a negative value over some frequency range. Their existence was  postulated by Veselago  in 1964  \cite{Veselago} and  confirmed experimentally by Shelby, Smith, and Schultz in 2001 \cite{SSS01}.   Negative-index metamaterial research  has been a very active topic of investigation  not only because of potentially interesting applications, such as  superlensing \cite{Veselago,PendryNegative}, cloaking \cite{Lai1, MN06, Ng-CALR-O}, biomedical imaging \cite{Jain}, and heat generation \cite{BGQ}, but also because of challenges in understanding their surprising properties.  From a mathematical point of view,   
the subtlety and the challenging in the study of negative-index metamaterials come from the sign-changing  coefficients in the modeling  equations, hence the ellipticity and the compactness  are lost in general. Moreover, localize resonance, i.e.,  a phenomenon in which the field blows up in some regions and remains bounded in some others as the loss goes to 0, might occur.

In this paper, we investigate cloaking property of negative-index metamaterials for electromagnetic waves in the time-harmonic regime for the so-called doubly complementary media. This is a part of our program on understanding of properties and applications of negative-index metamaterials in the electromagnetic setting  from  mathematical perspectives \cite{Ng-Superlensing-Maxwell, Ng-Negative-Cloaking-M, Ng-CALR-M, NgSil}. Doubly complementary media, introduced in  \cite{Ng-CALR-M},  are media consisting   of negative-index metamaterials in a shell and positive-index materials in its complement for which  the shell is complementary to a part of the core and a part of the exterior of the core-shell structure (\Cref{def-DCM}).  In this paper, we show that an arbitrary  object with  small cross-section placed close to a plasmonic structure of a doubly complementary medium is cloaked (\Cref{thm-cloaking}). 
 This cloaking property is  know as cloaking an object via anomalous localized resonance. 
We also address  the necessity for having doubly complementary properties in various schemes of cloaking and superlensing using complementary media (\Cref{pro-lensing,pro-cloaking}). In particular, the possibility that a cloak can act like a lens and conversely is confirmed. This possibility has been debated from  Pendry's  celebrate work on superlensing \cite{PendryNegative} (see also \cite{MN06}).

One of the consequences of our result on cloaking property established in \Cref{sect-CALR} (\Cref{thm-cloaking}) can be described as follows (see also \Cref{rem-5.2}). 
Denote $B_R(x)$ as the  open ball in $\mR^d$ ($d \ge 2$) centered at $x \in \mR^d$ and of radius $R>0$; when $x = 0$, we simply denote $B_{R}$.   
Let $d=3$, $0 < r_1 < r_2$, and set 
 $m = r_2^{2} / r_1^2$. Set, for $\delta > 0$,  
\begin{equation}\label{def-epsmu-special}
\dsp (\eps_\delta, \mu_\delta) = \left\{ \begin{array}{cl} \dsp  \big(  -  \frac{r_2^2}{ r^{2}} I + i \delta I,  -  \frac{r_2^2}{ r^{2}} I + i \delta I\big) & \mbox{ in } B_{r_2} \setminus B_{r_1}, \\[6pt]
\big(m I, m I \big) & \mbox{ in } B_{r_1}, \\[6pt]
(I, I) & \mbox{ otherwise}. 
\end{array} \right. 
\end{equation}
Denote  $\Gamma_2 = \Big\{ x \in \mR^3; |x| = r_2 \mbox{ and } x_3 =0 \Big\}$ and $\Gamma_1 = \Big\{ x \in \mR^3; |x| = r_1 \mbox{ and } x_3 =0 \Big\}$, and   set $
O_{j, \gamma}: = \Big\{x \in \mR^3; \dist(x, \Gamma_j) < \gamma \Big\}$ for $j =1, 2$,  and $\gamma >  0$. Let $(\eps_c, \mu_c)$ be a pair of piecewise  $C^1$, real, symmetric, uniformly elliptic, matrix-valued functions defined in $\cO_\gamma: = (O_{1, \gamma} \cup O_{2, \gamma}) \setminus  (B_{r_2} \setminus B_{r_1})$. Define
\begin{equation}\label{def-em-c}
(\eps_{c, \delta}, \mu_{c, \delta}) = \left\{ \begin{array}{cl} (\eps_c, \mu_c) & \mbox{ in } \cO_\gamma, \\[6pt]
(\eps_\delta, \mu_\delta) & \mbox{ otherwise}. 
\end{array}\right. 
\end{equation} 
Set $r_3 = r_2^2/ r_1$ and let $\omega > 0$. There exists $\gamma_0  = \gamma_0 (r_2, r_3)$ depending only on $r_2$, and $r_3$ ($\gamma_0$ is independent of $(\eps_c, \mu_c)$) such that for $0< \gamma < \gamma_0$, and for $J \in [L^2(\mR^3)]^3$ with compact support in $\mR^3 \setminus B_{r_3}$, we have
\begin{equation}\label{thm1-special-state}
\lim_{\delta \to 0} \| (E_{c, \delta}, H_{c, \delta}) - (\tE, \tH)\|_{L^2(B_R \setminus B_{r_3})} = 0. 
\end{equation}
Here  $(E_{c, \delta}, H_{c, \delta})$,  $(\tE, \tH)$ are respectively the unique radiating solution of the Maxwell equations 
$$
\left\{\begin{array}{cl}
\nabla \times E_{c, \delta} = i \omega \mu_{c, \delta} H_{c, \delta}&  \mbox{ in } \mR^3, \\[6pt]
\nabla \times H_{c, \delta} = - i \omega \eps_{c, \delta} E_{c, \delta} + J &  \mbox{ in } \mR^3 
\end{array} \right.
\mbox{ and }
\left\{\begin{array}{cl}
\nabla \times \tE = i \omega \tH&  \mbox{ in } \mR^3, \\[6pt]
\nabla \times \tH = - i \omega \tE + J &  \mbox{ in } \mR^3.  
\end{array} \right.  
$$
Physically, $\eps_\delta$ and $\mu_\delta$ describe the permittivity and the permeability of the considered medium, $B_{r_2} \setminus B_{r_1}$ is   a (shell) plasmonic structure  in which the permittivity and the permeability are negative, and  $i \delta I$ describes its loss, $\omega$ is the frequency, and $J$ is a density of charge.  As a consequence of  \eqref{thm1-special-state}, $\lim_{\delta \to 0} (E_{c, \delta}, H_{c, \delta}) = (\tE, \tH)$ in $\mR^3 \setminus B_{r_3}$ for all $J$ with compact support outside $B_{r_3}$.  One therefore cannot detect the difference between $(\eps_{c, \delta}, \mu_{c, \delta})$ and the homogeneous medium $(I, I)$, where $I$ denotes the $3 \times 3$ identity matrix,  as $\delta \to 0$ by observing  $(E_{c, \delta}, H_{c,\delta})$ outside $B_{r_3}$ using the excitation $J$: cloaking is achieved for observers outside $B_{r_3}$ in the limit as $\delta \to 0$.

Cloaking property of a plasmonic structure for small objects/sources nearby in some  superlensing settings satisfying doubly complementary property was raised in the literature about fifteen years ago \cite{MN06}. The possibility  that a lens consisting of negative-index materials can act  like a cloak and conversely was also debated in the literature, see e.g. \cite{MN06, Lai-Cloak}. The mathematical study for these problems was given in   \cite{Ng-CALR-O} for the acoustic setting for a subclass of complementary media. This class contains some but not  all plasmonic structures which are complementary with homogeneous medium  in three dimensions.  
This left widely open the question of whether cloaking property of plasmonic structures holds for the whole class of doubly complementary media  in the electromagnetic setting. This work answers this question completely. In fact, we establish a stronger statement saying that not only small objects but also  objects with small cross section close to the plasmonic structure are cloaked.

The cloaking method/property considered in this paper is related to but different from the so called cloaking using complementary media \cite{Ng-Negative-Cloaking-M} and is inspired by cloaking a source via anomalous localized resonance \cite{Ng-CALR-M} with its roots in \cite{MN06, Ng-CALR-O} (see also \cite{NMM94}). Mathematical works on applications and properties of negative-index metamaterials in the acoustic setting  such as superlensing, 
cloaking using complementary media,  cloaking via anomalous localized resonance for a source or for an object, and stability aspects of negative-index materials  can be found in  \cite{Ng-Superlensing}, \cite{Ng-Negative-Cloaking, MinhLoc2},  \cite{A-M13,  KLSW14, Ng-CALR, Ng-CALR-F}, \cite{Ng-CALR-O}, \cite{CS85, BCC12, Ng-WP},  respectively, and the references therein. 

Our analysis is in the spirit of \cite{Ng-CALR-O} but requires essentially  new ideas and techniques. To deal with the loss of ellipticity and compactness, and the occurence of  localized resonance, we first derive Cauchy's problems associated with two Maxwell systems from reflections originally proposed in \cite{Ng-Complementary} for the acoustic setting. We then  
apply the removing  localized singularity technique introduced in \cite{Ng-Negative-Cloaking, Ng-Superlensing}.  To be able to apply these techniques, the crucial and  difficult  point  is to  establish three-sphere inequalities with partial data for  Maxwell  equations (Theorem~\ref{thm-3SP-M}). To this end, we first prove new three-sphere inequalities with partial data for general elliptic systems (Theorems~\ref{thm-3SP}). We then derive the corresponding ones for the Maxwell  equations using their weakly coupled,  second-order elliptic property. These inequalities are the core part of our analysis.  They  are interesting in themselves, and can be used in other contexts, e.g. control theory  \cite{LR95,Coron07} or inverse problems \cite{KSU07,Y09}.

\medskip
 \noindent{\bf Outline of the paper:} The rest of the paper is organized as follows.  In \Cref{sect-3SP-statement}, we state  three-sphere inequalities for second-order elliptic systems and Maxwell equations with partial data. \Cref{sect-3SP,sect-3SP-M} are devoted to the proof of these inequalities for elliptic systems and Maxwell equations, respectively.  In \Cref{sect-CALR}, we state and give the proof of the main cloaking results for doubly complementary media.   In \Cref{sect-discussion}, we make several comments on the construction of  cloaking and superlensing devices using complementary media used in the literature and give the analysis for various contexts where a lens can act like a cloak and conversely.

\section{Three-sphere inequalities with partial data}\label{sect-3SP-statement}

Let $v$ be an holomorphic function defined in  $B_{R_3}$, Hadamard \cite{Hadamard} proved 
 the following famous three-sphere (circle) inequality:
 \beq
	\|v\|_{L^{\infty}(\pB_{R_2})} \leq \|v\|_{L^{\infty}(\pB_{R_1})}^{\alpha} \|v\|_{L^{\infty}(\pB_{r_3})}^{1 - \alpha}
\eeq{Hadamard}  
for all $0 < R_1 < R_2 < R_3$,  where 
\begin{equation*} 
\dsp 	\alpha = \ln \left( \frac{R_3}{R_2} \right) \Big/ \ln \left(\frac{R_3}{R_1} \right).
\end{equation*}
A three-sphere inequality for general second-order elliptic equations was established by Landis \cite{Landis} using Carleman type estimates with its roots in \cite{Carleman}. 
His result \cite[Theorem 2.2 on page 44]{Landis} can be stated as follows:  if $v$ is a solution to 
\begin{equation}\label{Equation}
\dive (\M \nabla v) +  c \cdot \nabla v + b v  = 0 \mbox{ in } B_{R^*},  
\end{equation}
where $\M$ is elliptic, symmetric, matrix-valued defined in $B_{R^*}$ of class $C^2$,  $c \in [C^1(\bar B_{R^*})]^d,  \, b \in C^1(\bar B_{R^*})$, {\it and  $b \le 0$}, then 
there is a constant $C>0$ such that, for $0 < R_1  < R_2 < R_3 < R^*$, 
\beq
	\|v\|_{L^{\infty}(\pB_{R_2})} \leq C \|v\|_{L^{\infty}(\pB_{R_1})}^{\alpha} \|v\|_{L^{\infty}(\pB_{r_3})}^{1 - \alpha}
\eeq{Landis}  
for some $\alpha  \in (0, 1)$ depending only on $R_2/R_1, R_2/R_3$, the ellipticity constant of $\M$,  and the regularity constants of $\M$, $b$, and $c$, and $R^*$. The assumption $b \le 0$ is  necessary to avoid the scenario  in which $v = 0$ on $\partial B_{R_1}$ or on $\partial B_{R_3}$ and $v \neq 0$ on $\partial B_{R_2}$, see, e.g. \cite{MinhLoc2} for comments on this point. Another proof of this inequality was obtained by Agmon \cite{Agmon} in which he used the logarithmic convexity. Garofalo and Lin \cite{GL86, GL87} established similar results for singular coefficients where the $L^\infty$-norm is replaced by the $L^2$-norm, and $\M$ is of class $C^1$, $c$ and $b$ are in $L^\infty$: 
\beq
	\|v\|_{L^{2}(\pB_{R_2})} \leq C \|v\|_{L^{2}(\pB_{R_1})}^{\alpha} \|v\|_{L^2(\pB_{r_3})}^{1 - \alpha}
\eeq{Lin}  
using  the  Almgren type frequency function approach.  For a general $b$,  a variant of \eqref{Lin} using balls holds (see e.g.  \cite[Theorem 1.10]{AR09}). The proof given in  \cite{AR09} is first based on a variant of \eqref{Lin} with some limitation on the radii and then involves arguments of  propagation of smallness.  In  \cite[Theorem 2]{MinhLoc2}, it was shown that \eqref{Lin} holds with the $\| v\|_{L^2(\partial B_r)}$-norm replaced by 
\begin{equation}\label{def-Hnorm}
\| v\|_{\bH(\partial B_r)} = \|v \|_{H^{1/2}(\partial B_r)}  + \| \M \nabla v \cdot e_r\|_{H^{-1/2}(\partial B_r)}
\end{equation}
for $r  = R_1, R_2$, or $R_3$, where $\alpha$  can be chosen independently of  $b$ and $c$. Here  $H^{-1/2}(\partial B_r)$ denotes the dual space of $H^{1/2}_0(\partial B_r) \big( =  H^{1/2}(\partial B_r)\big)$ and is equipped with the corresponding norm. The same notations are also used later for an open subset of a boundary of a smooth open subset of $\mR^d$ ($d \ge 2$).   

A closely related topic is the unique continuation principle. Some  seminar contributions in this context include the work of  Aronszajn \cite{Ar57}, Protter \cite{Protter60}, H\"ormander \cite{HorIII},  Kenig,  Ruiz, and Sogge \cite{KRS87}, Jerison and Kenig \cite{JK85}, and  Koch and Tataru \cite{KT01}.   Interesting surveys on these aspects can be found in \cite{RL12, AR09}. 

\medskip 
In this section, we  are concerned about  three-sphere inequalities for second-order,  elliptic systems and Maxwell equations with partial data.  
These inequalities have their own interests beside their applications in cloaking studied in this paper.   For $d \ge 2$, denote 
$$
\mR^d_+ = \Big\{x \in \mR^d; x_1  > 0 \Big\}  \quad \mbox{ and } \quad  
\mR^d_0 = \Big\{x \in \mR^d; x_1  = 0 \Big\}. 
$$
Set  $Q = (-1, 1)^d$   and $Q_+ = Q \cap \mR^d_+$ and  $Q_0 = Q \cap \mR^d_0$.  We first introduce 

\begin{definition}\label{def-geometry} Let $\Omega$ be a bounded, open subset $\Omega \subset \mR^d$ of class $C^1$. 
A compact subset $\Gamma$ of  $\partial \Omega$ is called a (d-2)-compact,  smooth submanifold of $\partial \Omega$ if for every $x \in \Gamma$, there exists a diffeomorphism $\bF: Q \to U$ for some open neighborhood $U$ of $x$   such that 
\begin{equation*}
\bF(Q_+) = U \cap \Omega, \quad \bF(Q_0) =  U  \cap \partial \Omega, \quad \bF(Q_0 \cap \{x_2  = 0 \}) = \Gamma \cap U. 
\end{equation*}
When $d=3$, a 1-compact, smooth submanifold of $\partial \Omega$ is simply called a compact,  smooth curve of $\partial \Omega$. 
\end{definition}

Our main result on three-sphere inequalities for second-order elliptic systems with partial data is 

\begin{theorem} \label{thm-3SP}
Let $d \ge 2, \, m \ge 1$, $\Lambda \ge 1$,  $0  < R_1 < R_3 $,   and let $\Gamma$ be a  $(d-2)$-compact, smooth submanifold of $\partial B_{R_1}$.
Denote  $O_r = \Big\{x \in \mR^d; \dist(x, \Gamma) < r \Big\}$, $D_r = B_{R_3} \setminus (\overline{B_{R_1} \cup O_r})$, and  $\Sigma_{r} = \partial B_{R_1} \setminus \bar O_{r}$ for $r>0$. For every $\alpha \in (0,  1)$, there exists $r_2 \in (0, R_3 - R_1)$ depending only on $\alpha$, $\Lambda$, $\Gamma$, 
$R_1$, and $R_3$   such that for every $r_1 \in (0, r_2)$, there exists $r_0 \in (0, r_1)$ depending only on $r_1$,  $\alpha$, $\Lambda$, $\Gamma$, 
$R_1$, and $R_3$ such that  for  $(d \times d)$ Lipschitz, uniformly elliptic, symmetric, matrix-valued function $\M^\ell$ defined in $D_{r_0}$ for $1 \le \ell \le m$, verifying, in $D_{r_0}$,  
\begin{equation}\label{thm-3SP-cdM}
\Lambda^{-1} |\xi|^2 \le \langle \M^{\ell} (x) \xi, \xi \rangle \le \Lambda |\xi|^2  \; \;  \forall \,  \xi \in \mR^d  \quad \mbox{ and } \quad   |\nabla \M^{\ell}(x) | \le \Lambda, 
\end{equation}
for $g \in L^2(D_{r_0})$, and for  $V  \in [H^1(D_{r_0})]^m$ satisfying, for  $1 \le \ell \le m$,  
\begin{equation}\label{fund-thm-Ineq} 
|\dive (\M^\ell \nabla V_\ell)| \le  \Lambda_1 \big( |\nabla V| + |V| + |g|\big)  \mbox{ in } D_{r_0} \mbox{ for some } \Lambda_1 \ge 0, 
\end{equation}
we have 
\begin{equation}\label{thm-3SP-cl}  
\| V \|_{H^1(B_{R_1 + r_2} \setminus B_{R_1 + r_1})} \le C  \Big(\| V\|_{\bH(\Sigma_{r_0})} + \| g\|_{L^2(D_{r_0})} \Big)^\alpha \Big( \|V\|_{H^1(D_{r_0})} + \| g\|_{L^2(D_{r_0})} \Big)^{1 - \alpha}, 
\end{equation}
for some positive constant $C$ depending only on $\alpha$, $\Lambda$,  $\Lambda_1$, $\Gamma$, $R_1$, $R_3$, $m$, and $d$. 
\end{theorem}

The geometry of \Cref{thm-3SP} is given in \Cref{fig-Thm}.

\begin{figure}
\centering
\begin{tikzpicture}[scale=1.8]

\filldraw[gray!20!] (0,0) circle (2);

\filldraw[gray!80!] (0,0) circle (1.25);

\filldraw[gray!20!] (0,0) circle (0.95);

\draw[blue, line width=1mm] (0,0) circle (0.6);

\filldraw[white] (0,0) circle (0.6);

\filldraw[white] (0.6,0) circle (0.2);

\draw[blue] (0,0.6) node[below]{$\Sigma_{r_0}$};

\draw (0,0) node[right]{$0$};

\draw [->] (1.32,0) -- (0.72,0);  

\draw[] (1.35,0) node[right]{$O_{r_0}$};

\draw[dashed] (0.6,0) circle (0.2);

\draw (0.68,0.05) node[left]{$\Gamma$};

\fill (0.6,0) circle(0.5pt); 

\fill[black] (0,0) circle(0.5pt); 

\draw (0,0) -- ({0.6*cos(-60)}, {0.6*sin(-60)});

\draw ({0.35*cos(-60)}, {0.35*sin(-60)}) node[left]{$R_1$};

\draw ({0.6*cos(-120)}, {0.6*sin(-120)}) -- ({0.95*cos(-120)}, {0.95*sin(-120)});

\draw ({0.7*cos(-120)}, {0.7*sin(-120)}) node[left]{$r_1$};

\draw ({0.6*cos(-180)}, {0.6*sin(-180)}) -- ({1.25*cos(-180)}, {1.25*sin(-180)});

\draw ({0.95*cos(-180)}, {0.95*sin(-180)}) node[below]{$r_2$};

\draw (0,0) -- ({2*cos(-220)}, {2*sin(-220)});

\draw ({1.6*cos(-220)}, {1.6*sin(-220)}) node[below]{$R_3$};

\draw ({1.6*cos(-90)}, {1.6*sin(-90)}) node[]{$D_{r_0}$};

\draw[->] ({2.5*cos(45)}, {2.5*sin(45)}) -- ({1.1*cos(45)}, {1.1*sin(45)});

\draw[] ({2.5*cos(45)}, {2.5*sin(45)}) node[above]{$B_{R_1 + r_2} \setminus B_{R_1 + r_1}$};

\end{tikzpicture}
\caption{Geometry of \Cref{thm-3SP} in two dimensions}
\label{fig-Thm}
\end{figure}
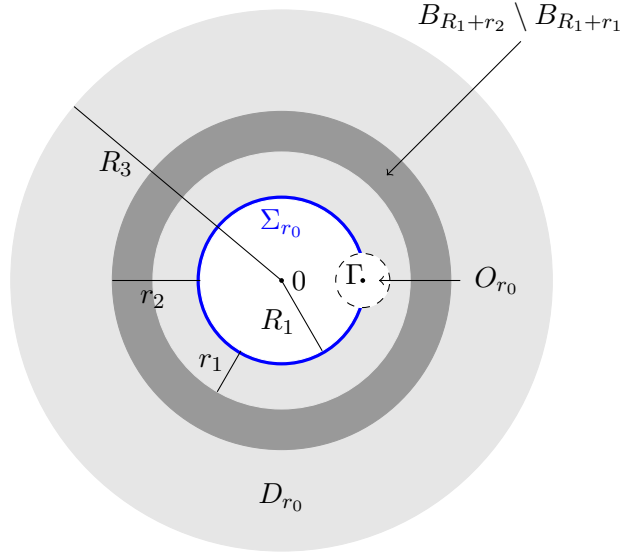

\medskip 
Some comments on \Cref{thm-3SP} are in order.  In  \eqref{thm-3SP-cl}, one only uses the information of $\Sigma_{r_0}$ (a portion  of $\partial B_{R_1}$, see \Cref{fig-Thm}) in the first interpolation term (see \Cref{fig-Thm}). The terminology {\it partial data} comes from this.  The constant $r_1$, $r_2$, and $r_0$ are independent of $\Lambda_1$ but the constant $C$ does.  If instead of $\Sigma_{r_0}$, one uses $\partial  B_{R_1}$, inequality  \eqref{thm-3SP-cl} 
is then known.  In this case, one can even use $\| V\|_{\bH(\partial B_{R_3})}$ instead of $\| V \|_{H^1(D_{r_0})}$ in \eqref{thm-3SP-cl} by applying  \cite[Theorem 2]{MinhLoc2}. 
Using known three-sphere inequalities and the arguments of propagation of smallness, one can prove \eqref{thm-3SP} for some $\alpha \in (0, 1)$. Nevertheless, the non-triviality and the novelty of \Cref{thm-3SP} lies on the fact that, for a given $\alpha \in (0, 1)$  arbitrary, \eqref{thm-3SP-cl} holds for some $r_0, r_1, r_2$. Even if $v$ is a solution of the Laplace equation in two dimensions, using Hadamard  there-sphere (circle) inequalities  and  the arguments of propagation of smallness, as far as we know, one can only obtain  \eqref{thm-3SP-cl} for some small $\alpha$ even though one replaces $\Sigma_{r_0}$ by $\partial B_{R_1} \setminus \{x_0 \}$ for some $x_0 \in \partial B_{R_1}$.  The requirement $\alpha > 1/2$ is necessary for the cloaking application presented later in this paper. Moreover, the rate of the convergence in our cloaking application is  better and better if $\alpha$ is closer and closer to 1. 
Note that, for $d \ge 3$,  the set $B_{R_1} \cap O_r$ in \Cref{thm-3SP} is not small in term of radius but has a small cross-section.  To our knowledge, \Cref{thm-3SP} is new even for the Laplace equation and $d \ge 3$, i.e., $m=1$, $\M^{1}  = I$, $\Lambda_1 = 0$, and $d \ge 3$. For $d=2$, the setting for the Laplace equation is previously established in \cite{Ng-CALR-O}, which is indeed one of the main motivations of our present work (see \Cref{sect-heu}). \Cref{thm-3SP} is  surprising when compared with the fact that the Cauchy problems for second-order elliptic equations are  unstable.  It is worth noting that obtaining sharp interpolation constants for three-sphere inequalities has its own interest and plays an important role in 
several applications. For example, this is the main tool in the analysis of Landis' conjecture, see \cite{Meshkov91,BK05, KSW15} and the references therein, where the real potential case is still widely open.

As a consequence of  Theorem~\ref{thm-3SP}, one can derive  three-sphere inequalities with partial data for $R_1 < R_2 < R_3$ from classical three-sphere inequalities.  Here is an illustration in the spirit of Hadamard. Let $d = 2, 3$, $\omega > 0$, and $R_1< R_2 < R_3$, and let $v \in H^1(B_{R_3} \setminus B_{R_1})$ be a solution of  $\Delta v + \omega^2 v = 0$. We have, see e.g. \cite[Lemma 4.2]{Ng-CALR-F}, with $\alpha_0 = \ln(R_3/ R_2) \Big/ \ln (R_3/ R_1)$, 
\begin{equation}\label{3SP-H}
\| v\|_{{\bf H}(\partial B_{R_2})} \le C \| v\|_{{\bf H}(\partial B_{R_1})}^{\alpha_0} \| v\|_{{\bf H}(\partial B_{R_2})}^{1 - \alpha_0},  
\end{equation}
for some positive constant $C$ depending only on $\omega$, $R_1$, $R_2$, and $R_3$.  Using \Cref{thm-3SP}, we establish the following new variant of \eqref{3SP-H}:

\begin{corollary}\label{cor-3SP} Let $d = 2, 3$, and $0< R_1< R_2 < R_3$, and let 
$\Gamma$ be a  $(d-2)$-compact, smooth submanifold of $\partial B_{R_1}$.  Denote  $O_r = \Big\{x \in \mR^d; \dist(x, \Gamma) < r \Big\}$, $D_r = B_{R_3} \setminus (\overline{B_{R_1} \cup O_r})$, and $\Sigma_{r} = \partial B_{R_1} \setminus O_{r}$ for $r>0$.
Set $\alpha_0 = \ln(R_3/ R_2) \Big/ \ln (R_3/ R_1)$. Then, for any $ \alpha \in (0,  \alpha_0)$, there exists $r_0  \in (0, R_2 - R_1)$, depending only on $R_1$, $R_2$,  $R_3$, $\Gamma$, and $\alpha$ such that,  for $\omega > 0$ and  for $v \in H^1(D_{r_0})$ satisfying  $\Delta v + \omega^2 v = 0$ in $D_{r_0}$, we have
\begin{equation*}  
\| V \|_{\bH(\partial B_{R_2})} \le C \| V\|_{\bH(\Sigma_{r_0})}^\alpha \|V\|_{H^1(D_{r_0})}^{1 - \alpha}, 
\end{equation*}
for some positive constant $C$ depending only on  $\alpha$, $\omega$, $\Gamma$, $R_1$, and $R_3$. 
\end{corollary}

The proofs of  \Cref{thm-3SP} and \Cref{cor-3SP},  and variants of  \Cref{thm-3SP}  (\Cref{fund-thm} and \Cref{cor-3SP-*}) are given in \Cref{sect-3SP}. 

\medskip 
We next discuss the Maxwell equations. Our main result in this direction  is 

\begin{theorem} \label{thm-3SP-M}
Let $d = 3$, $\Lambda \ge 1$,  $0  < R_1 < R_3 $,   and let $\Gamma$ be a compact, smooth curve of  $\partial B_{R_1}$. 
For $r>0$, denote  $O_r = \Big\{x \in \mR^3; \dist(x, \Gamma) < r \Big\}$, $D_r = B_{R_3} \setminus (\overline{B_{R_1} \cup O_r})$, and $\Sigma_{r} = \partial B_{R_1} \setminus O_{r}$. For every $\alpha \in (0,  1)$, there exists $r_2 \in (0, R_3 - R_1)$ depending only on $\alpha$, $\Lambda$, $\Gamma$, 
$R_1$, and $R_3$  such that for every $r_1 \in (0, r_2)$, there exists $r_0 \in (0, r_1)$ 
depending only on $r_1$,  $\alpha$, $\Lambda$, $\Gamma$, $R_1$, and $R_3$ 
such that  for $(\eps, \mu)$ a pair of   $(3 \times 3)$ real, uniformly elliptic, symmetric, matrix-valued functions defined in $D_{r_0}$ of class $C^2$ verifying, in $D_{r_0}$,  
\begin{equation}
\Lambda^{-1} |\xi|^2 \le \langle \M^{\ell} (x) \xi, \xi \rangle \le \Lambda |\xi|^2  \; \;  \forall \,  \xi \in \mR^d  \quad \mbox{ and } \quad   |\nabla \M^{\ell}(x) | \le \Lambda, 
\end{equation}
for $\omega > 0$, for $J_e, \, J_m \in [L^2(D_{r_0})]^3$,  and  for
  $(E, H) \in [H(\curl, D_{r_0})]^2$ satisfying 
\begin{equation*} 
\left\{\begin{array}{cl}
\nabla \times E = i \omega \mu H + J_e& \mbox{ in } D_{r_0}, \\[6pt]
\nabla \times H = - i \omega \eps H + J_m & \mbox{ in } D_{r_0}, 
\end{array}\right. 
\end{equation*}
 we have
\begin{multline*}  
  \| (E,   H) \|_{L^2(B_{R_1 + r_2} \setminus B_{R_1 + r_1})}
  \le   C \Big(   \| (E \times \nu, H \times \nu)\|_{H^{-1/2}(\dive_\Gamma, \Sigma_{r_0})}  + \| (J_e, J_m) \|_{L^2(D_{r_0})} \Big)^\alpha  \times \\[6pt] \times \Big(  \|(E, H)\|_{L^2(D_{r_0})} + \| (J_e, J_m) \|_{L^2(D_{r_0})} \Big)^{1 - \alpha} , 
\end{multline*}
for some positive constant $C$ depending only  on $r_1$, $\alpha$,  $\omega$, $\Lambda$, $\Gamma$,  $R_1$, $R_3$, and 
the upper bound of \\$\| (\eps, \mu) \|_{C^2(\bar D_{r_0})}$.
\end{theorem}

Here and in what follows, for an open,   bounded   subset $\Omega$  of $\mR^3$ of class $C^1$,   one  denotes,  with $\Gamma = \partial \Omega$, 
\begin{equation*}
H^{-1/2}(\dive_\Gamma, \Gamma): = \Big\{ \phi \in [H^{-1/2}(\Gamma)]^3; \; \phi \cdot \nu = 0 \mbox{ and } \dive_\Gamma \phi \in H^{-1/2}(\Gamma) \Big\},
\end{equation*}
\begin{equation*}
\| \phi\|_{H^{-1/2}(\dive_\Gamma, \Gamma)} : = \| \phi\|_{H^{-1/2}(\Gamma)} +  \| \dive_\Gamma \phi\|_{H^{-1/2}(\Gamma)}.
\end{equation*}
For  an open subset $\Omega$ of $\mR^3$,  the following standard notations  are used: 
\begin{equation*}
H(\curl, \Omega) : = \Big\{ u \in [L^2(\Omega)]^3; \; \nabla \times u  \in [L^2(\Omega)]^3 \Big\}, 
\end{equation*}
\begin{equation*}
\| u\|_{H(\curl, \Omega)} : = \| u\|_{L^2(\Omega)} + \| \nabla \times u \|_{L^2(\Omega)}, 
\end{equation*}
\begin{equation*}
H_{\loc}(\curl, \Omega) : = \Big\{ u \in [L_{\loc}^2(\Omega)]^3; \; \nabla \times u  \in [L^2_{\loc}(\Omega)]^3 \Big\}. 
\end{equation*}

\begin{remark} \rm In \Cref{thm-3SP-M}, one requires $(\eps, \mu)$ to be of class $C^2$. Nevertheless, the constant $r_2$ depends on $\Lambda$ not on $\|(\eps, \mu)\|_{C^2(D_{r_0})}$.
\end{remark}

\Cref{thm-3SP-M} and its variant (\Cref{fund-thm-M}) are the {\it new crucial} ingredients in the proof of cloaking property for the doubly complementary media.  A consequence of \Cref{thm-3SP-M} in the spirit of Hadamard is given in \Cref{cor-Maxwell} in \Cref{sect-3SP-M}. 

\medskip 
The proofs of \Cref{thm-3SP} and \Cref{cor-3SP}  are given in \Cref{sect-3SP}. 
The proofs of \Cref{thm-3SP-M} and its consequence (\Cref{cor-Maxwell})  are given in  \Cref{sect-3SP-M}. The most important  ingredient of these proofs is \Cref{fund-thm} in Section~\ref{sect-3SP}. 
Concerning the proof of \Cref{fund-thm}, we  first use an appropriate change of variables inspired by conformal maps. We then establish a new type of three-sphere inequalities for a class of ``degenerate" second-order,  elliptic inequalities in which not only  the properties of the coefficients but also the way they interact with the domain considered  play an important role (see also the paragraph right after \Cref{fund-thm} and the first paragraph of \Cref{sect-UL}). 
 
\section{Three-sphere inequalities for second-order elliptic inequalities}\label{sect-3SP}

This section is on three-sphere inequalities for second-order elliptic inequalities with partial data. The key ingredient is the variant  of 
\Cref{thm-3SP} in a half plane. Throughout this section,  for $d \ge 2$ and $x = (x_1, x_2, \tx) \in \mR \times \mR \times \mR^{d-2}$, we use the polar coordinate $(\rr, \theta)$ with $\theta \in (-\pi, \pi]$ for the pair $(x_1, x_2)$;  the variable $\tx$ is irrelevant for $d = 2$. 
For $0 < \gamma_1 < \gamma_2 < 1$ and for $R>0$, we denote 
\begin{equation}\label{def-Y-M}
Y_{\gamma_1, \gamma_2,  R} 
= \Big\{ x  \in \mR^d; \; \theta \in (-\pi/2, \pi/2), \;   \gamma_1 R < \rr < \gamma_2 R,  \mbox{ and } |\tx| < R \Big\}
\end{equation}
(see also \Cref{fig}).

\medskip 
We have

\begin{theorem} \label{fund-thm}  Let $d \ge 2$, $m \ge 1$, $\Lambda \ge 1$,  and $R_* < R < R^*$.  Then, for any  $\alpha \in (0, 1)$, there exists a  constant $\hgamma_2 \in (0, 1)$, depending only on $\alpha$, $\Lambda$, $R_*$, $R^*$,   $m$, and $d$ such that
for every $\hgamma_1 \in (0,  \hgamma_2)$, there exists  $\hgamma_0 \in (0, \hgamma_1)$ depending only on $\alpha$, $\hgamma_1$,  $\Lambda$, $R_*$, $R^*$, $m$, and $d$ such that,   for real, symmetric, uniformly elliptic, Lipschitz matrix-valued functions  $\M^\ell$ with $1 \le \ell \le m$  defined in $D_{\hgamma_0}: = Y_{\hgamma_0, 1,  R}$  verifying, in $D_{\hgamma_0}$, 
\begin{equation}\label{fund-thm-pro-M}
\Lambda^{-1} |\xi|^2 \le \langle \M^{\ell} (x) \xi, \xi \rangle \le \Lambda |\xi|^2  \; \;  \forall \,  \xi \in \mR^d  \quad \mbox{ and } \quad   |\nabla \M^{\ell}(x) | \le \Lambda, 
\end{equation}
for $g \in L^2(D_{\hgamma_0})$, and for  $V  \in [H^1(D_{\hgamma_0})]^m$ satisfying, for  $1 \le \ell \le m$,  
\begin{equation} 
|\dive (\M^\ell \nabla V_\ell)| \le  \Lambda_1 \big( |\nabla V| + |V| + |g| \big)  \mbox{ in } D_{\hgamma_0} \mbox{ for some } \Lambda_1 \ge 0, 
\end{equation}
we have, with $\Sigma_{\hgamma_0} = \partial D_{\hgamma_0} \cap \big\{ x_1= 0 \big\}$,  
\begin{equation}\label{fund-thm-S}
\| V \|_{H^1(Y_{\hgamma_1, \hgamma_2, \frac{R}{4}})} 
 \le C \Big( \|V\|_{{\bf H}(\Sigma_{\hgamma_0})} + \| g\|_{L^2(D_{\hgamma_0})} \Big)^{\alpha}  
\Big(\| V\|_{H^1(D_{\hgamma_0})}  + \| g\|_{L^2(D_{\hgamma_0})}  \Big)^{1-\alpha},  
\end{equation}
for some positive constant $C$ depending only on $\alpha, \, \hgamma_1,  \,  \Lambda, \, \Lambda_1, \, R_*, \, R^*, \,  m$, and  $d$. 
\end{theorem}

\begin{figure}
\centering
\begin{tikzpicture}[scale=1.6]


\fill[black!10] (-90:2.4) arc (-90:90:2.4)-- (90:0.5) arc (90:-90:0.5) -- cycle;

\fill[black!25] (-90:1.1) arc (-90:90:1.1)-- (90:1.6) arc (90:-90:1.6) -- cycle;

\draw (0,0) -- ({1.1*cos(150-90)},{1.1*sin(150-90)});
\draw ({0.6*cos(150-90)},{0.6*sin(150-90)}) node[right]{$\frac{\hgamma_1 R}{4}$};

\draw (0,0) -- ({1.6*cos(120-110)},{1.6*sin(120-110)});
\draw ({1.4*cos(120-110)},{1.4*sin(120-110)}) node[below]{{$\frac{\hgamma_2 R}{4}$}};

\draw[thick, blue] (0, -2.4) -- (0, -0.5);
\draw[thick, blue] (0, 0.5) -- (0, 2.4);

\draw[->] (-0.5, 2.15) --  (0, 2.15);
\draw[] (-0.5, 2.15) node[left]{$\Sigma_{\gamma_0}$};

\draw [white, domain =-90:90] plot ({0.5*cos(\x)}, {0.5*sin(\x)});

\draw [white, domain =-90:90] plot ({1.6*cos(\x)}, {1.6*sin(\x)});


\draw[] (0,0) -- ({0.5*cos(30-90)},{0.5*sin(30-90)});
\draw ({0.27*cos(30-90)},{0.27*sin(30-90)}) node[left]{{\small $\hgamma_0 R$}};

\draw ({1.55*cos(75)},{1.55*sin(75)}) node[below]{$Y_{\hgamma_1, \hgamma_2, \frac{R}{4}}$}; 

\draw ({2*cos(-75)},{2*sin(-75)}) node[]{$D_{\gamma_0}$};

\fill[black] (0,0)  circle(0.5pt);

\draw (-0.1,0)  node[left, above]{$0$};

\end{tikzpicture}
\caption{Geometry of $Y_{\gamma_1, \gamma_2, \frac{R}{4}}$, $\Sigma_{\gamma_0}$, and $D_{\gamma_0}$ in two dimensions.}
\label{fig}
\end{figure}
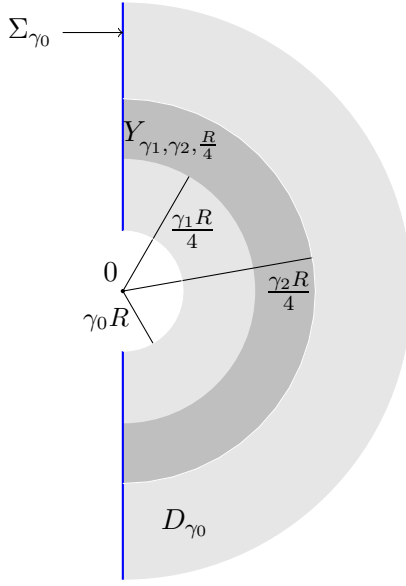

The geometry of $Y_{\gamma_1, \gamma_2, \frac{R}{4}}$, $\Sigma_{\gamma_0}$, and $D_{\gamma_0}$ are given in \Cref{fig} in two dimensions (see also \Cref{fig-Thm} for a comparison).

\begin{remark} \rm The constants $\gamma_0, \gamma_1$, and $\gamma_2$ are independent of $\Lambda_1$ but the constant $C$ does. 
\end{remark}

The proof of Theorem~\ref{fund-thm} is based on Carleman's estimates. The weight used is $e^{\beta r^{-p}}$ 
for which $\beta$ and $p$ are two (large) parameters as in \cite{MinhLoc2} (the work of Protter's \cite{Protter60} and of Fursikov and Imanuvilov's \cite{FI96} are also worth mentioning). A simple but critical  step of the proof is the  use of transformations,  related to the conformal type map,  $(x_1, x_2, \tx) \to \big(\rr^{1/n} \cos (\theta/n), \rr \sin(\theta/n), \tx \big)$  to transform the domain $Y_{\gamma, 1, R}$ into a domain for which the first two variables are in a 
sector of circulars with a small angle ($n$ is large then). We  then apply three-sphere inequalities for this domain  to deduce the desired estimate. The advantage of this process is that three-sphere inequalities with partial data are easier to handle for this new geometry, as noted in \cite{Ng-CALR-O} (see also \Cref{sect-heu}).  However, {\it new} difficulties appear. On one hand,   the lower bound of the ellipticity of the new set of matrix-valued functions obtained from $\M^\ell$ ($1 \le \ell \le m$) goes to 0 as $n \to + \infty$.
On the other hand, to be able to carry on three-sphere inequalities with partial data in this domain, one requires to establish three-sphere inequalities associated with the new set of matrices in which the output (the parameter $\alpha$) is independent of $n$ (see \Cref{sect-heu} for more details).  To overcome this obstacle,  a structure of the new set of  matrix-valued functions is formulated (see e.g. \Cref{rem-lem1} and \eqref{LF-claim2-C}) and new Carleman's estimates capturing this structure are derived.

\medskip 
As a direct consequence of \Cref{fund-thm}, we have the following result whose proof is omitted.

\begin{corollary} \label{cor-3SP-*}
Let $d \ge 2, \, m \ge 1$, $\Lambda \ge 1$,  let $\Omega$ be a bounded, open subset of $\mR^d$ of class $C^1$, and let $\Gamma$ be a $(d-2)$-compact, smooth submanifold of $\partial \Omega$ and belong to a connected component $\Sigma$ of $\partial \Omega$.
Denote  $O_r = \Big\{x \in \mR^d; \dist(x, \Gamma) < r \Big\}$, $D_r = \Omega \setminus \bar O_r$, and  $\Sigma_{r} = \Sigma \setminus \bar O_{r}$ for $r>0$. Then, for every $\alpha \in (0,  1)$, there exists 
$r_2>0$ depending only on $\alpha$, $\Lambda$, $\Gamma$, and  $\Omega$ such that for $r_1 \in (0, r_2)$, there exists $r_0 \in (0, r_1)$, depending only on $r_1$,  $\alpha$, $\Lambda$, $\Gamma$, and  $\Omega$ such that 
for  $(d \times d)$ Lipschitz, uniformly elliptic,  matrix-valued function $\M^\ell$ defined in $\Omega$ with $1 \le \ell \le m$, verifying, in $D_{r_0}$,   
\begin{equation*}
\Lambda^{-1} |\xi|^2 \le \langle \M^{\ell} (x) \xi, \xi \rangle \le \Lambda |\xi|^2  \; \;  \forall \,  \xi \in \mR^d  \quad \mbox{ and } \quad   |\nabla \M^{\ell}(x) | \le \Lambda, 
\end{equation*}
for $g \in L^2(D_{r_0})$, and for  $V  \in [H^1(D_{r_0})]^m$ satisfying, for  $1 \le \ell \le m$,  
\begin{equation}
|\dive (\M^\ell \nabla V_\ell)| \le  \Lambda_1 \big( |\nabla V| + |V| + |g| \big)  \mbox{ in } D_{r_0} \mbox{ for some } \Lambda_1 \ge 0, 
\end{equation}
we have 
\begin{equation*}  
\| V \|_{H^1(O_{r_2} \setminus O_{r_1})} \le 
C \Big( \| V\|_{\bH(\Sigma_{r_0})} + \| g\|_{L^2(D_{r_0})} \Big)^\alpha \Big( \|V\|_{H^1(D_{r_0})} + \| g\|_{L^2(D_{r_0})} \Big)^{1 - \alpha}, 
\end{equation*}
for some positive constant $C$ depending only on $r_1$, $\alpha$, $\Lambda$, $\Lambda_1$,  $\Gamma$, $\Omega$, $m$, and $d$. 
\end{corollary}


The rest of this section is organized as follows. In Section~\ref{sect-UL}, we establish several lemmas used in the proof of \Cref{fund-thm}. 
The main step of the proof is given in \Cref{sect-mainstep}. The complete proof of \Cref{fund-thm} is given in \Cref{sect-fund-thm}. 


\subsection{Preliminaries}\label{sect-UL}

In this section, we establish several lemmas used in the proof of Theorem~\ref{fund-thm}. The computations are as  in the spirit of Carleman's estimates,  nevertheless, the assumptions and conclusions are  importantly formulated  in a way that  can be used in the context of partial data. 
In what follows, $\Y$  denotes a bounded connected open subset of $\mR^d$ with Lipschitz boundary,  for $x \in \mR^d$ $(d \ge 2)$, $r$ denotes its (Euclidean) length, i.e., $r = |x|$, and $\langle \cdot, \cdot \rangle$ denotes the standard Euclidean scalar product unless otherwise stated. 
All quantities considered in this section are real. The key results of this section are \Cref{lem1} and its consequence \Cref{lem-main}.  Note that even the matrix-valued function $M$ considered here is assumed to be Lipschitz, uniformly elliptic,  and symmetric, the constants in Carleman's type-inequalities depend {\it  partially} on this information,  and depend on {\it the domain considered} as well (see e.g.  the assumptions in \Cref{lem1} and \Cref{rem-lem1}).   
The weight used in this section is of the form $e^{\beta r^{-p}}$ for large $\beta$ ($\beta$ can be positive or negative;  we consider here  Carleman's estimates in $\Y$,  which is not assumed to be a ball, with boundary terms) and for large positive $p$.  Other weights satisfying the strong pseudoconvexity condition with respect to the principal parts of the elliptic system might be also used. Nevertheless, we prefer to choose this type of weights since we can explicitly keep track the dependence of various constants on $\alpha$, $\beta$, and  $M$, and the domain considered $\Y$ in formulating and proving the new results. The computations/analysis in this section are quite standard as long as the results/statements are formulated (which are based on the need of the analysis in \Cref{sect-mainstep}). For the convenience of the reader, the details are mainly given   in the appendix so that the constants can be tracked.  

\medskip 

We begin with  

\begin{lemma} \label{lem-prepare1} Let $w \in H^2(\Y)$ and let $M$ be a Lipschitz, symmetric, uniformly elliptic, matrix-valued function defined in $\Y$.  
We have
\begin{equation*}
 \int_\Y   (x \cdot M \nabla w) \,   \Div (M\nabla w) \ge   -   \int_{\Y} \langle B \nabla w, \nabla w \rangle 
   -  \int_{\pY} C r  |M|^2 |\nabla w|^2, 
\end{equation*}
for some positive constant $C$ depending only on $d$. Here, for $ x \in \Y$,  
\begin{multline}\label{def-B}
\langle B (x) y, y \rangle  : =  \langle [(M (x) y) \cdot \nabla_x]  (M(x) x), y \rangle +  \frac{1}{2} \langle \dive_x(M(x) x) M(x) y, y \rangle \\[6pt]
+ \frac{1}{2} \langle [(M (x) x) \cdot \nabla_x ] M(x)  y , y \rangle \quad  \mbox{ for } y \in \mR^d. 
\end{multline}
\end{lemma}

\begin{remark} \rm The quantity $B$ given in \eqref{def-B} plays a role in our analysis. An upper bound for $|\langle B (x) y, y \rangle|$ instead of $\| M \|_{W^{1, \infty}}$ is sufficient for our analysis (see \Cref{lem1}). 
\end{remark}

\begin{proof} The proof is standard as follows. An integration by parts gives
\begin{equation}\label{1.4}
	 \int_\Y   (x \cdot M \nabla w) \,   \Div (M\nabla w) 
	 = - \int_{\Y} \nabla ( x \cdot M  \nabla w) \cdot M\nabla w \\+ \int_{\pY}  (x \cdot M\nabla w) \,  M \nabla w \cdot \nu. 
\end{equation}
Using the symmetry of $M$, we have \footnote{In what follows, the repeated summation is used. }
\begin{equation}\label{1.6}
\frac{\partial }{\partial x_i}(x \cdot M\nabla w) = \frac{\partial }{\partial x_i} \Big( M_{k j} x_j \frac{\partial w}{\partial x_k} \Big) = M_{k j} x_j \frac{\partial^2 w}{\partial x_i \partial x_k} +  \frac{\partial }{\partial x_i} (Mx) \cdot \nabla w
\end{equation}
and
\begin{multline}\label{1.6-1}
		-\int_{\Y} 2 x_j M_{kj} \frac{\partial^2 w}{\partial x_i \partial x_k}M_{il} \frac{\partial w}{\partial x_l}
		 = - \int_\Y x_j M_{kj} M_{il} \frac{\partial}{\partial x_k}\left(\frac{\partial w}{\partial x_i} \frac{\partial w}{\partial x_l}\right) \\
		= \int_\Y \frac{\partial (x_j M_{kj} M_{il})}{\partial x_k}\frac{\partial w}{\partial x_i} \frac{\partial w}{\partial x_l} - \int_{\pY}  x_j M_{kj} M_{il}\frac{\partial w}{\partial x_i} \frac{\partial w}{\partial x_l} \nu_k. 
\end{multline}
The conclusion now follows from  \eqref{1.4}, \eqref{1.6}, and \eqref{1.6-1}.
\end{proof}

The second lemma, whose proof is given in \Cref{ap-lem-prepare2-0},  is 

\begin{lemma}\label{lem-prepare2-0} Let  $p \ge 1$, $\beta \in \mR$, $w \in H^2(\Y)$, 
and let $M$ be a Lipschitz, symmetric, uniformly elliptic, matrix-valued function defined in $\Y$. Set, for $x \in \Y$, 
\begin{equation*}
T_1(x)= (2p + 4) (x \cdot Mx)^2  r^{-4}  -  r^{-2}   \Div \big[ (x \cdot Mx) Mx\big ]  - r^{-2}   |\dive (Mx)|^2 \langle M  x, x \rangle 
\end{equation*}
and 
\begin{equation*}
T_2 (x) = -  (p + 4) (x \cdot Mx)^2 r^{-4}    +   r^{-2}   \Div \big[ (x \cdot Mx) Mx\big ]. 
\end{equation*}
Assume that $|x| \le 1$ for $x \in \Y$. 
We have
\begin{multline}\label{lem-prepare-2-1}
 \int_\Y    e^{\beta r^{-p}} (M x \cdot \nabla |w|^2) \,  \Div (M\nabla  e^{-\beta r^{-p}}) \\[6pt]
   \ge  \int_\Y  \Big(p^2 \beta^2 r^{-2p -2}T_1 + \beta p (p+2) r^{- p - 2} T_2 \Big) |w|^2 - \int_\Y   \langle M \nabla w, \nabla w \rangle    - \int_{\pY} C \beta^2 p^2 r^{-2p - 1} |M|^2 |w|^2, 
\end{multline}
for some positive constant $C$ depending only on $d$. 
\end{lemma}

Using \Cref{lem-prepare2-0}, one can derive 


\begin{lemma}\label{lem-prepare2} Let $w \in H^2(\Y)$, 
and let $M$ be a Lipschitz, symmetric, uniformly elliptic, matrix-valued function defined in $\Y$. Assume that $|x| \le 1$ for $x \in \Y$,  and   for some $\Lambda \ge 1$, 
$$
\langle M x, x \rangle \ge \Lambda^{-1} |x|^2  \quad \mbox{ for } x \in \Y
$$
and 
$$
|M| +   |\dive(Mx)| + |x|^{-2} |\nabla (x \cdot M x)  \cdot Mx | \le \Lambda \quad  \mbox{ for } x \in \Y. 
$$
There exist two constants $p_{\Lambda}, \beta_{\Lambda} \ge 1$, depending only on $\Lambda$ and $d$, such that,  for $p \ge p_{\Lambda}$ and $|\beta| \ge \beta_{\Lambda}$, we have 
\begin{multline}\label{lem-prepare-2-2}
 \int_\Y    e^{\beta r^{-p}} (M x \cdot \nabla |w|^2) \,  \Div (M\nabla  e^{-\beta r^{-p}}) \\[6pt]
   \ge  \int_\Y   p^3 \beta^2 \Lambda^{-2} r^{-2p - 2}|w|^2 - \int_\Y   \langle M \nabla w, \nabla w\rangle   - \int_{\pY} C \beta^2 p^2 r^{-2p - 1} \Lambda^2 |w|^2
\end{multline}
for some positive constant $C$ depending only on $d$. 
\end{lemma}

\begin{remark} \rm Note that $\Lambda$ is determined via just for $x \in \Y$. It  contains only partially the information of the ellipticity of $M$ and the Lipschitz constant of $M$. 
\end{remark}

\begin{proof}  Estimate~\eqref{lem-prepare-2-2} is a direct consequence of \eqref{lem-prepare-2-1} for large $\beta$ and $p$. \end{proof}

Using Lemmas~\ref{lem-prepare1} and \ref{lem-prepare2}, we can establish the following result.

\begin{lemma}\label{lem1} Let    $v \in H^2(\Y)$, 
and let $M$ be a Lipschitz, symmetric, uniformly elliptic, matrix-valued function defined in $\Y$. Assume that,  for $x \in \Y$, the following conditions hold: $|x| \le 1$, 
\begin{equation}\label{lem1-as00}
\langle M x, x \rangle \ge \Lambda^{-1} |x|^2,
\end{equation}
\begin{equation}\label{lem1-as01}
|M| +  |\dive(Mx)| + |x|^{-2} |\nabla (x \cdot M x) \cdot Mx | \le \Lambda, 
\end{equation}
	 \begin{equation}\label{lem1-as1}
	 |\langle B y, y \rangle| \le \Lambda \langle M y, y \rangle \mbox{ for } y \in \mR^d, 
	 \end{equation}
for some $\Lambda \ge 1$ where $\langle B y, y \rangle$ is defined in \eqref{def-B}.  There exist two constants  $p_{\Lambda}, \beta_{\Lambda} \geq 1$, depending only on $\Lambda$ and $d$,  such that if $p \ge p_{\Lambda}$ and $|\beta| \ge \beta_{\Lambda}$ then
\begin{multline*}
\int_\Y \frac{r^{p + 2}e^{2\beta r^{-p}}}{2p|\beta|}\big[\Div (M\nabla v)\big]^2  
\ge  \int_\Y  \Lambda^{-2} p^3 \beta^2  r^{-2p - 2} e^{2 \beta r^{-p}}   |v|^2 -  C \Lambda  e^{2 \beta r^{-p}}  \langle M \nabla v, \nabla v \rangle  \\[6pt]
 -\int_\pY C \Lambda^2 r e^{2 \beta r^{-p}} \Big(  |\nabla v|^2 + \beta^2 p^2 r^{-2p -2} |v|^2 \Big) 
\end{multline*}
for some positive constant $C$ depending only on $d$. 
\end{lemma}

\begin{remark} \rm \label{rem-lem1} It is worth noting that we do not assume that $M$ has a positive lower bound in \Cref{lem1}; the term $ \langle M \nabla v, \nabla v \rangle $ still appears in the conclusion. We only assume \eqref{lem1-as00}  for $x \in \Y$. 
Moreover, the constant $\Lambda$ encodes only partly the information of the Lipschitz property of $M$ through \eqref{lem1-as00} and \eqref{lem1-as1}. 
Conditions \eqref{lem1-as00}-\eqref{lem1-as1} are satisfied for the new set of matrix-valued functions obtained from ${\M}^{\ell}$ by the conformal type map $\big(\rr \cos(\theta/n), \rr \sin (\theta/n), \tx \big)$, see \Cref{sect-heu} and the proof of \Cref{fund-pro}. 
\end{remark}

The proof of \Cref{lem1} is quite standard from Lemmas~\ref{lem-prepare1} and \ref{lem-prepare2}. For the convenience of the reader, we give the details in \Cref{ap-lem1} so that  one can keep track the dependence of constants on various parameters. 

\medskip 

Another ingredient in the proof of the key result of this section, \Cref{lem-main}, is the following standard result: 

\begin{lemma}\label{lem2}  Let $p \ge 1$,  $\beta \in \mR$, and $v \in H^2(\Y)$,  and let $M$ be a Lipschitz, symmetric, uniformly elliptic, matrix-valued function defined in $\Y$. Then
\begin{multline*}
 \int_\Y e^{2\beta r^{-p}} v \,  \Div (M\nabla v) +  \int_\Y  \frac{1}{2} e^{2 \beta r^{-p}} \langle M \nabla v, \nabla v \rangle \\[6pt]
 \le \int_\Y  C |M| \beta^2 p^2 r^{-2p -2} e^{2 \beta r^{-p}}  |v|^2+ \int_{\pY}  C |M| e^{2\beta r^{-p}} \Big(|\nabla v|^2 +  |v|^2 \Big) 
\end{multline*}
for some positive constant $C$ depending only on $d$. 
\end{lemma}

The proof of \Cref{lem2} is as usual. For the convenience of the reader, we reproduce the proof in  \Cref{ap-lem2}. 

\medskip 

Combining the inequalities in Lemmas \ref{lem1} and \ref{lem2}, we obtain
\begin{lemma}\label{lem-main}
	 Let $\beta \in \mR$,  $v \in H^2(\Y)$ and let $M$ be a Lipschitz, symmetric, uniformly elliptic, matrix-valued function defined in $\Y$. Assume that $|x| \le 1$ for $x \in \Y$, and,  for some $\Lambda \ge 1$, the following three conditions hold in $\Y$: 
$$
\langle M x, x \rangle \ge \Lambda^{-1} |x|^2, 
$$
$$
|M| + |\dive(Mx)| + |x|^{-2} |\nabla (x \cdot M x) \cdot Mx | \le \Lambda, 
$$
	 \begin{equation*}
	 |\langle B y, y \rangle| \le \Lambda \langle M y, y \rangle \mbox{ for } y \in \mR^d, 
	 \end{equation*}
where $\langle B y, y  \rangle $ is defined in \eqref{def-B}.  There exist two constants $p_{\Lambda}, \beta_\Lambda \geq 1$ such that if   $p \ge p_{\Lambda}$ and  $|\beta| \ge \beta_\Lambda$,  then
\begin{multline}
	\int_{\Y} \testfnn \Big(p^3 \beta^2 r^{-2p - 2}  |v|^2 +  \langle M \nabla v, \nabla v \rangle  \Big) \\[6pt]
	\leq C_{\Lambda} \int_\Y \frac{1}{p |\beta|} r^{p+2} \testfnn|\Div(M\nabla v)|^2   + C_{\Lambda} \int_{\pY}    e^{2\beta r^{-p}}  \Big(|\nabla v|^2 + p^2 \beta^2  r^{-2p -2}|v|^2 \Big),
\end{multline}
for some positive constant $C_\Lambda$ depending only on  $\Lambda$ and $d$. 
\end{lemma}

The proof of \Cref{lem-main} is quite standard after using Lemmas \ref{lem1} and \ref{lem2}. For the convenience of the reader, we ``reproduce" the proof in \Cref{lem-main} in \Cref{ap-lem-main}.

\subsection{Main step of the proof  Theorem~\ref{fund-thm}}  \label{sect-mainstep} This section, which is the main step of the proof of \Cref{fund-thm},  is devoted to the proof of  the following result

\begin{proposition} \label{fund-pro}  Let $d \ge 2$, $m \ge 1$, $\Lambda \ge 1$,  and $R_* < R < R^*$.  Then, for any  $\alpha \in (0, 1)$, there exists a  constant $\gamma_2 \in (0, 1)$, depending only on $\alpha$, $\Lambda$, $R_*$, $R^*$,   $m$, and $d$ such that
for every $\gamma_1 \in (0,  \gamma_2)$, there exists  $\gamma_0 \in (0, \gamma_1)$ depending only on $\alpha$, $\gamma_1$, $\gamma_2$, $\Lambda$, $R_*$, $R^*$, $m$, and $d$ such that,   for real, symmetric, uniformly elliptic, Lipschitz matrix-valued functions  $\M^\ell$ with $1 \le \ell \le m$  defined in $D_{\gamma_0}: = Y_{\gamma_0, 1,  R}$  verifying, in $D_{\gamma_0}$, 
\begin{equation}\label{fund-pro-pro-M}
\Lambda^{-1} |\xi|^2 \le \langle \M^{\ell} (x) \xi, \xi \rangle \le \Lambda |\xi|^2  \; \;  \forall \,  \xi \in \mR^d  \quad \mbox{ and } \quad   |\nabla \M^{\ell}(x) | \le \Lambda, 
\end{equation}
 for $g \in L^2(D_{\gamma_0})$, and for  $V = (V_1, \dots, V_m)  \in [H^2(D_{\gamma_0})]^m$ satisfying, for  $1 \le \ell \le m$,  
\begin{equation}\label{fund-pro-Ineq} 
|\dive (\M^\ell \nabla V_\ell)| \le  \Lambda_1 \big( |\nabla V| + |V| + |g| \big)  \mbox{ in } D_{\gamma_0} \mbox{ for some } \Lambda_1 \ge 0, 
\end{equation}
we have, with $\Sigma_{\gamma_0} = \partial D_{\gamma_0} \cap \big\{ x_1= 0 \big\}$,  
\begin{multline}\label{fund-pro-S}
\| V \|_{H^1(Y_{\gamma_1, \gamma_2, \frac{R}{2}})} \\[6pt]
 \le C \Big( \|(V, \nabla V)\|_{L^2(\Sigma_{\gamma_0})} + \| g\|_{L^2(D_{\gamma_0})} \Big)^{\alpha}  
\Big(\| V\|_{H^1(D_{\gamma_0})} + \|(V, \nabla V)\|_{L^2(\Sigma_{\gamma_0})} + \| g\|_{L^2(D_{\gamma_0})}  \Big)^{1-\alpha},  
\end{multline}
for some positive constant $C$ depending only on $\alpha, \, \gamma_1, \,  \Lambda, \, \Lambda_1, \, R_*, \, R^*, \,  m$, and  $d$. 
\end{proposition}

Recall that $Y_{\gamma_1, \gamma_2, R}$ is defined in \eqref{def-Y-M}.

The rest of this section consisting of two subsections. In the first one, we give the heuristic arguments when $d=2$, $\Lambda_1 = 0$, and $m=1$ (this means that one deals with an  elliptic equation instead of elliptic inequalities). The proof of \Cref{fund-pro} is given in the second subsection.

\subsubsection{Heuristic arguments in two dimensions for an elliptic equation}\label{sect-heu}
Before giving the proof of \Cref{fund-pro}, which is quite technical, let us describe here the heuristic arguments  in two dimensions when $m= 1$ and   $\Lambda_1= 0$. We first briefly reformulate heuristically the analysis  given in \cite{Ng-CALR-O} in the case  where $\M^1 = I$, i.e., $\Delta V = 0$. By scaling, one might  assume that $R = 1$, which is assumed from later on. 

Define, for $n > 1$ (large), 
\begin{equation}\label{heu-T}
T_n  (x) = \big(r^{1/n} \cos (\theta/n), r^{1/n} \sin (\theta/n) \big) \mbox{ for } x \in \mR^2_+, 
\end{equation}
and set 
\begin{equation}\label{heu-cT}
\T_n : = T_n (\mR^2_+) \cap \Big\{ x \in \mR^2;  1/ (4n) < |x| < 2/ n \Big\}. 
\end{equation}
Define 
\begin{equation}\label{heu-U}
U_n : = V  \circ T_n^{-1} \mbox{ for } x  \in \T_n. 
\end{equation}
Since $\Delta V = 0$ in $D_{\gamma_0}$, it follows that 
$$
\Delta U_n = 0 \mbox{ in } \T_n 
$$
provided that  $n$ is large but fixed and $\gamma_0$ is sufficiently small so that $T_n^{-1} (\T_n) \subset D_{\gamma_0}$. The assumption $d=2$ and $\Delta V = 0$ are essential here. Set 
$$
R_1 = R_1(n) = 1/n, \quad R_2  = R_2(n) = R_1(n) + 8/n^2, \quad R_3 = R_3 (n) = 5/(4n), 
$$
and denote 
$$
\hat Z_0 = \big( 1/n, 1/n  + \pi / (2n^2) \big)  \in \mR^2. 
$$
See \Cref{fig-heu} for a picture of $\T_n$, $\hat Z_0$, etc. 
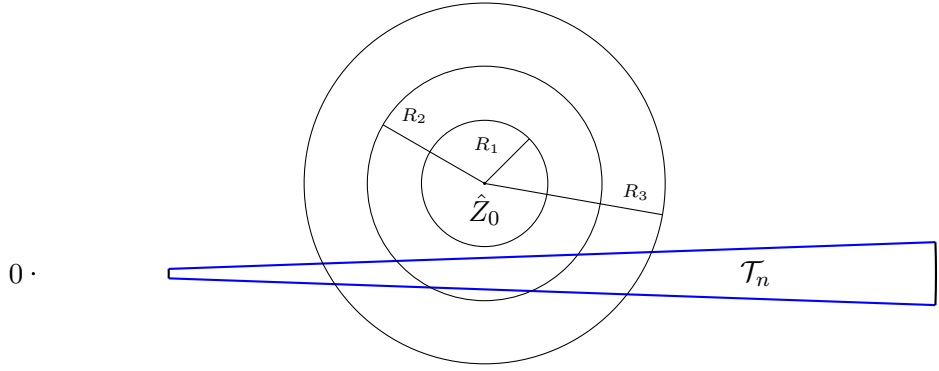
\begin{figure}
\centering
\begin{tikzpicture}[scale=1.2]

\fill[black](0,0)   circle(0.5pt);

\draw(0,0)   node[left]{$0$};

\draw[blue,thick] ({1.5*cos(2)}, {1.5*sin(2)}) -- ({10*cos(2)}, {10*sin(2)});
\draw[blue,thick] ({1.5*cos(2)}, {-1.5*sin(2)})  -- ({10*cos(2)}, {-10*sin(2)});
\draw [black,thick, domain =-2:2] plot ({10*cos(\x)}, {10*sin(\x)});
\draw [black,thick,domain =-2:2] plot ({1.5*cos(\x)}, {1.5*sin(\x)});

\draw [] (5,1) circle (1.3 cm);

\draw [] (5,1) circle (2. cm);

\draw [] (5,1) -- ({5+2*cos(-10)},{1+2*sin(-10)});

\draw  ({5+1.7*cos(-10)},{1+1.7*sin(-10)})  node[left, above]{{\tiny $R_3$}};

\draw[ -] (5,1) -- ({5+1.3*cos(150)},{1+1.3*sin(150)});
\draw[] ({5+1.2*cos(150)},{1.1+1.3*sin(150)}) node[right]{{\tiny $R_2$}};

\draw[black] (5,1) circle (0.7 cm);

\draw[-] (5,1) -- ({5 + 0.7*cos(45)}, {1+0.7*sin(45)});

\draw ({5 + 0.4*cos(45)}, {1+0.6*sin(45)})  node[left]{{\tiny $R_1$}};

\fill[black](5,1)   circle(0.5pt);

\draw (5, 1) node[right,below]{$\hat Z_0$};

\draw (8, 0) node[]{$\T_n$};

\end{tikzpicture}
\caption{Geometry of  $\T_n$ and $\hat Z_0$ in two dimensions.}\label{fig-heu}
\end{figure}

We can now derive the conclusion by  applying  the (Hadamard) there-sphere inequality for the functions $U_n \mathds{1}_{\T_n} - W_n$ for the spheres centered $\hat Z_0$ and of radius $R_1$, $R_2$, and $R_3$. Here $W_n$ is chosen such that  
$$
\Delta W_n = 0 \mbox{ in } B_1 \setminus \partial T_n, 
$$
$$
[\partial_\nu W_n] - [ \partial_\nu (U_n \mathds{1}_{\T_n}) ] = [W_n] - [ U_n \mathds{1}_{\T_n}]   = 0  \mbox{ on } \partial T_n \cap B_{4/(3n)}(\hat Z_0)
$$
(where $[\cdot]$ denotes the jump), 
and 
$$
\|W_n\|_{H^1(B_1 \setminus \partial \T_n)} \le C_n \| U_n \|_{\bH(\partial \T_n \cap B_{2/n}( \hat Z_0))}. 
$$ 
Since $\Delta (U_n \mathds{1}_{\T_n}- W_n) = 0$  in $B_{4/(3n)} (\hat Z_0)$, we have, by a variant of the Hadamard three-sphere inequality \footnote{The $L^2$-version for balls of Hadamard's three-sphere inequality is also valid.},  
\begin{equation}\label{heu-1}
\| U_n \mathds{1}_{\T_n} - W_n \|_{L^2\big(B_{R_2}(\hat Z_0) \big)} \le \| U_n \mathds{1}_{\T_n} - W_n \|_{L^2 \big(B_{R_1}(\hat Z_0) \big)}^{\alpha'} \| U_n \mathds{1}_{\T_n} - W_n \|_{L^2\big(B_{R_3}(\hat Z_0) \big)}^{1 - \alpha'},  
\end{equation}
with 
$$
\alpha' = \ln (R_3/ R_2) / \ln (R_3 / R_1) \mbox{ which is close to 1 for large $n$}.   
$$
Hence $\alpha' > \alpha$ for $n$ sufficiently large. Fix such an $n$.  Since $\Delta U_n = 0$ in $\T_n$, the conclusion follows from the fact that, since $\T_n \cap B_{R_1}(\hat Z_0) = \emptyset$,  
$$
U_n \mathds{1}_{\T_n} - W_n = - W_n \mbox{ in } B_{R_1}(\hat Z_0), 
$$ 
and 
$$
\big(B_{1/n + c/ n^2} \setminus  B_{1/n - c/ n^2}  \big) \cap \T_n  \subset B_{R_2} (\hat Z_0) \cap \T_n 
$$
for some positive constant $c$ independent of $n$. 

We now extend the previous ideas for the general case (still for $m = 1$ and $\Lambda_1 = 0$). We first note that the equation of $U_n$ given in \eqref{heu-U} can be obtained by a change of variables rule. Indeed, set, in $\T_n$,  
\begin{equation}\label{heu-K}
K_n (x)= \frac{\nabla T_n}{\big|\det \nabla T_n \big|^{1/2} } \circ T_n^{-1}  (x), 
\end{equation} 
and 
\begin{equation}\label{heu-A-M}
A_n(x) = \M \circ T_n^{-1} (x), \quad \mbox{ and } \quad 
M_n (x) = K_n A_n K_n\tr (x).
\end{equation}
We have, see e.g.  \cite[Lemma 6]{NgV-A} (see also \cite[Section 2.2]{KOVW10}, 
$$
\dive (M_n \nabla U_n) = 0 \mbox{ in } \T_n.   
$$
One can check that 
$$
K_n (x)  =   \left( \begin{array}{cc}
 \cos \big( (n-1)\theta \big)  &\sin \big( (n-1)\theta \big)  \\[6pt]
 - \sin \big( (n-1)\theta \big) &  \cos \big( (n-1)\theta \big)
\end{array}\right) 
$$ 
(see also Step 1 of the proof of \Cref{fund-pro}).  This yields, if $\M = I$, that $
M_n(x) = I$ as previous noted ($\Delta U_n = 0$). 

We next consider the case where $\M$ is not constant but still $\M(0) = I$. We thus still have $M_n(0) = I$. The idea now is to use Carleman's  estimates  in \Cref{sect-UL}  instead of the Hadamard three-sphere inequality.  To this end, 
set, $\lambda \in \big(5/(4 n),   4/(3 n)\big)$, 
 \begin{equation}\label{heu-def-O}
\hat \T_n = \T_n - \hat Z_0 \quad \mbox{ and } \quad  \quad  \Y_n   = B_{\lambda}  \cap \hat \T_n, 
\end{equation}
\begin{equation}\label{heu-LF-def-hM-0}
\hat M_n (\cdot) = M_n (\cdot + \hat Z_0) \mbox{ in }  \Y_n,  
\end{equation}
(see the geometry in \Cref{fig2}).  At this stage, one can assume that $\lambda$ is fixed. The role of $\lambda$ is to deal with the boundary terms which can be ignored here. 

One can then check, for large $n$ that, in $\Y_n$, 
\begin{equation}\label{heu-lem1-as00}
\langle \hat M_n x, x \rangle \ge \Lambda^{-1} |x|^2,
\end{equation}
\begin{equation}\label{heu-lem1-as01}
|\hat M_n| +  |\dive(\hat M_n x)| + |x|^{-2} |\nabla (x \cdot \hat M_n x) \cdot \hat M_n x | \le \Lambda, 
\end{equation}
\begin{equation}\label{heu-lem1-as1}
|\langle \hat B_n y, y \rangle| \le \Lambda \langle \hat M_n y, y \rangle \mbox{ for } y \in \mR^d, 
\end{equation}
where $\hat B_n$ is defined in \eqref{def-B} corresponding for $\hat M_n$. The assumptions made in \Cref{sect-UL} are motivated from the structure of $\hat M_n$ and the geometry of $\Y_n$.  Indeed, \eqref{heu-lem1-as00} holds since $\hat M_n (0) \sim I$ \footnote{When $d \ge 3$, one can check this inequality only for $x \in \Y_n$, see Step 1 in the proof of \Cref{fund-pro}.}. On the other hand, since $|K_n| \le C$,  $|\nabla K_n| \le C n^2$ (we emphasize here that we mean $Cn^2$ not $C n$) and $|x| \le C/n$ for $|x| \sim 1/n$, and 
\begin{align}\label{heu-key}
K_n A_n K_n \tr (x) = & K_n (x) (A_n(x) - A_n(0))  K_n (x) \tr + K_n (x) A_n(0)  K_n (x) \tr \nonumber \\[6pt]
= & K_n (x) (A_n (x) - A_n(0))  K_n (x) \tr + I, 
\end{align}
it follows that, since $\M$ is Lipschitz,  
$$
|\nabla (K_n A_n K_n \tr)| \le C/ n \mbox{ for } |x| \sim 1/n. 
$$
One then can derive   \eqref{heu-lem1-as01}. Similarly, one can check \eqref{heu-lem1-as1}. We can now apply the results in \Cref{sect-UL}, more precisely \Cref{lem-main}, to derive the conclusion. 

As seen in \eqref{heu-key}, one of the key points to verify \eqref{heu-lem1-as01} and \eqref{heu-lem1-as1} is to have $A_n(0)$ isotropic, i.e.,  $A_n(0)$ is a multiple of the identity (this condition will be relaxed in three and higher dimensions, see \eqref{LF-CV}).  
The previous definition of $T_n$ does not imply this property when $\M(0)$ is not isotropic. 
To overcome this issue, instead of \eqref{heu-T}, 
we decompose the map defined in \eqref{heu-T} with the linear application $H$ where $H x =  SQ^T x$. Here 
$$
S =  \mbox{diag} (\lambda_1^{-1/2},  \lambda_{2}^{-1/2}), 
$$
and $Q$ is a rotation such that 
$$
Q\tr \M(z_0) Q = \mbox{diag}(\lambda_1,  \lambda_2). 
$$
One can check in this case that $K_n A_n (0) K_n^T = |\det H|^{-1} I = \lambda_1^{1/2} \lambda_2^{1/2} I$ (where $A_n$ is now defined via the new map) and the analysis follows similarly. This point is simple but is one of the key observations to extend the analysis for the general case. 

In three and higher dimensions, the analysis is in the same spirit after somehow fixing the last $(d-2)$ variables. Nevertheless, the analysis is more involved and further decompositions are required to somehow fix the plane $(x_1, x_2)$ (for which another rotation $Q_1$ is introduced, see \eqref{def-Tn}). The details are given in \Cref{sect-fund-pro}.  

\medskip 

We are ready to give 

\subsubsection{Proof of \Cref{fund-pro}} \label{sect-fund-pro} By a scaling argument, one might assume that  $R = 1$.  For simplicity of presentation, we will assume that $m = 1$ and drop the corresponding indices (e.g. $\M^1$ becomes $\M$, etc). Using a covering argument, it suffices to prove that  there exists a constant $\gamma_2 \in (0, 1)$, depending only on $\alpha$, $\Lambda$, $R_*$, $R^*$,   and $d$ such that
for every $\gamma_1 \in (0,  \gamma_2)$, there exist  $\gamma_0 \in (0, \gamma_1)$ and $\widetilde \gamma_0  \in (0, \gamma_1) $ depending only on $\alpha$,  $\gamma_1$, $\Lambda$, $R_*$, $R^*$, and $d$ such that for all $\tz_0 \in \mR^{d-2}$ with $|\tz_0| \le 1/2$, we have 
\begin{multline}\label{LF-equiv}
\mathop{\mathop{\int_{D_{\gamma_0}}}_{\gamma_1 <  \rr < \gamma_2}}_{|\tx - \tz_0| < \widetilde \gamma_0} (|V|^2 + |\nabla V|^2 )  
\le C \Big(  \| (V, \nabla V)\|_{L^2(  \Sigma_{\gamma_0})} + \|g \|_{L^2(D_{\gamma_0})}\Big)^{2  \alpha } \times \\[6pt] \times  \Big(  \| V\|_{H^1(D_{\gamma_0})}^2 + \| (V, \nabla V)\|_{L^2(\Sigma_{\gamma_0})}^2 +  \|g\|_{L^2(D_{\gamma_0})} \Big)^{2(1- \alpha)}, 
\end{multline}
where $V \in H^2(D_{\gamma_0})$ satisfies \eqref{fund-pro-Ineq}. 

Our goal is now to establish \eqref{LF-equiv}.  Fix  $n \ge 10 \Lambda$ \footnote{$n$ is not assumed to be an integer to facilitate handling a technical points right after Step 5.}, define $L_n: \mR^d \cap \{x \in \mR^d; x_1 \ge 0\} \to \mR^d$ by  \footnote{When $(x_1, x_2) = (0, 0)$, we define $L_n(x) = (0, 0, \tx)$ as a  convention.}
\begin{equation}\label{LF-def-Ln}
L_n(x_1, x_2, \tx) = \big(\rr^{1/n} \cos (\theta/n), \rr^{1/n} \sin (\theta/n), \tx \big).  
\end{equation}
Recall that,  for $d \ge 2$ and $x = (x_1, x_2, \tx) \in \mR \times \mR \times \mR^{d-2}$, we use the polar coordinate $(\rr, \theta)$ for the pair $(x_1, x_2)$;  the variable $\tx$ is irrelevant for $d = 2$.

Fix $\tz_0 \in \mR^{d-2}$ with $|\tz_0| \le 1/2$. Denote 
$$
z_0 = (0, 0, \tz_0).
$$ 
Let $Q$ be a rotation, i.e.,  $Q\tr Q = I$,  and $\Lambda^{-1} \le \lambda_1, \cdots  , \lambda_d  \le \Lambda$ be such that
$$
Q\tr \M(z_0) Q = \mbox{diag}(\lambda_1, \cdots, \lambda_d). 
$$
Since $\M(z_0)$ is symmetric and uniformly elliptic, such $Q$ and $\lambda_j$ ($1 \le j \le d$) exist; in fact $\lambda_j$ ($1 \le j \le d$) are  eigenvalues of $\M(z_0)$ and $Q$ is formulated from a corresponding  orthogonal eigenvectors of $\M(z_0)$.  

Set  
$$
S =  \mbox{diag} (\lambda_1^{-1/2}, \cdots, \lambda_{d}^{-1/2}), 
$$

Let $Q_1$ be a rotation such that 
\begin{equation}\label{coucoucou1}
Q_1\tr e_1, Q_1\tr e_2  \perp SQ\tr \big( \{ (0, 0, \tx); \; \tx \in \mR^{d-2} \} \big). 
\end{equation}
Moreover, since $S^{-1} Q\tr e_1  \perp SQ\tr \big( \{ (0, 0, \tx); \; \tx \in \mR^{d-2} \} \big)$, one can assume in addition that 
\begin{equation}\label{coucoucou}
Q_1\tr e_1 = S^{-1} Q\tr e_1/ |S^{-1} Q\tr e_1|. 
\end{equation}
Since  $SQ\tr (\mR^d \cap \{x_1 > 0 \})  = \{x \in \mR^d; \langle x,  S^{-1} Q\tr e_1 \rangle > 0 \}$ and $Q_1 (S^{-1} Q\tr e_1/ |S^{-1} Q\tr e_1|) =  e_1$ by \eqref{coucoucou}, one derives that 
$$
Q_1 S Q\tr (\mR^d \cap \{x_1 > 0 \}) = \mR^d \cap \{x_1 > 0 \}.
$$ 
From \eqref{coucoucou1}, we have 
\begin{equation}\label{pro1-Ln}
\mbox{the first two components of  $Q_1S Q\tr (0, 0, \tx)$ are 0} \mbox{ for } \tx \in \mR^{d-2}. 
\end{equation}

Fix such a rotation $Q_1$. Define  
\begin{equation}\label{def-Tn}
T_n = L_n \circ H \mbox{ where } H (x) = H x \mbox{ with } H = Q_1 S Q\tr. 
\end{equation}
The definition  of $H$ and the choice  of $Q_1$, which are motivated from the heuristic arguments in \Cref{sect-heu},  will be clearer when establishing  \eqref{LF-CV}.  

Denote
\begin{equation}\label{LF-hZ}
Z_0 = T_n (z_0). 
\end{equation}
By the choices of $Q_1$ and $L_n$, the first two components of $Z_0$ are 0, which yields   
$$
Z_0 = (0, 0, \tZ_0), 
$$
for some $\tZ_0 \in \mR^{d-2}$.  Set 
\begin{equation}\label{LF-z0}
\hat Z_0 =  \big(1/n,  1/n +   \Lambda \pi /  (2n^2), \tZ_0 \big)  =  Z_0 + \big(1/n,  1/n +   \Lambda \pi/  (2n^2), 0\big),
\end{equation}
\begin{equation}\label{LF-T}
\T_n =   \Big\{x = T_n(y); \;  y = (y_1, y_2, \ty) \in  \mR^d_+, |\ty| < 1; 1/(4n) < \rr (x) < 2/n \Big\}, 
\end{equation}
\begin{equation}\label{LF-def-Yn}
Y_n = T_n^{-1} (\T_n), \quad \mbox{ and } \quad \Sigma_{Y_n} = \partial Y_n \cap \{x \in \mR^d; x_1 = 0\} 
\end{equation}
(the geometry of $\T_n - \hat Z_0$ is given in \Cref{fig2} when  $d=2$). 
Define, in $\T_n$,  
\begin{equation}\label{LF-K-C-0}
K_n (x)= \frac{\nabla T_n}{\big|\det \nabla T_n \big|^{1/2} } \circ T_n^{-1}  (x), 
\end{equation} 
\begin{equation}\label{LF-def-M-1-C}
A_n(x) = \M \circ T_n^{-1} (x), \quad \mbox{ and } \quad 
M_n (x) = K_n A_n K_n\tr (x).
\end{equation}
The quantity $M_n$ appears in a change of variables $U_n  = V \circ T_n^{-1}$ (see \eqref{LF-def-VM}, \eqref{def-f-F} and \eqref{form-CV}).

For 
\begin{equation}\label{LF-lambda}
\lambda \in \big(5/(4 n),   4/(3 n)\big),
\end{equation}
set (see \Cref{fig2})
\begin{equation}\label{def-O}
\hat \T_n = \T_n - \hat Z_0 \quad \mbox{ and } \quad  \quad  \Y_n   = B_{\lambda}  \cap \hat \T_n, 
\end{equation}
\begin{equation}\label{LF-def-hM-0}
\hat M_n (\cdot) = M_n (\cdot + \hat Z_0) \mbox{ in }  \Y_n,  
\end{equation}
and, for $x \in \Y_n$ and $y \in \mR^d$, 
\begin{multline}\label{hB-LF}
  \langle \hat B_n (x) y, y \rangle =   \langle [(\hM_n y) \cdot \nabla_x]  (\hM_n(x) x), y \rangle +  \frac{1}{2} \langle \dive_x(\hM_n x) \hM_n(x) y, y \rangle \\[6pt]
  + \frac{1}{2} \langle [(\hM_n(x) x) \cdot \nabla_x ] \hM_n(x)  y , y \rangle. 
\end{multline}
Note that $\Y_n$ also depends on $\lambda$; however, the dependence is not written explicitly for notational ease.  

\medskip 

We are planning to derive \eqref{LF-equiv}  by applying \Cref{lem-main} to $\hM_n$ and $\hU_n$ (defined in \eqref{LF-def-VM} below) in $\Y_n$ for appropriate $n$. 
\medskip 

We claim that 
\begin{equation}\label{LF-claim0-C}
\langle \hM_n x, x \rangle \ge \hLambda^{-1} |x|^2 \quad \mbox{ in } \Y_n, 
\end{equation}
\begin{equation}\label{LF-claim1-C}
|\dive(\hM_n x)| + |x|^{-2} |\nabla (x \cdot \hM_n x) \cdot \hM_n x | \le \hLambda 
 \quad \mbox{ in } \Y_n,  
\end{equation}
\begin{equation}\label{LF-claim2-C}
|\langle \hB_n y, y \rangle | \le \hLambda \langle \hM_n y, y \rangle \quad   \mbox{ in } \Y_n, 
\end{equation}
for some $\hLambda \ge 1$,  for all $\lambda \in \big(5/(4 n),   3/(2 n)\big)$. Here and in what follows, $\hLambda$ denotes a positive constant  depending only on $\Lambda$  and $d$;  it is thus independent of $\tz_0$ and $n$. The proof of this claim is given in Step 1 below.

Let $p = p_{\hLambda}$ where $p_{\hLambda}$ is  the constant in \Cref{lem-main} corresponding to $\hLambda$ and $\hM_n$.  Set
\begin{equation}\label{def-ts-n}
\tau_n = (1/n  - \Lambda/ n^2)^n \quad \mbox{ and }  \quad s_n = (1/n  + \Lambda/ n^2)^n.
\end{equation}
Denote 
\begin{equation}\label{LF-def-R13}
R_1(n)   = 1/ n, \quad R_2 (n)    =  R_{1}(n) + 8 \Lambda / n^2,  \quad \mbox{ and } \quad  R_3 (n)   = 5/(4n),
\end{equation}
and define
\begin{equation}\label{LF-def-ha}
\rho (n) = \frac{R_1(n)^{-p} -  R_3 (n)^{-p}}{R_{2}(n)^{-p} -  R_3(n)^{-p}}.
\end{equation}
Note that 
$$
\lim_{n \to + \infty} \rho(n) = 1. 
$$

Let 
$$
n_0  = \min \Big\{n \in \mN; n \ge 10 \Lambda \mbox{ and } \rho(n) \ge (1 + \alpha)/2 \Big\}. 
$$
Set 
$$
\mbox{$\gamma_2 = s_{n_0}/\Lambda$ where $s_n$ is defined in \eqref{def-ts-n}}.
$$
Given $\gamma_1 < \gamma_2$, fix $N \in \mN$ with $N > n_0$ and  $\Lambda \tau_N \le \gamma_1$ where $\tau_n$ is defined in \eqref{def-ts-n}. Set 
$$
\gamma_0 = \Lambda^{-1}/(4N)^N. 
$$ 
In what follows in this proof, we always assume that $n_0 \le n \le N$. 
Define
\begin{equation}\label{LF-def-VM}
U_n(x) = V \circ T_n^{-1} (x)  \mbox{ for } x \in \T_n,  \quad \hU_n (\cdot) = U_n(\cdot + \hat Z_0)  \mbox{ in } \Y_n,  
\end{equation}
\begin{equation}\label{LF-def-hM}
g_n(x) = f \circ T_n^{-1} (x)  \mbox{ for } x \in \T_n,  \quad \mbox{ and } \quad \hg_n (\cdot) = g_n(\cdot + \hat Z_0)  \mbox{ in } \Y_n. 
\end{equation}

\begin{figure}
\centering
\begin{tikzpicture}[scale=1.2]


\begin{scope}
\clip (2:1) arc (2:-2: 1.5) -- (-2:12) arc (-2:2:12) -- cycle;
\fill[black!20] (5,1) circle (2.4 cm);
\end{scope}


\fill(0,0)   circle(0.5pt);

\draw(0,0)   node[left]{$- \hat Z_0$};

\draw[blue,thick] ({1.5*cos(2)}, {1.0*sin(2)}) -- ({10*cos(2)}, {10*sin(2)});
\draw[blue,thick] ({1.5*cos(2)}, {-1.5*sin(2)})  -- ({10*cos(2)}, {-10*sin(2)});
\draw [black,thick, domain =-2:2] plot ({10*cos(\x)}, {10*sin(\x)});
\draw [black,thick,domain =-2:2] plot ({1.5*cos(\x)}, {1.5*sin(\x)});

\draw[->] (2,-1) --({3.5*cos(2)}, {-3.5*sin(2)});
\draw (2, -1) node[below]{$\partial \Y_n \cap \hat \Sigma_n$};

\draw[->] (9,-1) --({5+2.4*cos(25)}, {1-2.4*sin(25)});

\draw (9, -1) node[below]{$\partial \Y_n   \cap \partial B_\lambda$};

\draw [red,thick] (5,1) circle (2.4 cm);

\draw[blue, -] (5,1) -- ({5+2.4*cos(180)},{1+2.4*sin(180)});

\draw[blue] ({5+1.6*cos(180)},{1+1.6*sin(180)}) node[below]{{\tiny $\lambda$}};

\draw [orange] (5,1) circle (1.3 cm);

\draw[orange, -] (5,1) -- ({5+1.3*cos(150)},{1+1.3*sin(150)});
\draw[orange] ({5+1.2*cos(150)},{1.1+1.3*sin(150)}) node[right]{{\tiny $R_2(n)$}};

\draw[black] (5,1) circle (0.7 cm);

\draw[-] (5,1) -- ({5 + 0.7*cos(45)}, {1+0.7*sin(45)});

\draw ({5 + 0.4*cos(35)}, {1+0.6*sin(35)})  node[left]{{\tiny $R_1(n)$}};

\draw[->] ({5.5+3.5*cos(30)},{3.5*sin(30)}) --(5.5,0);

\draw({6+3*cos(30)},{4.5*sin(30)})   node[right]{$ B_{R_2(n)} \cap \hat \T_n $};

\draw(6.25,0)   node[right]{$\Y_n$};

\draw (5, 1) node[]{$.$};

\draw (5, 1) node[right,below]{$0$};

\draw (8.8, 0) node[]{$\hat \T_n = \T_n - \hat Z_0$};

\end{tikzpicture}
\caption{Geometry of  $\Y_n = \hat \T_n \cap B_\lambda$ and $\Y_n$ in two dimensions.}\label{fig2}
\end{figure}
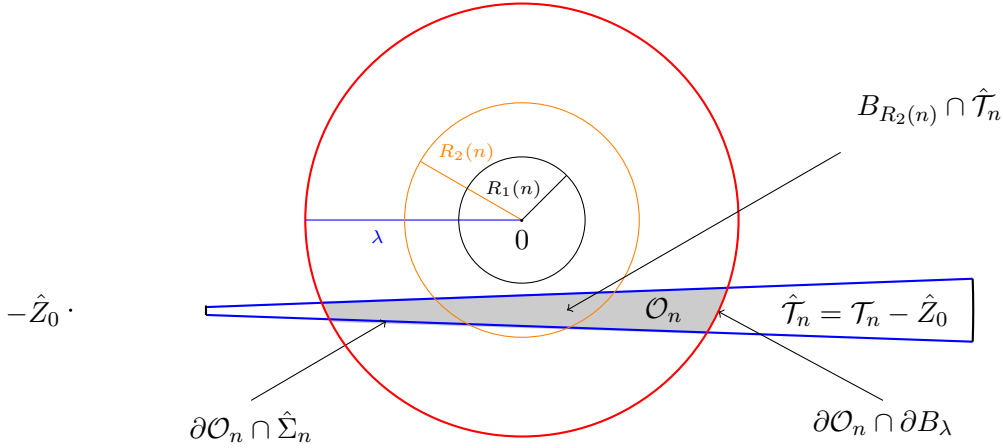

\medskip 

The proof of \eqref{LF-equiv} is now divided into the following five steps: 

$\bullet$ Step 1: We prove  \eqref{LF-claim0-C}, \eqref{LF-claim1-C}, and \eqref{LF-claim2-C}.

$\bullet$ Step 2: Using \eqref{LF-claim0-C}, \eqref{LF-claim1-C}, and \eqref{LF-claim2-C}, 
 we prove that there exists  a constant $\beta_{\Lambda} \ge 1$ depending only on $n_0$, $N$,  $\Lambda$, and $d$, such that, for $\beta \ge \beta_{\Lambda}$,  it holds
\begin{equation}\label{LF-step3.2}
	\int_{\Y_n} \testfnn \Big(\beta^2 e^{2 \beta r^{-p}} |\hU_n|^2 + |\nabla \hU_n|^2  \Big)
	\le  C \Big( \int_{\pY_n}   \beta   e^{2\beta r^{-p}}  \big(|\nabla \hU_n|^2 + \beta^2  |\hU_n|^2 \big) +  \int_{\Y_n} \testfnn |\hg_n|^2 \Big), 
\end{equation}
for all $\lambda \in \big(5/(4 n),   4/(3 n)\big)$.  Here and in what follows in this proof, $C$ denotes a positive constant depending only on $n_0$, $N$, $\Lambda$, and $d$. 

$\bullet$ Step 3: Set
\begin{equation*}
\wSigma_{n} = T_n (\Sigma_{Y_n}) \quad \mbox{ and } \quad \hat \Sigma_{n} = \wSigma_{n} - \hat Z_0
\end{equation*}
(see \eqref{LF-def-Yn} for the definition of $\Sigma_{Y_n}$). 
Using Step 2, we prove, for $\beta \ge \beta_{\Lambda}$, 
\begin{multline}\label{LF-step3.3}
	\int_{\Y_n}   (|\hU_n|^2 + |\nabla \hU_n|^2 ) e^{2 \beta r^{-p}} 	\\[6pt]\le C \beta^2  e^{2 \beta R_3^{-p} } \| \hU_n\|_{H^1(\hat \T_n)}^2  +  C \beta^2  e^{2 \beta R_1^{-p} }  \Big( \| (\hU_n, \nabla \hU_n)\|_{L^2( \hat \Sigma_n)}^2 + \| \hg_n\|_{L^2(\Y_n)}^2 \Big),  
\end{multline}
for some $\lambda \in \big(5/(4 n),   4/(3 n)\big)$.

$\bullet$ Step 4:  Using Step 3, we prove 
\begin{multline}\label{LF-step3.4}
\mathop{\int}_{B_{R_2(n)} \cap \hat \T_n}  (|\hU_n|^2 + |\nabla \hU_n|^2 ) \le C  \Big( \| (\hU_n, \nabla \hU_n)\|_{L^2( \hat \Sigma_n)}^2 + \| \hg_n\|_{L^2(\Y_n)}^2 \Big)^{2 \alpha } \times  \\[6pt] \times \Big(  \| \hU_n\|_{H^1(\hat \T_n)}^2 + \| (\hU_n, \nabla \hU_n)\|_{L^2( \hat \Sigma_n)}^2 + \| \hg_n\|_{L^2(\Y_n)}^2 \Big)^{2(1- \alpha)}.  
\end{multline} 

$\bullet$ Step 5:  Using Step 4, we prove
\begin{multline}\label{LF-step3.5}
\mathop{\mathop{\int}_{  \Lambda \tau_n  \le \rr  \le  \Lambda^{-1} s_n }}_{|\tx - \tz_0| < C/n^2} (|V|^2 + |\nabla V|^2 )  \le C  \Big( \| (V, \nabla V)\|_{L^2(  \Sigma_{\gamma_0})} + \|g \|_{L^2(D_{\gamma_0})}^2 \Big)^{2  \alpha }  \times \\[6pt] \times \Big(  \| V\|_{H^1(D_{\gamma_0})}^2 + \| (V, \nabla V)\|_{L^2(\Sigma_{\gamma_0})}^2 + \|g \|_{L^2(D_{\gamma_0})}^2 \Big)^{2(1- \alpha)}, 
\end{multline}
for some positive constant $C$ depending only on $\Lambda$ and  $d$.

Note that 
$$
\lim_{n \to + \infty} s_n / \tau_n = e^{2 \Lambda}. 
$$
Since \eqref{LF-step3.5} holds for $n_0 \le n \le N$ even for $n$ real, assertion~\eqref{LF-equiv} is now a consequence of the assertion in Step 5. 

\medskip

The proof of Steps 1 and 2 are new in comparison with the standard arguments in Carleman's estimates.  The proof of Steps 3, 4, and 5 are  quite standard after Step 2. We now proceed Steps 1-5.

\medskip 
\noindent $\bullet$ Step 1:  We have  
\begin{equation}\label{LF-T-C}
\nabla L_n (x)=  \left(\begin{array}{ccc}
L_{2, 2, n}(x) & 0_{2, d-2} \\[6pt]
0_{d-2, 2} & I_{d-2, d-2}
\end{array}\right), 
\end{equation}
where 
\begin{align*}
L_{2, 2, n} (x) = & \frac{1}{n \rr^{1 - 1/n}}   \left(
\begin{array}{cc}
 \cos \theta \cos (\theta/ n)  + \sin \theta  \sin (\theta/ n) & \sin \theta \cos (\theta/ n)  - \cos \theta  \sin (\theta/ n)  \\[6pt]
 \cos \theta \sin (\theta/ n)  - \sin \theta  \cos (\theta/ n) &  \sin \theta \sin (\theta/ n)  +  \cos \theta  \cos (\theta/ n)
\end{array}\right)\\
= & \frac{1}{n \rr^{1 - 1/n}}  \left(
\begin{array}{cc}
 \cos \big( (n-1)\theta/n \big)  &\sin \big( (n-1)\theta/n \big)  \\[6pt]
 - \sin \big( (n-1)\theta/n \big) &  \cos \big( (n-1)\theta/n \big)
\end{array}\right).
\end{align*}
Here and in what follows,  $0_{i, j}$ denotes the zero $(i \times j)$-matrix and $I_{k,k}$ denotes the identity matrix of size $(k \times k)$ for $i, j, k \ge 0$.

Set 
\begin{equation}\label{LF-def-tK}
\tK_n (x)=  \left(\begin{array}{ccc}
\tK_{2, 2, n}(x) & 0_{2, d-2} \\[6pt]
0_{d-2, 2} & n \rr^{n-1}I_{d-2, d-2} 
\end{array}\right),
\end{equation}
where 
\begin{equation*}  \tK_{2, 2, n} (x) =    \left(
\begin{array}{cc}
 \cos \big( (n-1)\theta \big)  &\sin \big( (n-1)\theta \big)  \\[6pt]
 - \sin \big( (n-1)\theta \big) &  \cos \big( (n-1)\theta \big)
\end{array}\right).
\end{equation*}
It is clear from the formula of $\nabla L_n$ that, in $\T_n$,  
\begin{equation}\label{LF-eq-nabla-Tn}
 \nabla T_n \circ T_n^{-1}(x) = \nabla L_n \big(L_n^{-1}(x) \big) H = \frac{1}{n \rr^{n-1}} \tK_n (x) H, 
\end{equation} 
\begin{equation}\label{LF-det-T}
|\det \nabla T_n| \circ T_n^{-1} (x) = \frac{1}{\gamma n^2 \rr^{2n - 2}} \mbox{ where } \gamma =  |\det H|^{-1}, 
\end{equation}
and, by  \eqref{LF-K-C-0}, 
\begin{equation}\label{LF-K-C}
K_n (x)=  \gamma^{1/2}  \tK_n H. 
\end{equation}

From \eqref{LF-K-C}, we have
\begin{equation}\label{LF-c4-C}
|K_n| \le \hLambda \mbox{ in } \T_n \quad \mbox{ and } \quad |\nabla K_n| \le C / n^2  \mbox{ in } \T_n. 
\end{equation}
Note that  $1/ n \le \rr$ and $|x| \le 4/(3n)$ for $x \in \Y_n$,  and $|K_{2, 2, n}(x) y'| = |y'|$ for $x \in \T_n$ and $y' \in \mR^2$.  We derive from \eqref{LF-K-C} that 
\begin{equation*}
|K_n\tr(\cdot + \hat Z_0) x| \ge \hLambda |x| \mbox{ for } x \in \Y_n. 
\end{equation*}
It follows from the ellipticity of $\M$, \eqref{LF-def-M-1-C}, and \eqref{LF-def-hM-0} that 
\begin{equation*}
\langle \hM_n x, x \rangle \ge \hLambda^{-1} |x|^2 \mbox{ in } \Y_n,  
\end{equation*}
which is \eqref{LF-claim0-C}. 

Since, by \eqref{LF-def-tK} and  \eqref{LF-eq-nabla-Tn}, 
\begin{equation}\label{LF-nablaTn-1}
\nabla T_n^{-1} (x) = \Big(\nabla T_n \circ T_n^{-1} \Big)^{-1} = H^{-1}  \left(\begin{array}{ccc}
n \rr^{n-1} \tK_{2, 2, n}^{-1} & 0_{2, d-2} \\[6pt]
0_{d-2, 2} & I_{d-2, d-2}
\end{array}\right), 
\end{equation}
we deduce from \eqref{LF-def-M-1-C} that
\begin{equation}\label{LF-c4-1-C}
|\nabla A_n (x)| \le \hLambda \mbox{ in } \T_n.
\end{equation}

From \eqref{LF-hZ}, \eqref{LF-def-M-1-C}, and \eqref{LF-K-C}, we obtain 
\begin{equation*}
M_n(Z_0) = K_n  A_n K_n\tr (Z_0) = \gamma \tK_n (Z_0) H \M(z_0) H\tr \tK_n\tr (Z_0)  = \gamma \tK_n (Z_0) \tK_n\tr (Z_0)
\end{equation*}
(thanks to the choice of $H$), which yields, by \eqref{LF-def-tK},  
\begin{equation}\label{LF-CV}
M_n(Z_0) =   \gamma P \mbox{ where } P: =  \left(\begin{array}{ccc}
I_{2, 2} & 0_{2, d-2} \\[6pt]
0_{d-2, 2} & n^2 \rr^{2(n-1)}I_{d-2, d-2} 
\end{array}\right). 
\end{equation}
This fact plays an important role in establishing 
\eqref{LF-claim1-C} and \eqref{LF-claim2-C} as mentioned in \Cref{sect-heu}. 

We have 
\begin{equation*}
\dive (\hM_n x) = \sum_{j = 1}^{d} \partial_{x_j} \langle \hM_n x, e_j \rangle =  \sum_{j = 1}^{d} \partial_{x_j}  \big( \langle \hM_n e_i, e_j \rangle x_i  \big)\mbox{ in } \Y_n, 
\end{equation*}
\begin{equation*}
\langle M_n e_i, e_j \rangle  
\mathop{=}^{\eqref{LF-CV}}  \gamma  P_{ij} + \big\langle K_n  \big(A_n-A_n(Z_0) \big) K_n\tr e_i, e_j \big\rangle, 
\end{equation*}
where $P_{ij}$ denotes the $(i, j)$ component of $P$, and 
\begin{equation*}
|x + \hat Z_0 - Z_0| \le 5/n \mbox{ for } x \in \Y_n,  \quad \mbox{ and } \quad |x| \le 4/(3n) \mbox{ for } x \in \Y_n. 
\end{equation*}
It follows from \eqref{LF-c4-C} and  \eqref{LF-c4-1-C} that
\begin{equation}\label{LF-M-e1-C}
|\dive (\hM_n x)| \le \hLambda.  
\end{equation}

We have 
$$
\langle M_n y, y \rangle \mathop{=}^{\eqref{LF-CV}} \gamma \langle Py, y \rangle  + \big\langle K_n \big(A_n - A_n(Z_0) \big) K_n\tr y, y \big\rangle \mbox{ for } y \in \mR^d,  
$$
$$
\nabla (x \cdot \hM_n x ) \cdot \hM_n x =\sum_{j=1}^{d} \partial_{x_j} (x \cdot \hM_n x) \langle \hM_n x, e_j \rangle \mbox{ in } \Y_n,  
$$
$$
|x + \hat Z_0 - Z_0| \le 5/n \mbox{ for } x \in \Y_n, \quad \mbox{ and } \quad |x| \le 4/(3n) \mbox{ for } x \in \Y_n. 
$$
Using these facts,  we derive from \eqref{LF-c4-C} and  \eqref{LF-c4-1-C} that  
\begin{equation}\label{LF-M-e2-C}
|\nabla (x \cdot \hM_n x ) \cdot \hM_n x| \le \hLambda |x|^2  \mbox{ in } \Y_n. 
\end{equation}
Combining \eqref{LF-M-e1-C} and \eqref{LF-M-e2-C} yields \eqref{LF-claim1-C}. 

\medskip 
Using the same arguments, one also obtains  \eqref{LF-claim2-C}. The proof of Step  1 is complete. 
\medskip

$\bullet$ Step 2.  Using \eqref{LF-claim0-C}, \eqref{LF-claim1-C},  and \eqref{LF-claim2-C}, we can  
apply Lemma~\ref{lem-main} with $v = \hU_n$ and $M = \hM_n$ in $\Y_n$.  We then obtain
\begin{multline}\label{LF-step3.2-1}
	\int_{\Y_n} \testfnn \Big(p^3 \beta^2 r^{-2p - 2} |\hU_n|^2 +  \langle \hM_n \nabla \hU_n, \nabla \hU_n \rangle  \Big) \\[6pt] \leq C \int_{\Y_n} \frac{1}{p |\beta|} r^{p+2} \testfnn|\Div(\hM_n \nabla \hU_n)|^2   \\[6pt] + C \int_{\pY_n}   |\beta| p r e^{2\beta r^{-p}}  \big(|\nabla \hU_n|^2 + p^2 \beta^2  r^{-2p -2}|\hU_n|^2 \big), 
\end{multline}
for  $|\beta| \ge \beta_{\Lambda}$ for some constant  $\beta_{\Lambda} \ge 1$ depending only on $\Lambda$  and $d$ since $\hat \Lambda$ depends only on $\Lambda$  and $d$. 

We claim that 
\begin{equation}\label{LF-step2}
|\dive (M_n \nabla U_n)| \le C \left(\frac{1}{n^3 \rr^{3n-3}}|K_n\tr \nabla U_n| + \frac{1}{n^2 \rr^{2n-2}}|U_n| + \frac{1}{n^2 \rr^{2n-2}} |g_n| \right)  \mbox{ in } \T_n. 
\end{equation}
Indeed, set
\begin{equation}\label{def-f-F}
f (x) = \dive \big(\M(x) \nabla V(x) \big) \mbox{ for  } x \in   \Omega
\quad \mbox{ and }  \quad F_n (x) = \frac{ f \circ T_n^{-1}}{|\det \nabla T_n| \circ T_n^{-1}} (x) \mbox{ for } x \in \T_n.
\end{equation}
Then, in $\T_n$,  
\begin{equation}\label{LF-Fn}
|F_n(x)| \mathop{\le}^{\eqref{LF-det-T}} \frac{1}{\gamma n^2 \rr^{2n-2}} |f \circ T_n^{-1}(x)|.
\end{equation}
By a change of variables, see, e.g. \cite[Lemma 6]{NgV-A} (see also \cite[Section 2.2]{KOVW10}), we have 
\begin{equation}\label{form-CV}
\dive (M_n \nabla U_n) = F_n \mbox{ in } \T_n. 
\end{equation}
Recall that 
$$
|f| \le \Lambda_1 (|\nabla V| + |V| + |g|) \mbox{ in } T_n^{-1} (\T_n). 
$$
We have, for $x \in \T_n$, 
$$
\nabla V \big(T_n^{-1} (x) \big) \mathop{=}^{\eqref{LF-def-VM}} \nabla T_n \tr  \big(T_n^{-1} (x)\big) \nabla U_n (x) \mathop{=}^{\eqref{LF-eq-nabla-Tn}, \eqref{LF-K-C}}  \frac{1}{\gamma^{1/2} n \rr^{n-1}} K_n\tr(x) \nabla U_n (x).
$$
It follows from \eqref{LF-Fn} that  
$$
|F_n| \le C \left(\frac{1}{n^3 \rr^{3n-3}}|K_n\tr \nabla U_n| + \frac{1}{n^2 \rr^{2n-2}}|U_n| +  \frac{1}{n^2 \rr^{2n-2}}|g_n| \right) \mbox{ in } \T_n, 
$$
which implies claim  \eqref{LF-step2}. 

We have, in $\T_n$,  
\begin{equation}\label{LF-lb-M-0}
\langle M_n y, y \rangle \mathop{\ge}^{\eqref{LF-def-M-1-C}} C |K_n\tr y|^2  \mbox{ for } y \in \mR^d. 
\end{equation}

Considering \eqref{LF-step3.2-1},  and using \eqref{LF-step2} and \eqref{LF-lb-M-0}, we deduce  that, where $\rr$ is considered at the point $x + \hat Z_0$ instead of the origin, i.e. $\rr = \rr(x + \hat Z_0)$,   
\begin{multline}\label{LF-step3.2-2}
	\int_{\Y_n} \testfnn \Big(p^3 \beta^2 r^{-2p - 2} |\hU_n|^2 +  |\hat K_n\tr \nabla \hU_n|^2  \Big) \\[6pt] 
	\leq C \left\{ \int_{\Y_n} \frac{1}{ p |\beta|} r^{p+2} e^{2\beta r^{-p}} \left(    \frac{1}{n^4 \rr^{4n -4}}| \hU_n|^2 + \frac{1}{n^6 \rr^{6n -6}}  |\hat K_n\tr \nabla \hU_n|^2 +  \frac{1}{n^4 \rr^{4n -4}}| \hg_n|^2  \right) \right.  \\[6pt] + \left. \int_{\pY_n}   |\beta| p r e^{2\beta r^{-p}}  \big(|\nabla \hU_n|^2 + p^2 \beta^2  r^{-2p -2}|\hU_n|^2 \big) \right\}. 
\end{multline} 

Fix $\beta_\Lambda$ large, the largeness depends only on $n_0$, $N$, $\Lambda$ and $d$, so that for $|\beta| \ge \beta_\Lambda$, $n_0 \le n \le N$, $1/(4n) < \rr < 2/n$, and $r> 1/n$,  we have 
$$
p^3 \beta^2 r^{-2p - 2} \ge   \frac{C}{2 p |\beta|} r^{p+2} \frac{1}{n^4 \rr^{4n -4}}
\quad \mbox{ and } \quad 
1 \ge   \frac{2C}{ p |\beta|} r^{p+2} \frac{2}{n^6 \rr^{6n -6}}, 
$$
where $C$ is the constant appearing in  \eqref{LF-step3.2-2}. We derive from the definition of $\Y_n$ and $\T_n$ that, for $|\beta| \ge \beta_{\Lambda}$,
\begin{equation*}
	\int_{\Y_n} \testfnn \Big(\beta^2 e^{2 \beta r^{-p}} |\hU_n|^2 + | \nabla \hU_n|^2  \Big)
	\le  C \int_{\pY_n}   |\beta|  e^{2\beta r^{-p}}  \big(|\nabla \hU_n|^2 + \beta^2  |\hU_n|^2 \big) + C \int_{\Y_n} \testfnn  |\hg_n|^2. 
\end{equation*}
The proof of Step 2 is complete. 

\medskip
$\bullet$ Step 3. In what follows, for notational ease, we denote $R_1(n)$, $R_2(n)$, and $R_3(n)$ by $R_1$, $R_2$, and $R_3$. We have 
$$
\partial \Y_n = (\partial \Y_n \cap \partial B_\lambda) \cup (\partial \Y_n \setminus \partial B_\lambda),  
$$
$$
|x| \ge  R_3 \mbox{ for } x \in  \partial \Y_n \cap \partial B_\lambda \mbox{ by \eqref{LF-lambda} and the definition of $R_3$, } 
$$
$$
|x|  \ge R_1 \mbox{ for } x \in (\partial \Y \setminus \partial B_\lambda) \mbox{ since } \dist (z_0, \T_n) \ge 1/ n.  
$$
We derive from \eqref{LF-def-R13} and \eqref{LF-step3.2} that,   for $\beta \ge \beta_\Lambda$,  
\begin{multline}\label{LF-p1}
	\int_{\Y_n}   (|\hU_n|^2 + |\nabla \hU_n|^2 ) e^{2 \beta r^{-p}} 	
	\le C  \beta^2  e^{2 \beta R_3^{-p} } \| (\hU_n, \nabla \hU_n)\|_{L^2(\pY_n \cap \partial B_{\lambda} )}^2 \\[6pt] +  C  \beta^2  e^{2 \beta R_1^{-p} }  \Big( \| (\hU_n, \nabla \hU_n)\|_{L^2(\pY_n \setminus \partial B_{\lambda})}^2 + \|\hg_n \|_{L^2(\Y_n)}^2 \Big). 
\end{multline}
Since $n \ge 10 \Lambda$, we have 
$$
\pY_n \cap \partial B_{\lambda}  \subset \hat \T_n \mbox{ for all } \lambda \in \big(5/(4n), 4/(3n)\big). 
$$
This implies that, for some $\lambda \in \big(5/(4n), 4/(3n)\big)$, 
\begin{equation}\label{LF-p2}
\| (\hU_n, \nabla \hU_n)\|_{L^2(\pY_n \cap \partial B_{\lambda} )} \le C \| \hU_n\|_{H^1(\hat \T_n)}. 
\end{equation}
It is clear that 
$$
\pY_n \setminus \partial B_{\lambda} \subset \hat \Sigma_n, 
$$
which yields 
\begin{equation}\label{LF-p3}
\| (\hU_n, \nabla \hU_n)\|_{L^2(\pY_n \setminus \partial B_{\lambda})} \le \| (\hU_n, \nabla \hU_n)\|_{L^2(\hat \Sigma_n)}. 
\end{equation}
Combining \eqref{LF-p1}, \eqref{LF-p2}, and \eqref{LF-p3} yields,  for $\beta \ge \beta_\Lambda$,  
\begin{multline*}
	\int_{\Y_n}   (|\hU_n|^2 + |\nabla \hU_n|^2 ) e^{2 \beta r^{-p}} \\[6pt]	\le C  \beta^2  e^{2 \beta R_3^{-p} } \| \hU_n\|_{H^1(\hat \T_n)}^2  +  C \beta^2  e^{2 \beta R_1^{-p} }\Big( \| (\hU_n, \nabla \hU_n)\|_{L^2(\hat \Sigma_n)}^2 + \|\hg_n \|_{L^2(\Y_n)}^2 \Big).
\end{multline*}
The proof of Step 3 is complete. 

\medskip 
$\bullet$ Step 4. Note that, for  $\lambda \in \big(5/(4n), 4/(3n) \big)$, 
$$
B_{R_2}  \cap \hat \T_n \subset B_{5/(4n)}  \cap \hat \T_n \subset \Y_n. 
$$
As a consequence of \eqref{LF-step3.3}, we have
\begin{multline*}
	\int_{B_{R_2}  \cap \hat \T_n}   (|\hU_n|^2 + |\nabla \hU_n|^2 ) e^{2 \beta r^{-p}}  \\[6pt]	\le C  \beta^2  e^{2 \beta R_3^{-p} } \| \hU_n\|_{H^1(\hat \T_n)}^2  +  C \beta^2  e^{2 \beta R_1^{-p} } \Big( \| (\hU_n, \nabla \hU_n)\|_{L^2(\hat \Sigma_n)}^2 + \|\hg_n \|_{L^2(\Y_n)}^2 \Big). 
\end{multline*}
This implies 
\begin{multline}\label{LF-p4}
	\int_{B_{R_2}  \cap \hat \T_n}   (|\hU_n|^2 + |\nabla \hU_n|^2 ) 	\le C  \beta^2  e^{2 \beta (R_3^{-p} - R_2^p) } \| \hU_n\|_{H^1(\hat \T_n)}^2  \\[6pt]+  C \beta^2  e^{2 \beta (R_1^{-p} - R_2^p) } \Big( \| (\hU_n, \nabla \hU_n)\|_{L^2(\hat \Sigma_n)}^2 + \|\hg_n \|_{L^2(\Y_n)}^2 \Big).
\end{multline}

Set
$$
a_n = \| (\hU_n, \nabla \hU_n)\|_{L^2(\hat \Sigma_n)}^2 + \|\hg_n \|_{L^2(\Y_n)}, \quad  b_n =   \| \hU_n\|_{H^1(\hat T)}^2 + \| (\hU_n, \nabla \hU_n)\|_{L^2(\hat \Sigma_n)}^2 + \|\hg_n \|_{L^2(\Y_n)}, 
$$
and 
$$
\hat \beta= \big(1 - \rho(n) \big) \ln (b_n / a_n), 
$$
where $\rho(n)$ is given in \eqref{LF-def-ha}.  A straightforward estimate gives, with $\beta = \hat  \beta$,  that 
\begin{multline}\label{LF-p5}
e^{2 \beta (R_3^{-p} - R_2^p) } \| \hU_n\|_{H^1(\hat \T_n)}^2  +    e^{2 \beta (R_1^{-p} - R_2^p) }\Big( \| (\hU_n, \nabla \hU_n)\|_{L^2(\hat \Sigma_n)}^2 + \|\hg_n \|_{L^2(\Y_n)}^2 \Big) \\[6pt]
\le a_n^{2 \rho(n) } b_n^{2(1- \rho(n))}. 
\end{multline}
We claim that 
\begin{equation}\label{LF-step3.4-1}
 \int_{B_{R_2}  \cap \hat \T_n }  (|\hU_n|^2 + |\nabla \hU_n|^2 )  \le C a_n^{2 \alpha }  b_n^{2(1- \alpha)},     
\end{equation}
which is \eqref{LF-step3.4}.  Indeed, if $\hat \beta  \ge \beta_{\Lambda}$,  then take $\beta = \hat \beta$ in \eqref{LF-p4}. We then obtain \eqref{LF-step3.4-1} using  \eqref{LF-p5} and the fact $\rho(n) \ge (1 + \alpha)/2 > \alpha$. If $\hat \beta  <  \beta_{\Lambda}$, inequality \eqref{LF-step3.4-1} also holds for a different constant $C$ by taking $\beta = \beta_\Lambda$  in \eqref{LF-p4}. The proof of Step 4 is complete.

\medskip 
$\bullet$ Step 5.  Let  $x \in \T_n$ be such that  $\rr \in (1/n - \Lambda /n^2, 1/n + \Lambda /n^2)$ and $|\tx - \tZ_0| \le 1/n^2$. We claim that 
\begin{equation}\label{LF-claim-2}
x - Z_0 \in B_{R_2} \cap \hat \T_n. 
\end{equation}

Indeed, for such an $x$,  we have 
\begin{equation}\label{LF-step3.5-p1}
|x_1 - 1/n| \le \Lambda /n^2 \quad \mbox{ and } \quad |x_2| <  (1/n+ \Lambda /n^2)  \pi/(4n).  
\end{equation}
Since
$$
x - Z_0  = (x_1 - 1/n, x_2 - 1/n - \pi \Lambda/n^2, \tx - \tZ_0), 
$$
it follows from \eqref{LF-step3.5-p1} that 
$$
|x - Z_0 | \le  \Lambda/n^2  + 1/n +  \pi \Lambda/n^2 +  \pi (1/n^2 + \Lambda/n^3)/4 + 1/n^2 \le  1/n + 8 \Lambda /n^2 = R_2, 
$$
since $n \ge 10 \Lambda$. Since $x - Z_0 \in \hat \T_n$, claim~\eqref{LF-claim-2} follows. 

\medskip 
Note that   $T_n = L_n \circ H$, $H(\{x \in \mR^d; \rr = 0) \} =  \{x \in \mR^d; \rr = 0) \}$, and $\Lambda^{-1} |x| \le |H(x)| \le \Lambda |x|$. We obtain 
\begin{equation} \label{Step5-p1}
\mbox{if  $\Lambda  \tau_n \le \rr(x) \le \Lambda^{-1} s_n$,  then $1/ n - \Lambda/ n^2 \le \rr (T_n(x))  \le 1/ n + \Lambda/ n^2$}. 
\end{equation}
Recall that 
$$
z_0 = (0, 0, \tz_0) \quad \mbox{ and } \quad Z_0 = T_n (z_0) = (0, 0, \tZ_0) \big( = H(z_0) \big). 
$$
If  $\Lambda  \tau_n \le \rr(x) \le \Lambda^{-1} s_n$  and $|\tx - \tz_0| \le C / n^2$, then $|x-z_0| \le 2C/ n^2$ (for large $n$). This implies, with $y = (y_1, y_2, \ty) = H(x)$, that   $|y-H(z_0)| \le 2C \Lambda/n^2$ and hence $|\ty - \tZ_0| \le 2 C \Lambda/n^2$. We thus obtain that, for $C$ sufficiently small (and large $n$),  
\begin{equation}\label{Step5-p2}
\mbox{if  $\Lambda  \tau_n \le \rr(x) \le \Lambda^{-1} s_n$  and $|\tx - \tz_0| \le C / n^2$, then $|\ty - \tZ_0| \le 1/ n^2$}. 
\end{equation}

Inequality \eqref{LF-step3.5} now follows from \eqref{LF-step3.4}, \eqref{Step5-p1}, and \eqref{Step5-p2}.  The proof of Step 5 is complete. \qed

\subsection{Proof of Theorem~\ref{fund-thm}}  \label{sect-fund-thm}
Extend $\M^\ell$ evenly for $x_1< 0$ and denote
\begin{equation*}
\widetilde Y_{\gamma_1, \gamma_2,  R} 
= \Big\{ x  \in \mR^d;  -3\pi/4 \le \theta \le 3 \pi/4, \,  \gamma_1 R < \rr < \gamma_2 R,  \mbox{ and } |\tx| < R \Big\}.
\end{equation*}
Note that $\hL_{3/2}$ is a diffeomorphism from $\widetilde Y_{\gamma_1, \gamma_2,  R} $ onto $Y_{\gamma_1, \gamma_2,  R}$,  where $\hL_{3/2}$ is defined by
\begin{equation}\label{LF-def-hLn}
\hL_n(x_1, x_2, \tx) = \big( \cos (\theta/n),  \sin (\theta/n), \tx \big) \mbox{ with } n = 3/2. 
\end{equation}
By \Cref{fund-pro}, there exists $\hgamma_2>0$ such that for $\hgamma_1 \in (0, \hgamma_2)$, there exists $\hgamma_0 \in (0, \hgamma_1/2)$ such that,  with $\wtD_{\hgamma_0} = \widetilde Y_{4\hgamma_0, 1, R/2}$,  for $h \in L^2(\wtD_{\hgamma_0})$, for $W \in [H^2(\wtD_{\hgamma_0})]^m$  satisfying 
$$
|\dive(\M^\ell \nabla W_\ell)| \le \Lambda_1(|W| + |\nabla W| + |h|) \mbox{ in } \wtD_{\hgamma_0}, 
$$
then, with $\wtSigma_{\hgamma_0} = \partial \wtD_{\hgamma_0} \cap \{\theta = \pm 3\pi/4\}$, it holds 
\begin{multline}\label{fund-thm-p1}
\| W \|_{H^1(\widetilde Y_{\hgamma_1, \hgamma_2, \frac{R}{4}})} 
 \le C \Big( \|(W, \nabla W)\|_{L^2(\wtSigma_{\hgamma_0})} + \|h\|_{L^2(\wtD_{\hgamma_0})} \Big)^{\alpha}   \times \\[6pt] \times 
\Big(\| W\|_{H^1(\wtD_{\hgamma_0})} + \|(W, \nabla W)\|_{L^2(\wtSigma_{\hgamma_0})} + \| h\|_{L^2(\wtD_{\hgamma_0})}  \Big)^{1-\alpha}.
\end{multline}

Set 
\begin{equation*}
\hat Y_{\gamma_1, \gamma_2,  R} 
= \Big\{ x  \in \mR^d;  \gamma_1 R < \rr < \gamma_2 R,  \mbox{ and } |\tx| < R \Big\}, 
\end{equation*}
and  fix $\varphi \in C^1_{c} (\hat Y_{\hgamma_0, 1,  R})$  such that $\varphi = 1$ for $x \in \hat Y_{2\hgamma_0, 1, R/2}$. 
Let  $U_\ell \in H^1(\hat Y_{\hgamma_0, 1,  R} \setminus \{x_1 = 0\})$ be such that 
$$
\dive (\M^\ell \nabla U_\ell) = 0  \mbox{ in } \hat Y_{\hgamma_0, 1,  R} \setminus \{x_1= 0\}, \quad U_\ell = 0 \mbox{ on } \partial \hat Y_{\hgamma_0, 1,  R}, 
$$
and, on $\hat Y_{\hgamma_0, 1,  R} \setminus \{x_1 = 0\}$, 
$$
[U_\ell] = [\varphi V_\ell \mathds{1}_{x_1 > 0} ]   \quad \mbox{ and } \quad  [\M^\ell \nabla U_\ell \cdot \nu] = [\varphi \M^\ell \nabla (V_\ell \mathds{1}_{x_1 > 0}) \cdot \nu].   
$$
We have
\begin{equation}\label{fund-thm-p2}
\| U \|_{H^1(\hat Y_{\hgamma_0, 1,  R} \setminus \{x_1 = 0\})} \le C \| V\|_{{\bf H} (\Sigma_{\hgamma_0})}, 
\end{equation}
and, by the regularity theory of elliptic equations,  
\begin{equation}\label{fund-thm-p3}
\| U, \nabla U \|_{L^2(\hat Y_{4\hgamma_0,  1, R/2} \cap  \{\theta  = \pm 3 \pi/4\})} \le C \| V\|_{{\bf H} (\Sigma_{\hgamma_0})}. 
\end{equation}

Set,  in $\wtD_{\hgamma_0}$, 
\begin{equation}\label{fund-thm-eq-W}
W_\ell = U_{\ell} \mathds{1}_{x_1>0} - V_{\ell} \quad  \mbox{ and } \quad h = (|g| + |\nabla U| + |U|) \mathds{1}_{x_1 > 0}. 
\end{equation}
Applying \eqref{fund-thm-p1} with $W$ and $h$ given by \eqref{fund-thm-eq-W}, noting that 
$$
\| h\|_{L^2(\wtD_{\hgamma_0})} \le C \Big( \| g\|_{L^2(\wtD_{\hgamma_0} \cap \{x_1>0\})} + \| U \|_{H^1(\wtD_{\hgamma_0} \cap \{x_1 > 0\})} \Big), 
$$
  and using \eqref{fund-thm-p2} and \eqref{fund-thm-p3}, we obtain 
\begin{multline}\label{fund-thm-p4}
\| W \|_{H^1(\widetilde Y_{\hgamma_1, \hgamma_2, \frac{R}{4}})} \\[6pt]
 \le C \Big( \| V\|_{{\bf H}(\Sigma_{\hgamma_0})} + \|g\|_{L^2(D_{\hgamma_0})} \Big)^{\alpha}  
\Big(\| V\|_{H^1(D_{\hgamma_0})} + \| V\|_{{\bf H}(\Sigma_{\hgamma_0})} + \|g\|_{L^2(D_{\hgamma_0})} \Big)^{1-\alpha}. 
\end{multline}
The conclusion now follows from \eqref{fund-thm-p2} and \eqref{fund-thm-p4}. \qed

\subsection{Proofs of \Cref{thm-3SP} and \Cref{cor-3SP}}\label{sect-thm-3SP}

We begin this section with a  variant of \Cref{fund-thm}: 

 \begin{proposition} \label{fund-thm-V}  Let $d \ge 2$, $m \ge 1$, $\Lambda \ge 1$,  and $R > R_* > 0$.  Then, for any  $\alpha \in (0, 1)$, there exists a  constant $r \in (0, R_*)$, depending only on $\alpha$, $\Lambda$, $R_*$,  $m$, and $d$ such that  for real, symmetric, uniformly elliptic, Lipschitz matrix-valued functions  $\M^\ell$ with $1 \le \ell \le m$  defined in $\Omega : = B_R \cap \{x_1 >0\}$  verifying, in $\Omega$, 
\begin{equation*}
\Lambda^{-1} |\xi|^2 \le \langle \M^{\ell} (x) \xi, \xi \rangle \le \Lambda |\xi|^2  \; \;  \forall \,  \xi \in \mR^d  \quad \mbox{ and } \quad   |\nabla \M^{\ell}(x) | \le \Lambda, 
\end{equation*}
for $g \in L^2(\Omega)$, and for  $V  \in [H^1(\Omega)]^m$ satisfying, for  $1 \le \ell \le m$,  
\begin{equation}\label{fund-thm-V-cd1}
|\dive (\M^\ell \nabla V_\ell)| \le  \Lambda_1 \big( |\nabla V| + |V| + |g| \big)  \mbox{ in } \Omega \mbox{ for some } \Lambda_1 \ge 0, 
\end{equation}
we have, with $\Sigma = \partial \Omega \cap \big\{ x_1= 0 \big\}$,  
\begin{equation*}
\| V \|_{H^1(B_r \cap \Omega)}  \le C  \Big( \| V\|_{{\bf H}(\Sigma)}  + 
\| g\|_{L^2(\Omega)}\Big)^{\alpha} \Big(\| V\|_{H^1(\Omega) } + 
\| g\|_{L^2(\Omega)} \Big)^{1-\alpha},  
\end{equation*}
for some positive constant $C$ depending only on $\alpha,  \Lambda, \, \Lambda_1, \, R_*,  \,  m$, and  $d$. 
\end{proposition}

The proof of \Cref{fund-thm-V} is in the same spirit as the one of \Cref{fund-thm} but much simpler; one does not need to make any change of variables in the spirit of conformal maps as in the proof of \Cref{fund-pro}. Similar arguments are also used in  \cite{MinhLoc2}. The proof is omitted. 

\medskip 
We are ready to give

\begin{proof}[Proof of \Cref{thm-3SP}] The proof is based on \Cref{cor-3SP-*}, \Cref{fund-thm-V}, and a covering argument.  By \Cref{cor-3SP-*},  there exists $\gamma_2 > 0$ depending only on $\alpha$, $\Gamma$, $\Lambda$, $R_1$ and $R_3$ such that for every $\gamma_1 \in (0, \gamma_2)$ there exists $\gamma_0 \in (0, \gamma_1)$, 
depending only on $\gamma_1$, $\alpha$, $\Gamma$, $\Lambda$, $R_1$ and $R_3$, such that  
\begin{equation}\label{thm-3SP-p1}
\| V\|_{H^1 \big((O_{\gamma_2} \setminus O_{\gamma_1} ) \setminus B_{R_1} \big)} \le C_{\gamma_1} \Big( \| V \|_{{\bf H}(\Sigma_{\gamma_0})} + \| g\|_{L^2(D_{\gamma_0})} \Big)^\alpha \Big( \| V\|_{H^1(D_{\gamma_0})} + \| g\|_{L^2(D_{\gamma_0})} \Big)^{1 -\alpha}, 
\end{equation}
for some positive constant $C_{\gamma_1}$ depending only on  $\gamma_1$, $\alpha$, $\Gamma$, $\Lambda$, $\Lambda_1$, $R_1$ and $R_3$.

Fix such a $\gamma_2$. By \Cref{fund-thm-V}, for $x \in \partial B_{R_1} \setminus O_{\gamma_2/3}$, there exists $\rho(x) \in (0,  \gamma_2/12)$ such that 
\begin{equation}\label{thm-3SP-p2}
\|V \|_{H^1(B_{\rho(x)} (x))} \le C \Big(\| V \|_{{\bf H}(\Sigma_{\gamma_2/4})} + \| g\|_{L^2(D_{\gamma_2/4})} \Big)^\alpha  \Big(\| V\|_{H^1(D_{\gamma_2/4})} + \| g\|_{L^2(D_{\gamma_2/4})} \Big)^{1 -\alpha}. 
\end{equation}
One can  also choose $\rho(x)$ such that it depends only on $\alpha$, $\Lambda$, $\gamma_2$, $\Gamma$, $d$, and $m$ (so it is independent of $x$). This will be assumed from now on and we will simply denote it by $\rho$ for notational ease.  
Since 
$$
\partial B_{R_1} \setminus O_{\gamma_2/2} \subset \mathop{\bigcup}_{x \in \partial B_{R_1} \setminus O_{\gamma_2/3}} B_{\rho/2}(x),
$$ 
it follows that there exists a finite set $\Big\{x_i \in \partial B_{R_1} \setminus O_{\gamma_2/3}; i \in I \Big\}$ such that 
\begin{equation*}
\partial B_{R_1} \setminus O_{\gamma_2/2} \subset \mathop{\bigcup}_{i \in I} B_{\rho /2}(x_i). 
\end{equation*}
Then 
\begin{equation}\label{thm-3SP-p3}
(B_{R_1 + \rho/2} \setminus B_{R_1}) \setminus O_{\gamma_2/2} \subset \mathop{\bigcup}_{i \in I} B_{\rho }(x_i). 
\end{equation}
We derive from \eqref{thm-3SP-p2} and \eqref{thm-3SP-p3} that 
\begin{multline}\label{thm-3SP-p4}
\| V\|_{H^1\big( (B_{R_1 + \rho/2} \setminus B_{R_1}) \setminus O_{\gamma_2/2} \big)}\\[6pt] \le C \Big( \| V \|_{{\bf H}(\Sigma_{\gamma_2/4})} + \| g\|_{L^2(D_{\gamma_2/4})} \Big)^\alpha \Big( \| V\|_{H^1(D_{\gamma_2/4})} + \| g\|_{L^2(D_{\gamma_2/4})} \Big)^{1 -\alpha}. 
\end{multline}

Set 
$$
r_2 = \rho/2.  
$$
For $r_1 \in (0, r_2)$,  let $r_0 = \gamma_0$ where $\gamma_0$ is the constant corresponding to $\gamma_1 = r_1$  in \eqref{thm-3SP-p1}. Note that $r_1 < \gamma_2/24$ since $\rho \in (0,  \gamma_2/12)$. We have, by the definition of $r_2$,  
\begin{multline*}
B_{R_1 + r_2} \setminus B_{R_1 +  r_1} \subset
 \Big( (B_{R_1 + \rho/2} \setminus B_{R_1}) \setminus O_{\gamma_2/2} \Big) \cup \Big( O_{\gamma_2/2} \setminus B_{R_1 +  r_1}\Big) \\[6pt]
\subset \Big( (B_{R_1 + \rho/2} \setminus B_{R_1}) \setminus O_{\gamma_2/2} \Big) \cup \Big( \big(O_{\gamma_2} \setminus O_{r_1}\big) \setminus B_{R_1} \Big).
\end{multline*}
We derive from \eqref{thm-3SP-p1} and \eqref{thm-3SP-p4} that 
\begin{equation*}
\| V\|_{H^1(B_{R_1 +  r_2} \setminus B_{R_1 +  r_1})} \le C \Big(\| V \|_{{\bf H}(\Sigma_{r_0})} + \| g\|_{L^2(D_{r_0})} \Big)^\alpha  \Big(\| V\|_{H^1(D_{r_0})} + \| g\|_{L^2(D_{r_0})} \Big)^{1 -\alpha}. 
\end{equation*}
 The proof is complete. 
\end{proof}

We next give the 

\begin{proof}[Proof of  \Cref{cor-3SP}]
Fix  $s \in (0, 1 )$ and $\hat R_1 \in (R_1, R_2)$ be such that 
\begin{equation}\label{cor-3SP-p1}
s  \ln (R_3/ R_2) \Big/ \ln (R_3/ \hat R_1) >  \alpha.  
\end{equation}
Such $s$ and  $\hat R_1$ exist since 
$$
\alpha < \alpha_0 = \ln (R_3/ R_2) \Big/ \ln (R_3/ R_1) 
$$
(e.g. one can take $\hat R_1$ close to $R_1$ and $s$ close to 1). 

By  Theorem~\ref{thm-3SP}, there exist 
$r_* \in (R_1, \hat R_1)$ and $r_0 \in (0, r_* - R_1)$ such that if $\Delta v + \omega^2 v = 0$ in $D_{r_0}$, then  
\begin{equation}\label{pro-E-p1}  
\| V \|_{\bH(\partial B_{r_*})} \le C \| V\|_{\bH(\Sigma_{r_0})}^s  \|V\|_{H^1(D_{r_0})}^{1 - s}. 
\end{equation}
On the other hand, we have, by \eqref{3SP-H},  
\begin{equation}\label{pro-E-p2}
\| V \|_{\bH(\partial B_{R_2})} \le C \| V \|_{\bH(\partial B_{r_*})}^{\beta} \| V \|_{\bH(\partial B_{R_3})}^{1 - \beta}
\end{equation}
with $\beta = \ln (R_3/ R_2) \Big/ \ln (R_3/ r_*)$.  Combining \eqref{pro-E-p1} and \eqref{pro-E-p2} yields 
\begin{align*}
\| V \|_{\bH(\partial B_{R_2})} \le  C  \| V\|_{\bH(\Sigma_{r_0})}^{\beta s} \|V\|_{H^1(D_{r_0})}^{(1 - s) \beta}  \| V \|_{\bH(\partial B_{R_3})}^{1 - \beta}  
\le C   \| V\|_{\bH(\Sigma_{r_0})}^{\beta s}  \|V\|_{H^1(D_{r_0})}^{1 -  \beta s}. 
\end{align*}
In the last inequality, we used the fact $\| V\|_{{\bf H}(\partial B_{R_3})} \le C \| V\|_{H^1(D_{r_0})}$ by the trace theory.  
The conclusion  now follows  since  $\beta s >  \alpha$  by \eqref{cor-3SP-p1} and $\| V\|_{\bH(\Sigma_{r_0})} \le  C \|V\|_{H^1(D_{r_0})}$ by the trace theory. 
\end{proof}

\section{Three-sphere inequalities for Maxwell equations}\label{sect-3SP-M}


In this section, we establish three-sphere inequalities for Maxwell equations. As usual, see e.g.  in \cite{Leis, Ng-CALR-M, Tu} and the references therein, we also derive \Cref{thm-3SP-M} from three-sphere inequalities for second-order elliptic equations with partial data. 
In order to be able to apply the results established in \Cref{sect-3SP}, we will use the fact that Maxwell equations can be reduced to weakly coupled second order elliptic equations. More precisely, let $\Omega$ be an open subset of $\mR^3$. If $(E, H) \in [H^1(\Omega)]^6$ satisfies 
\begin{equation*}
\left\{\begin{array}{cl}
\nabla \times E = i  \omega \mu H & \mbox{ in } \Omega, \\[6pt]
\nabla \times H = - i \omega \eps E & \mbox{ in } \Omega, 
\end{array}\right.
\end{equation*}
then, for $1 \le a \le 3$,  
\begin{equation}\label{eq-H}
\dive (\mu \nabla \cH_a)  + \dive (\partial_a \mu \cH - i k \mu \epsilon^a \eps \cE) = 0 \mbox{ in } \Omega, 
\end{equation}
\begin{equation}\label{eq-E}
\dive(\eps \nabla \cE_a) + \dive (\partial_a \eps \cE + i k \eps \epsilon^a \mu \cH) = 0 \mbox{ in } \Omega. 
\end{equation}
Here $\cE_a$ and $\cH_a$  denote the $a$ component of $\cE$ and $\cH$,  respectively, and the $bc$ component $\epsilon^a_{bc}$ $(1 \le b, c \le 3)$  of $\epsilon^a$ $(1 \le a \le 3)$ denotes the usual Levi Civita permutation, i.e., 
\begin{equation*}
\epsilon^a_{bc} = \left\{\begin{array}{cl} \mbox{sign } (abc) & \mbox{ if  $abc$  is a permuation}, \\[6pt]
0 & \mbox{otherwise}. 
\end{array}\right.
\end{equation*}

We now present a variant of \Cref{fund-thm} for the Maxwell equations.

\begin{theorem} \label{fund-thm-M} Let $d=3$, $\Lambda \ge 1$,  and   $0< R_*< R < R^*$. Then, for any  $ \alpha \in (0,  1)$, there exists a positive constant $ \hgamma_2 \in (0, 1)$, depending only on $\Lambda$ and $R$ such that for every $\hgamma_1 \in (0, \hgamma_2)$, there exists  $\hgamma_0 \in (0,  \hgamma_1)$ depending only on  $\hgamma_1$, $\alpha$,  and  $\Lambda$,  such that, for 
a pair of symmetric, uniformly elliptic, matrix-valued functions  $(\eps, \mu)$ of class $C^2$ defined in $D_{\hgamma_0} : = Y_{\hgamma_0, 1,  R}$ verifying, in $D_{\hgamma_0} $, with $M = \eps$ and  $M =  \mu$, 
\begin{equation}
\Lambda^{-1} |\xi|^2 \le \langle M(x) \xi, \xi \rangle \le \Lambda |\xi|^2  \; \;  \forall \, \xi \in \mR^d  \quad \mbox{ and } \quad   |\nabla M(x) |  \le \Lambda, \end{equation}
then, for $\omega > 0$, for $J_e, \, J_m \in [L^2(D_{\hgamma_0} )]^3$,  and  for  $(E, H) \in [H(\curl, D_{\hgamma_0} )]^2$ satisfying 
\begin{equation*} 
\left\{\begin{array}{cl}
\nabla \times E = i \omega \mu H + J_e & \mbox{ in } D_{\hgamma_0} , \\[6pt]
\nabla \times H = - i \omega \eps H + J_m & \mbox{ in } D_{\hgamma_0} , 
\end{array}\right. 
\end{equation*}
we have, with $\Sigma_{\hgamma_0} = \partial D_{\hgamma_0}   \cap \{x_1 = 0 \}
$, 
\begin{multline}
\| (E, H) \|_{L^2(Y_{\hgamma_1, \hgamma_2,  \frac{R}{4}})} 
\le C \Big( \| (E \times \nu, H \times \nu)\|_{H^{-1/2}(\dive_\Sigma, \Sigma_{\hgamma_0})} +  \| (J_e, J_m) \|_{L^2(D_{\hgamma_0} )} \Big) ^{\alpha}   \times \\[6pt]
\times \Big(\| (E, H) \|_{L^2(D_{\hgamma_0} ) }+  \| (J_e, J_m) \|_{L^2(D_{\hgamma_0} )} \Big)^{1 - \alpha},  
\end{multline}
for some positive constant $C$ depending only on $\alpha, \,  \gamma_1, \,   \omega$, $R_*$, $R^*$, and the upper bound of \\
$\| (\eps, \mu) \|_{C^2(\bar D_{\hgamma_0})}$. 
\end{theorem}

\begin{remark} \rm The constant $\gamma_2$ depends on $\Lambda$ but is independent of the upper bound of \\ $\| (\eps, \mu) \|_{C^2(\bar D_{\hgamma_0})}$. 
\end{remark}

The proof of \Cref{fund-thm-M} is given in \Cref{sect-fund-thm-M} below and is the key part  of the proof of Theorem~\ref{thm-3SP-M}.  As a consequence of \Cref{fund-thm-M}, we obtain the following variant of \Cref{cor-3SP-*} for the Maxwell equations whose proof is omitted.  

\begin{corollary}\label{cor-Maxwell-Cl} Let $d = 3$, $\Lambda \ge  1$, let $\Omega$ be an open subset of $\mR^d$ of class $C^1$, and let 
$\Gamma$ be a  compact smooth curve of  $\partial \Omega$ which belongs to a connected component $\Sigma$ of $\partial \Omega$.  Denote  $O_r = \Big\{x \in \mR^d; \dist(x, \Gamma) < r \Big\}$, $D_r = \Omega \setminus \bar O_r$, and $\Sigma_r = \Sigma \setminus \bar O_r$ for $r>0$.
For any $\alpha \in (0,  \alpha_0)$, there exists $r_2>0$ depending only on $\alpha$, $\Gamma$,  and $\Omega$ such that for every $r_1 \in (0, r_2)$, there exists $r_0 \in (0, r_1)$, depending only on $r_1$, $\alpha$, $\Gamma$,  and $\Omega$,  such that  for 
a pair $(\eps, \mu)$ of symmetric, uniformly elliptic, Lipschitz matrix-valued functions   defined in $D_{r_0}$ verifying, with $M = \eps$ and  $M =  \mu$, 
\begin{equation}
\Lambda^{-1} |\xi|^2 \le \langle M(x) \xi, \xi \rangle \le \Lambda |\xi|^2  \; \;  \forall \, \xi \in \mR^d  \quad \mbox{ and } \quad   |\nabla M(x) |  \le \Lambda, 
\end{equation}
for $\omega > 0$, for $J_e, \, J_m \in [L^2(D_{r_0})]^3$,  and  for  $(E, H) \in [H(\curl, D_{r_0})]^2$ satisfying
\begin{equation*} 
\left\{\begin{array}{cl}
\nabla \times E = i \omega \mu H + J_e & \mbox{ in } D_{r_0}, \\[6pt]
\nabla \times H = - i \omega \eps H + J_m & \mbox{ in } D_{r_0}, 
\end{array}\right. 
\end{equation*}
we have
\begin{multline}
\| (E,  H) \|_{L^2(O_{r_2} \setminus O_{r_1})}  
\le C \Big( \| (E \times \nu, H \times \nu)\|_{H^{-1/2}(\dive_\Sigma, \Sigma_0)} +  \|(J_e, J_m) \|_{L^2(D_{r_0})} \Big)^{\alpha} \times \\[6pt] \times \Big( \| (E, H) \|_{L^2(D_{r_0}) } + \| (J_e, J_m) \|_{L^2(D_{r_0})} \Big)^{1-\alpha},  
\end{multline}
for some positive constant $C$ depending only on $\alpha, \,    \omega, \,  \Lambda$, $\Gamma$, $\Omega$, and  the upper bound of $\| (\eps, \mu) \|_{C^2(\bar D_{r_0})}$.  
\end{corollary}

Applying Theorem~\ref{thm-3SP-M}, one can derive various  three-sphere inequalities with partial data for $R_1 < R_2 < R_3$. Here is an example in the spirit of Hadamard. 
\begin{corollary}\label{cor-Maxwell} Let $d = 3$, and $R_1< R_2 < R_3$, and let 
$\Gamma$ be a  a compact smooth curve of  $\partial B_{R_1}$.  Denote  $O_r = \Big\{x \in \mR^d; \dist(x, \Gamma) < r \Big\}$, $D_r = B_{R_3} \setminus (\overline{B_{R_1} \cup O_r})$, and $\Sigma_r = \partial B_{R_1} \setminus \bar O_r$ for $r>0$.
Set $\alpha_0 = \ln(R_3/ R_2) \Big/ \ln (R_3/ R_1)$. Then, for any $ \alpha \in (0,  \alpha_0)$, there exists $r_0  \in (0, R_2 - R_1)$, depending only on $R_1$, $R_2$,  $R_3$, $\Gamma$, and $\alpha$ such that for $\omega>0$ and 
for $(E, H) \in [H(\curl, D_{r_0})]^2$ satisfying   
\begin{equation*} 
\left\{\begin{array}{cl}
\nabla \times E = i \omega H & \mbox{ in } D_{r_0}, \\[6pt]
\nabla \times H = - i \omega H & \mbox{ in } D_{r_0}, 
\end{array}\right. 
\end{equation*}
we have
\begin{equation*}  
\| (E \times \nu, H \times \nu) \|_{H^{-1/2}(\dive_\Sigma, \partial B_{R_2})}
\le C \| (E \times \nu, H \times \nu)\|_{H^{-1/2}(\dive, \Sigma_{r_0})}^\alpha   \|(E, H)\|_{L^2(D_{r_0})}^{1 - \alpha}, 
\end{equation*}
for some positive constant $C$ depending on  $\alpha$, $\omega$, $\Gamma$, $R_1$, $R_2$,  and $R_3$. 
\end{corollary}

The rest of this section is organized as follows. We first present two lemmas used in  the proofs of \Cref{fund-thm-M}, \Cref{thm-3SP-M}, and \Cref{cor-Maxwell}. The proofs of \Cref{fund-thm-M}, \Cref{thm-3SP-M}, and \Cref{cor-Maxwell} are then given in \Cref{sect-fund-thm-M}, \Cref{sect-thm-3SP-M}, and \Cref{sect-cor-Maxwell}, respectively.

\subsection{Two useful lemmas} We begin with 

\begin{lemma} \label{lem-EH} Let $\Lambda \ge 1$, $D \Subset \Omega \subset \mR^3$ be two connected, open, bounded subsets of $\mR^3$, and let $(\eps, \mu)$ be a pair of real, symmetric, Lipschitz, uniformly elliptic matrix-valued functions defined in $\Omega$ such that, with $M = \eps$ and  $M =  \mu$, in $\Omega$, 
\begin{equation}\label{lem-EH-proM}
\Lambda^{-1} |\xi|^2 \le \langle M(x) \xi, \xi \rangle \le \Lambda |\xi|^2  \; \;  \forall \, \xi \in \mR^d  \quad \mbox{ and } \quad   |\nabla M(x) |  \le \Lambda. 
\end{equation}
Let  $J_e, J_m \in [L^2(\Omega)]^3$,  and let  $f, g \in H^{-1/2} (\dive_\Gamma, \partial D)$. There exists a unique solution 
$(E, H) \in [H(\curl, \Omega \setminus \partial D)]^2$  of the system  
\begin{equation*}
\left\{\begin{array}{cl}
\nabla \times E = i \omega \mu H + J_e & \mbox{ in } \Omega, \\[6pt]
\nabla \times H = - i \omega \eps E + J_m & \mbox{ in } \Omega, \\[6pt]
[E \times \nu] = f, \quad [H \times \nu] = g &  \mbox{ on } \partial D, \\[6pt]
(H \times \nu) \times \nu - E \times \nu = 0 & \mbox{ on } \partial \Omega, 
\end{array}\right. 
\end{equation*}
Moreover, 
\begin{equation}\label{lem-EH-cl1}
\| (E, H) \|_{L^2(\Omega)} \le C  \Big( \| (f, g) \|_{H^{-1/2}(\dive_\Gamma, \partial D)} + \| (J_e, J_m)\|_{L^2(\Omega)} \Big), 
\end{equation}
for some positive constant $C$ depending only on $\Lambda$, $\Omega$, and $D$. 
\end{lemma}

The novelty of    \Cref{lem-EH}  lies in the fact that the constant $C$ does not depend on $(\eps, \mu)$ but on $\Lambda$. Nevertheless, one can derive the results by a quite standard contradiction argument as in \cite{Ng-Superlensing-Maxwell, Ng-Negative-Cloaking-M}. For the completeness, the proof is given in \Cref{ap-lem-EH}.

\medskip

The second lemma,  whose proof is given in \Cref{ap-lem-M},  is 

\begin{lemma}\label{lem-M} Let $\omega > 0$, $0< R_* < R_1< R_2 < R_3< R^*$, and let  $(E, H) \in [H(\curl, B_{R_3} \setminus B_{R_1})]^2$ be a solution of the system 
\begin{equation*} 
\left\{\begin{array}{cl}
\nabla \times E = i \omega H & \mbox{ in } B_{R_3} \setminus B_{R_1}, \\[6pt]
\nabla \times H = - i \omega H & \mbox{ in } B_{R_3} \setminus B_{R_1}. 
\end{array}\right. 
\end{equation*}
Then 
\begin{multline*}
\|(E \times \nu, H \times \nu) \|_{H^{-1/2}(\dive, \partial B_{R_2})} \\[6pt]
\le C \| (E \times \nu, H \times \nu)\|_{H^{-1/2}(\dive, \partial B_{R_1})}^{\alpha} \| (E \times \nu, H \times \nu) \|_{H^{-1/2}(\dive, \partial B_{R_3})}^{1 - \alpha},
\end{multline*}
with $\alpha = \ln (R_3/ R_2) \Big/ \ln (R_3/ R_1)$ for some positive constant $C$ depending only on $R_*, \, R^*$, and $\omega$.  
\end{lemma}

\subsection{Proof of \Cref{fund-thm-M}} \label{sect-fund-thm-M} The proof of \Cref{fund-thm-M}, which involves \Cref{lem-EH},  is in the same spirit as the one of \Cref{fund-thm}. Extend $\eps$ and $\mu$ evenly for $x_1< 0$ and still denote the extensions by $\eps$ and $\mu$. Set 
\begin{equation*}
\widetilde Y_{\gamma_1, \gamma_2,  R} 
= \Big\{ x  \in \mR^d;   \theta \in (-3\pi/4, 3 \pi/4),  \,  \gamma_1 R < \rr < \gamma_2 R,  \mbox{ and } |\tx| < R \Big\}. 
\end{equation*}
Note that $\hL_{3/2}$ is a diffeomorphism from  $\widetilde Y_{\gamma_1, \gamma_2,  R} $ onto $Y_{\gamma_1, \gamma_2,  R}$,  where $L_{3/2}$ is given by \eqref{LF-def-hLn}.  By \Cref{fund-pro}, and \eqref{eq-H}  and \eqref{eq-E}, there exists $\hgamma_2>0$ depending only on $\alpha$, $\Lambda$, $\Gamma$, $R_1$, and $R_3$,  such that for every $\hgamma_1 \in (0, \hgamma_2)$, there exists $\hgamma_0 \in (0, \gamma_1/2)$, depending only on $\gamma_1$, $\alpha$, $\Lambda$, $\Gamma$, $R_1$, and $R_3$,  such that,  with $\wtD_{\hgamma_0} = \widetilde Y_{4\hgamma_0, 1, R/2}$,  for $(\wtE, \wtH) \in [H^1(\wtD_{\hgamma_0})]^6$  satisfying  
\begin{equation*}
\left\{\begin{array}{cl}
\nabla \times \wtE = i \omega \mu \wtH  & \mbox{ in } \wtD_{\hgamma_0}, \\[6pt]
\nabla \times \wtH = - i \omega \eps \wtE & \mbox{ in } \wtD_{\hgamma_0},
\end{array}\right. 
\end{equation*}
we have,  with $\wtSigma_{\hgamma_0} = \partial \wtD_{\hgamma_0} \cap \{\theta = \pm 3\pi/4\}$,  
\begin{multline}\label{fund-thm-M-p1}
\|(\wtE, \wtH) \|_{H^1(\widetilde Y_{\hgamma_1, \hgamma_2, \frac{R}{4}})} \\[6pt]
 \le C \|(\wtE, \wtH,  \nabla \wtE, \nabla \wtH)\|_{L^2(\wtSigma_{\hgamma_0})}  \Big(\| (\wtE, \wtH)\|_{H^1(\wtD_{\hgamma_0})} + \|(\wtE, \wtH,  \nabla \wtE, \nabla \wtH)\|_{L^2(\wtSigma_{\hgamma_0})}  \Big)^{1-\alpha}.
\end{multline}

Set 
\begin{equation*}
\hat Y_{\gamma_1, \gamma_2,  R} 
= \Big\{ x  \in \mR^d;  \gamma_1 R < \rr < \gamma_2 R,  \mbox{ and } |\tx| < R \Big\}, 
\end{equation*}
and let  $(\hE, \hH) \in [H(\curl, \hat Y_{\hgamma_0, 1,  R} \setminus \{x_1 = 0\})]^2$ be such that 
\begin{equation*}
\left\{\begin{array}{cl}
\nabla \times \hE = i \omega \mu H + J_e \mathds{1}_{x_1 > 0} & \mbox{ in } \hat Y_{\hgamma_0, 1,  R} \setminus \{x_1 = 0\}, \\[6pt]
\nabla \times \hH = - i \omega \eps \hE + J_m \mathds{1}_{x_1 > 0} & \mbox{ in } \hat Y_{\hgamma_0, 1,  R} \setminus \{x_1 = 0\}, \\[6pt]
[\hE \times \nu] = [\varphi H \mathds{1}_{x_1 >0} \times \nu], \quad [\hH \times \nu] = [\varphi H \mathds{1}_{x_1 > 0} \times \nu] &  \mbox{ on } \hat Y_{\gamma_0, 1,  R} \setminus \{x_1 = 0\}, \\[6pt]
(\hH \times \nu) \times \nu - \hE \times \nu = 0 & \mbox{ on } \partial \hat Y_{\gamma_0, 1,  R}, 
\end{array}\right. 
\end{equation*}
where $\varphi \in C^1_{c} (\hat Y_{\hgamma_0, 1,  R})$ is fixed  such that $\varphi = 1$ for $x \in \hat Y_{4\hgamma_0, 1, R/2}$. By \Cref{lem-EH}, we have
\begin{equation}\label{fund-thm-M-p2}
\| (\hE, \hH) \|_{L^2(\hat Y_{\hgamma_0, 1,  R})} \le C \Big( \| (E \times \nu, H \times \nu)\|_{H^{-1/2} (\dive_\Sigma, \Sigma_{\hgamma_0})} + \| (J_e, J_m)\|_{L^2(D_{\hgamma_0})} \Big).
\end{equation}
This in turn implies,  by the regularity theory of elliptic equations, and \eqref{eq-H} and \eqref{eq-E},  that 
\begin{multline}\label{fund-thm-M-p3}
\| (\hE, \hH, \nabla \hE, \nabla \hH) \|_{L^2(\hat Y_{4\hgamma_0, 1,  R/2} \cap  \{\theta  = \pm 3 \pi/4\})} \\[6pt]
 \le C \Big( \| (E \times \nu, H \times \nu)\|_{H^{-1/2} (\dive_\Sigma, \Sigma_{\hgamma_0})} + \| (J_e, J_m)\|_{L^2(D_{\hgamma_0})} \Big). 
\end{multline}

Set 
\begin{equation}\label{def-wtEH}
(\wtE, \wtH) = (E, H) \mathds{1}_{x_1>0} - (\hE, \hH) \quad  \mbox{ in } \wtD_{\hgamma_0}.  
\end{equation}
Applying \eqref{fund-thm-M-p1} to $(\wtE, \wtH)$ given in \eqref{def-wtEH},  and using \eqref{fund-thm-M-p2} and \eqref{fund-thm-M-p3}, one reaches  the conclusion.  \qed

\subsection{Proof of \Cref{thm-3SP-M}} \label{sect-thm-3SP-M} The proof of \Cref{thm-3SP-M} is similar to the one of \Cref{thm-3SP}. However, instead of using \Cref{fund-thm} and \Cref{fund-thm-V}, one applies \Cref{fund-thm-M} and \Cref{fund-thm-M-V} below. The details are left to the reader. \qed

\medskip 
In the proof of \Cref{thm-3SP-M}, we also use  the following variant of \Cref{fund-thm-M}:
 
\begin{proposition} \label{fund-thm-M-V} Let $R> R_*>0$ and $\Lambda \ge 1$, and   set $\Omega = B_R \cap \{x_1 >0 \}$. 
For any  $0 < \alpha < 1$, there exists a positive constant $ r \in (0,  R)$, depending only on $\Lambda$ and $R_*$ such that for  a pair $(\eps, \mu)$ of symmetric, uniformly elliptic, Lipschitz matrix-valued functions   defined in $\Omega$ verifying, for $M = \eps$ and $M =  \mu$, 
\begin{equation*}
\Lambda^{-1} |\xi|^2 \le \langle M(x) \xi, \xi \rangle \le \Lambda |\xi|^2 \mbox{ for all } \xi \in \mR^d  \quad \mbox{ and } \quad   |\nabla M(x) | |x| \le \Lambda,  
\end{equation*}
for $\omega >0$, and  for  $(E, H) \in [H(\curl, \Omega)]^2$ satisfying
\begin{equation*} 
\left\{\begin{array}{cl}
\nabla \times E = i \omega \mu H + J_e& \mbox{ in } \Omega, \\[6pt]
\nabla \times H = - i \omega \eps E + J_m & \mbox{ in } \Omega, 
\end{array}\right. 
\end{equation*}
we have, with $\Sigma = \partial \Omega \cap \{x_1 = 0 \}$, 
\begin{multline*}
\| (E, H) \|_{L^2(B_r \cap \Omega)}  
\le C \Big(\| (E \times \nu, H \times \nu)\|_{H^{-1/2}(\dive_\Sigma, \Sigma)} +  \| (J_e, J_m)\|_{L^2(\Omega)} \Big)^{\alpha} \times 
\\[6pt] \times \Big(\| (E, H) \|_{L^2(\Omega) } +   \| (J_e, J_m)\|_{L^2(\Omega)} \Big)^{1-\alpha},  
\end{multline*}
for some positive constant $C$ depending only on $\alpha,  \,   \omega, \,  \Lambda$, $R_*$,  and the upper bound of $\| (\eps, \mu) \|_{C^2(\bar \Omega)}$. 
\end{proposition}

The proof of \Cref{fund-thm-M-V} is in the same spirit as the one of  \Cref{fund-thm-V} (see also the proof of of \Cref{fund-thm-M}) and is  omitted. 

\subsection{Proof of \Cref{cor-Maxwell}} \label{sect-cor-Maxwell}
Let $\hat R_1 \in (R_1,  R_3)$ and  $s \in (0, 1)$ be  such that 
$$
\alpha < \beta s < \alpha_0,  
$$
where $\beta = \ln (R_3/ R_2) \Big/ \ln (R_3/ \hat R_1)$. By Theorem~\ref{thm-3SP-M}, there exist $r_* \in (R_1, \hat R_1)$ and   $r_0 \in (0, r_* - R_1)$ such that
\begin{equation}\label{pro-M-p1}  
\| (E \times \nu, H \times \nu) \|_{H^{-1/2}(\dive_\Gamma, \partial B_{\hat R_1})} 
\le C \| (E \times \nu, H \times \nu)\|_{H^{-1/2}(\dive_\Gamma, \Sigma_{r_0})}^s  \|(E, H)\|_{L^2(D_{r_0})}^{1 - s}. 
\end{equation}
By \Cref{lem-M}, we have 
\begin{multline}\label{pro-M-p2}
\| (E \times \nu, H \times \nu) \|_{H^{-1/2}(\dive_\Gamma, \partial B_{R_2})} \\[6pt]
\le C \| (E \times \nu, H \times \nu)\|_{H^{-1/2}(\dive_\Gamma, \partial B_{\hat R_1})}^{\beta} \| (E \times \nu, H \times \nu) \|_{H^{-1/2}(\dive_\Gamma, \partial B_{R_3})}^{1 - \beta}. 
\end{multline}
Combining \eqref{pro-M-p1} and \eqref{pro-M-p2} yields 
\begin{equation*}
\|   (E \times \nu, H \times \nu) \|_{H^{-1/2}(\dive_\Gamma, \partial B_{R_2})} \le  C   \| V\|_{H^{-1/2}(\dive_\Gamma, \Sigma_{r_0})}^{\beta s} \|(E, H)\|_{L^2(D_{r_0})}^{1 -  \beta s}. 
\end{equation*}
The conclusion follows since $\beta s > \alpha$. \qed

\section{Cloaking property of plasmonic structures in  doubly complementary media} \label{sect-CALR}

This section is devoted to the cloaking property of plasmonic structures in  doubly complementary media. This cloaking phenomenon is also known as cloaking an object via anomalous localized resonance.  Let $\omega > 0$, and let $\Omega_1 \Subset  \Omega_2 \Subset  \mR^3$ be smooth,  bounded,  simply connected,  open subsets of $\mR^3$ \footnote{In this paper, the notation $D \Subset \Omega$ means $\bar D \subset \Omega$ for two subsets $D$ and $\Omega$ of $\mR^d$ ($d \ge 2$).}.  
Let $\ep, \mup$ be defined in $\mR^3 \setminus (\Omega_2 \setminus \Omega_1)$ and $\en, \mun$ be defined in $\Omega_2 \setminus \Omega_1$ such that 
$\ep, \mup$, $- \en$, and $- \mun$  are real,  symmetric,  {\it uniformly elliptic},  matrix-valued functions in their domains of definition.  Set, for $\delta \ge 0$,  
\begin{equation}\label{def-eDelta}
(\eps_\delta, \mu_\delta)  = \left\{\begin{array}{cl}
\en + i \delta I,  \mun + i \delta I  & \mbox{ in } \Omega_2 \setminus \Omega_1, \\[6pt]
\ep, \mup & \mbox{ in } \mR^3 \setminus (\Omega_2 \setminus \Omega_1). 
\end{array} \right. 
\end{equation}
As usual, we assume  that for some $R_0 > 0$,  $\Omega_2 \subset B_{R_0}$, $(\ep, \mup) = (I, I)$ in $\mR^3 \setminus B_{R_0}$. Here and in what follows, all matrix-valued functions are assumed to be  piecewise  $C^1$ in their domain of definition.  Given $\delta > 0$ and $J \in [L^2(\mR^3)]^3$ with compact support, let $(E_\delta, H_\delta) \in [H_{\loc}(\curl, \mR^d)]^2$ be  the unique radiating solution of the Maxwell equations 
\begin{equation}\label{Main-eq-delta}
\left\{\begin{array}{cl}
\nabla \times E_\delta = i \omega \mu_\delta H_\delta &  \mbox{ in } \mR^3, \\[6pt]
\nabla \times H_\delta = - i \omega \eps_\delta E_\delta + J &  \mbox{ in } \mR^3.  
\end{array} \right.
\end{equation}

Physically, $\eps_\delta$ and $\mu_\delta$ describe the permittivity and the permeability of the considered medium, $\Omega_2 \setminus \Omega_1$ is   a (shell) plasmonic structure  in which the permittivity and the permeability are negative and $i \delta I$ describes its loss, $\omega$ is the frequency, $J$ is a density of charge, and  $(E_\delta, H_\delta)$ is the electromagnetic field generated by $J$ in the medium $(\eps_\delta, \mu_\delta)$.  We assume here that the loss is $i \delta I$ for the simplicity of notation;  any quantity of the form $i \delta M$,  where $M$ is a real, symmetric, uniformly elliptic,  matrix-valued function defined in $\Omega_2 \setminus \Omega_1$,  is admissible. 

Recall that a solution $(E, H) \in [H_{\loc}(\curl, \R^3\setminus B_R)]^2$, for some $R> 0$, of the Maxwell equations 
\[
\begin{cases}
\nabla \times E = i \omega H  &\text{ in } \mathbb{R}^3\setminus B_R,\\[6pt]
\nabla \times H = -i \omega E  &\text{ in } \mathbb{R}^3\setminus B_R,
\end{cases}
\]
is called radiating if it satisfies one of the (Silver-M\"{u}ller) radiation conditions
\begin{equation*}
H \times x - |x| E = O(1/|x|) \quad   \mbox{ or } \quad  E\times x + |x| H = O(1/|x|) \qquad \mbox{ as } |x| \to + \infty. 
\end{equation*}
For a matrix-valued function  $A$  defined in $\Omega$, for a bi-Lipschitz homeomorphism ${\mathcal T}: \Omega \to \Omega'$,  and for a vector field $j$ defined in $\Omega$,  the following standard notations are used, for $y \in \Omega'$: 
\begin{equation*} 
{\mathcal T}_* A(y) = \frac{\nabla {\mathcal T}  (x)  A (x) \nabla  {\mathcal T} ^{T}(x)}{\det \nabla  {\mathcal T}(x)} \quad \mbox{ and } \quad {\mathcal T}_*  j (y) = \frac{j(x)}{\det \nabla  {\mathcal T}(x)},
\end{equation*}
with $x ={\mathcal T}^{-1}(y)$.

We next recall the definition of complementary media and doubly complementary media \cite{Ng-Superlensing-Maxwell, Ng-CALR-M}. We begin with

\begin{definition}[Complementary media]   \fontfamily{m} \selectfont
 \label{def-Geo} Let $\Omega_1 \Subset  \Omega_2 \Subset  \Omega_3 \Subset  \mR^3$ be smooth,  bounded,   simply connected,  open subsets of $\mR^3$. The medium in $\Omega_2 \setminus \Omega_1$ characterized by a pair of two symmetric matrix-valued functions $(\eps_1, \mu_1)$  and the medium in $\Omega_3 \setminus \Omega_2$ characterized by a pair of two symmetric, uniformly elliptic, matrix-valued functions  $(\eps_2, \mu_2)$  are said to be  {\it  complementary} if 
there exists a diffeomorphism $\cF: \Omega_2 \setminus \bar \Omega_1 \to \Omega_3 \setminus \bar \Omega_2$ such that $\cF \in C^1(\bar \Omega_2 \setminus \Omega_1)$, 
\begin{equation}\label{cond-ASigma}
(\cF_*\eps_1, \cF_*\mu_1)   = (\eps_2, \mu_2)   \mbox{ for  } x \in  \Omega_3 \setminus \Omega_2, 
\end{equation}
\begin{equation}\label{cond-F-boundary}
\cF(x) = x \mbox{ on } \partial \Omega_2, 
\end{equation}
and the following two conditions hold: 1) There exists an diffeomorphism extension of $\cF$, which is still denoted by  $\cF$, from $\Omega_2 \setminus \{x_1\} \to \mR^3 \setminus \bar \Omega_2$ for some $x_1 \in \Omega_1$; and 2) there exists a diffeomorphism $\cG: \mR^3 \setminus \bar \Omega_3 \to \Omega_3 \setminus \{x_1\}$ such that $\cG \in C^1(\mR^3 \setminus \Omega_3)$, 
$\cG(x) = x \mbox{ on } \partial \Omega_3$,
and $
\cG \circ \cF : \Omega_1  \to \Omega_3 \mbox{ is a diffeomorphism if one sets } \cG \circ \cF(x_1) = x_1.
$
\end{definition}

\begin{definition} \label{def-DCM} \fontfamily{m} \selectfont The medium $(\eps_0, \mu_0)$ given in \eqref{def-eDelta} with $\delta  = 0$ is said to be {\it doubly complementary} if  for some $\Omega_2 \Subset  \Omega_3$, $(\ep, \mup)$ in $\Omega_3 \setminus \Omega_2$ and $(\en, \mun)$ in $\Omega_2 \setminus \Omega_1 $ are  complementary, and 
\begin{equation}\label{DCM}
(\cG_* \cF_*\ep, \cG_* \cF_*\mup) = (\ep,  \mup)  \mbox{ in } \Omega_3 \setminus \Omega_2 
\end{equation}
for some $\cF$ and $\cG$ from Definition~\ref{def-Geo}. 
\end{definition}

We now address the point that  makes the doubly complementary media special.  Let $(\eps_\delta, \mu_\delta)$ be defined by \eqref{def-eDelta} such that $(\eps_0, \mu_0)$ is doubly complementary. Assume that  $(E_\delta, H_\delta)$ is a solution of \eqref{Main-eq-delta} with $J = 0$ in $\Omega_3$.  Set, for $x' \in \mR^3 \setminus \Omega_2$,  
\begin{equation}\label{EH1d}
E_{1, \delta}(x') = \nabla {\mathcal F}^{-T} (x) E_{\delta}(x) \mbox{ and  } 
H_{1, \delta}(x') = \nabla {\mathcal F}^{-T} (x) H_{\delta}(x)\mbox{ with } x  ={\mathcal F}^{-1}(x'), 
\end{equation}
and,  for $y' \in  \Omega_3$, 
\begin{equation}\label{EH2d}
E_{2, \delta}(y') = \nabla {\mathcal G}^{-T} (y) E_{1, \delta}(y)  \mbox{ and  }  
E_{2, \delta}(y') = \nabla {\mathcal G}^{-T} (y) H_{1, \delta}(y) \mbox{ with } y  ={\mathcal G}^{-1}(y'). 
\end{equation}
Here $\cF, \cG$, and $\Omega_3$ are from the definition of doubly complementary media. 
Here and in what follows, we use the notation, for a diffeomorphism $\T$ and a vector field $v$,  
$$
\T*v (x') = \nabla \T^{-\mathsf{T}}(x) v(x) \mbox{ with } x' = \T(x). 
$$
By a change of variables, see e.g. \cite[Lemma 7]{Ng-Superlensing-Maxwell}, up to a (small) perturbation, one can check that $(E_{1, \delta}, H_{1, \delta})$ and $(E_{2, \delta}, H_{2, \delta})$ satisfy the {\it same}  Maxwell equations in $\Omega_3 \setminus \Omega_2$ as the one of $(E_\delta, H_\delta)$. It is clear that 
$$
E_{1, \delta} \times \nu - E_\delta \times \nu = H_{1, \delta} \times \nu - H_\delta \times \nu = 0 \mbox{ on } \partial \Omega_2 
$$
and 
$$
E_{2, \delta} \times \nu - E_{1, \delta} \times \nu = H_{2, \delta} \times \nu - H_{1, \delta} \times \nu = 0 \mbox{ on } \partial \Omega_3.
$$
Here on a boundary of a bounded subset of $\mR^3$, $\nu$ denotes
its normal unit vector directed to the exterior.  Hence, one has two Cauchy's problems with the {\it same} equations, one  for $(E_\delta, H_\delta)$ and $(E_{1, \delta}, H_{1, \delta})$ in $\Omega_3 \setminus \Omega_2$ with the boundary data given on $\partial \Omega_2$, and one for $(E_{1, \delta}, H_{1, \delta})$ and $(E_{2, \delta}, H_{2, \delta})$  in $\Omega_3 \setminus \Omega_2$ with the boundary data given on $\partial \Omega_3$.  This is the essential  property of $(\eps_\delta, \mu_{\delta})$ for the cloaking purpose and  the root of the definition of doubly complementary media.  

We list here some examples of doubly complementary media for which the formulas of $(\eps_0, \mu_0)$ are explicit; a general way to obtain doubly complementary media is presented in \cite{Ng-CALR-M}. Fix $p > 1$ and $r_2 > r_1 > 0$, and define $r_3 = r_2^{p}/ r_1^{p -1}$ and $m = r_2^{p} / r_1^p$. Set,  with the standard notations of  polar coordinates, 
\begin{equation}\label{def-M}
M = -  \frac{r_2^p}{ r^{p}}  \left[ \frac{1}{p - 1} e_{r} \otimes e_{r } + (p -1) \Big( e_{\theta} \otimes e_{\theta}  + e_{\theta} \otimes e_{\varphi}  \Big) \right] \mbox{ in } B_{r_2} \setminus B_{r_1}, 
\end{equation}
and define, for $\delta \ge 0$,  
\begin{equation}\label{def-epsmu}
(\eps_\delta, \mu_\delta) = \left\{ \begin{array}{cl} \big(M+ i \delta I, M+ i \delta I\big) & \mbox{ in } B_{r_2} \setminus B_{r_1}, \\[6pt]
\big(m I, m I \big) & \mbox{ in } B_{r_1}, \\[6pt]
(I, I) & \mbox{ otherwise}. 
\end{array} \right. 
\end{equation}
One can check that $(\eps_0, \mu_0)$ is doubly complementary with $\Omega_j = B_{r_j}$ for $j = 1, 2, 3$,  $\cF(x) = r_2^p x / |x|^p $ and $\cG = r_3^q x  / |x|^{q}$ with $q = p  / (p -1)$, see \cite[(1.6)]{Ng-Superlensing-Maxwell}. In the case $p = 2$, it is easy to see that  $M = -r_2^2 I / |x|^2$ in $B_{r_2} \setminus B_{r_1}$.
\medskip

For a doubly complementary medium $(\eps_0, \mu_0)$ and $J \in [L^2(\mR^3)]^3$ with $\supp J \cap \Omega_2 = \emptyset$ , set
\begin{equation}\label{def-tepsmu}
(\teps, \tmu) : = \left\{\begin{array}{cl} (\ep, \mup) & \mbox{ in } \mR^3 \setminus \Omega_3, \\[6pt]
(\cG_*\cF_*\ep, \cG_*\cF_*\mup) & \mbox{ in } \Omega_3, 
\end{array}\right.
\end{equation}
and let $(\tE, \tH) \in [H_{\loc}(\curl, \mR^3)]^2$ be the unique radiating solution of 
\begin{equation}\label{sys-tEH} \left\{
\begin{array}{cl}
\nabla \times \tE = i \omega \tmu \tH, & \mbox{ in } \mR^3, \\[6pt]
\nabla \times \tH = - i \omega \teps \tE + J, & \mbox{ in } \mR^3. 
\end{array}\right. 
\end{equation}
Note that if $(\eps_0, \mu_0)$ is doubly complementary, then  $\teps$ and $\tmu$ are uniformly elliptic in $\mR^3$ since $\det \cF < 0$ and $\det \cG < 0$.

\medskip

The following result reveals an interesting property of doubly complementary media \cite[Theorem 2.1 and Proposition 2.1]{Ng-CALR-M}: 

\begin{proposition}\label{pro-DCM}
Let  $0< \delta < 1$, $J \in [L^2(\mR^3)]^3$, and let $(E_\delta, H_\delta) \in [H_{\loc}(\curl, \mR^3)]^2$ be the unique radiating  solution of \eqref{Main-eq-delta}.  Assume that $(\eps_0, \mu_0)$ is doubly complementary and  $\supp J \subset B_{R_0} \setminus \Omega_2$ for some $R_0 >0$.  
Then, for $R> 0$,   
\begin{equation}\label{part1}
\| (E_\delta, H_\delta)  \|_{L^2(B_R \setminus \Omega_3)} \le C_R \| J \|_{L^2(\mR^3)}
\end{equation}
for some positive constant $C_R$ that  depends on $R$ but is independent of $J$ and $\delta$.  
Moreover, 
\begin{equation}\label{part2}
(E_\delta, H_\delta)  \mbox{ converges to }  (\tE, \tH)  \mbox{ in } [L^2_{\loc}(\mR^3 \setminus \Omega_3)]^6 \mbox{ as } \delta \to 0, 
\end{equation}
where $(\tE, \tH) \in [H_{\loc}(\curl, \mR^3)]^2$ is the unique radiating solution of \eqref{sys-tEH}.  Assume in addition that  $\Omega_3 \Subset B_{R_0}$ and  $\supp J \cap \Omega_3 = \emptyset$. We have, for $R> R_0$, 
\begin{equation}\label{pro2-p2}
\| (E_\delta, H_\delta) - (\tE, \tH) \|_{L^2 \big(B_R \setminus \Omega_3\big)} \le C_R \delta \| J\|_{L^2(\mR^3)}.   
\end{equation}
Here  $C_R$ denotes a positive constant that depends on $R$ but is independent of $J$ and $\delta$. 
\end{proposition}

We now present our main result on the cloaking property of  doubly complementary media. 

\begin{theorem} \label{thm-cloaking} Assume that $(\eps_0, \mu_0)$ is doubly complementary and of class $C^2$ in $\Omega_3 \setminus \Omega_2$. Let $\Gamma_1$ be a compact smooth curve on $\partial \Omega_1$ and 
$\Gamma_2$ be a compact smooth curve on $\partial \Omega_2$. Set, for $\gamma > 0$,  
\begin{equation*}
O_{j, \gamma} = \Big\{x \in \mR^3; \dist(x, \Gamma_j) < \gamma \Big\} \mbox{ for } j =1, 2. 
\end{equation*}
For $\gamma > 0$, let $(\eps_c, \mu_c)$ be a pair of symmetric, uniformly elliptic, matrix-valued functions defined in $O_\gamma: = (O_{1, \gamma} \cup O_{2, \gamma}) \setminus  (\Omega_2 \setminus \Omega_1)$ and define
\begin{equation*}
(\eps_{c, \delta}, \mu_{c, \delta}) = \left\{ \begin{array}{cl} (\eps_c, \mu_c) & \mbox{ in } O_{\gamma}, \\[6pt]
(\eps_\delta, \mu_\delta) & \mbox{ otherwise}. 
\end{array}\right. 
\end{equation*} 
For all $0< \alpha < 1$, there exists $\gamma_0 > 0$ depending only  on $\alpha$, $\Gamma_1$, $\Gamma_2$, $\Omega_j$ for $j = 1, 2, 3$,  and $(\eps_0, \mu_0)$ such that for $0< \gamma \le \gamma_0$, and for $J \in [L^2(\mR^3)]^3$ with $\supp J \subset B_{R_0} \setminus \Omega_3$ for some $R_0>0$, 
we have, for $0 < \delta < 1$,  
\begin{equation}\label{conclusion}
\|(E_{c, \delta}, H_{c, \delta}) - (\tE, \tH) \|_{L^2(B_{R} \setminus \Omega_3)} \le C_R \delta^\alpha \| J \|_{L^2},  
\end{equation}
for some $C_R$ depending only  on $\alpha$, $\Gamma_1$, $\Gamma_2$, $\Omega_j$ for $j = 1, 2, 3$,  $(\eps_0, \mu_0)$,  $\omega$, $R_0$,  and $R$. Here $(\tE, \tH), \, (E_{c, \delta}, H_{c, \delta}) \in [H_{\loc}(\curl, \mR^3)]^2$  are respectively the unique radiating solutions of \eqref{sys-tEH}  and of the following system
\begin{equation*}
\left\{\begin{array}{cl}
\nabla \times E_{c, \delta} = i \omega \mu_{c, \delta} H_{c, \delta} &  \mbox{ in } \mR^3, \\[6pt]
\nabla \times H_{c, \delta} = - i \omega \eps_{c, \delta}  E_{c, \delta} + J &  \mbox{ in } \mR^3.  
\end{array} \right.
\end{equation*}
\end{theorem}

As a consequence of  Theorem~\ref{thm-cloaking}, $\lim_{\delta \to 0} (E_{c, \delta}, H_{c, \delta}) = (\tE, \tH)$ in $\mR^3 \setminus \Omega_3$ for all $J$ with compact support outside $\Omega_3$ if $\gamma$ is sufficiently small.  One,  therefore,  cannot detect the difference between two media $(\eps_{c, \delta}, \mu_{c, \delta})$ and $(I, I)$ as $\delta \to 0$ by observing  of $(E_{c, \delta}, H_{c,\delta})$ outside $\Omega_3$ using the excitation $J$: cloaking is achieved for observers outside $\Omega_3$ in the limit as $\delta \to 0$. It is worth noting that the constant $\alpha$ and $C_R$ in \eqref{conclusion} does not depend on the ellipticity and the Lipschitz  constant of $(\eps_c, \mu_c)$.

\begin{remark}\label{rem-5.2} \rm 
Given $0 < r_1 < r_2$, set $r_3  = r_2^2/ r_1$, $\cF (x) = r_2^2 x /|x|^2$, $\cG(x) = r_3^2 x/ |x|^2$, 
$\Omega_j = B_{r_j}$ for $j=1, \, 2, \, 3$, and 
$$
\Gamma_{j} = \Big\{x \in \mR^3; |x| = r_j \mbox{ and } x_3 = 0  \Big\} \mbox{ for } j=1, \, 2. 
$$
Consider $(\eps_0, \mu_0)$ given by \eqref{def-epsmu} where $M$ is from \eqref{def-M} with $p = 2$. 
Take $\alpha = 2/3$. Applying \Cref{thm-cloaking}, one derives the statement on cloaking associated with doubly complementary media mentioned in the introduction. 
\end{remark}

\begin{remark} \rm It would be very interesting to understand the cloaking property considered in this paper in the time domain for dispersive materials, whose material constants are frequency dependent., see e.g. \cite{Ng-Vinoles} for a discussion on these materials and their basis properties in the time domain. 
\end{remark}

\begin{proof}[Proof of \Cref{thm-cloaking}] The proof is inspired by  \cite{Ng-CALR-O}.  Set 
\begin{equation}\label{TC-data}
\D(J, \delta) = \frac{1}{\delta} \left| \Im \int_{\mR^3} i \omega J \bar E_{c, \delta} + \Im \I(H_{c, \delta}) \right| +  \| J\|_{L^2(\mR^3)}^2, 
\end{equation}
where $\Im$ denotes the imaginary part, and 
\begin{equation}\label{TC-def-IH}
\I(H_{c, \delta})  = \lim_{R \to + \infty} \int_{\partial B_R} i \omega |H_{c, \delta}|^2. 
\end{equation}
By \Cref{lem-stability} below, we have, for $\gamma < \gamma_0$ where $\gamma_0$ is a (fixed) positive constant determined later, 
\begin{equation}\label{thm-cloaking-p1}
\|(E_{c, \delta}, H_{c, \delta}) \|_{L^2(B_R \setminus O_{\gamma_0})} \le C_R \D(J, \delta)^{1/2}. 
\end{equation}

The starting point of the proof is the use of reflections $\cF$ and $\cG$ from the definition of doubly complementary media.  Set 
\begin{equation} \label{TC-def-EH1delta}
(E_{c, 1, \delta}, H_{c, 1, \delta}) = (\cF*E_{c, \delta}, \cF*H_{c, \delta}) \mbox{ in } \mR^3 \setminus \Omega_2
\end{equation}
and 
\begin{equation} \label{TC-def-EH2delta}
(E_{c, 2, \delta}, H_{c, 2, \delta}) =  (\cG*E_{c, 1, \delta}, \cG*H_{c,1,  \delta}) \mbox{ in } \Omega_3. 
\end{equation}

Set  
$$
\Gamma_3 = \cF(\Gamma_1) \quad \mbox{ and } \quad O_{3, \gamma} = \Big\{x \in \mR^3; \dist(x, \Gamma_3 ) < \gamma \Big\}. 
$$
It is clear that  $\cG \circ \cF (O_{1, \gamma} \cap \Omega_1) \subset \Omega_3 \cap  O_{3, \lambda \gamma}$ for some positive constant $\lambda$. For notational ease, we will assume that $\lambda = 1$ (see \Cref{fig-cloaking} for the geometry of the setting).  

By a change of variables, see e.g. \cite[Lemma 7]{Ng-Superlensing-Maxwell}, we have 
\begin{equation}\label{TC-sys-EH1}
\left\{\begin{array}{cl}
\nabla \times E_{c. 1, \delta} = i \omega \mu^+ H_{c, 1, \delta} - \omega \delta \cF_*I H_{c, 1, \delta} & \mbox{ in } \Omega_3 \setminus (\bar \Omega_2 \cup \bar O_{2, \gamma}), \\[6pt]
\nabla \times H_{c, 1, \delta} = -  i \omega \eps^+ E_{c, 1, \delta} + \omega \delta \cF_*I E_{c, 1, \delta} & \mbox{ in } \Omega_3 \setminus (\bar \Omega_2 \cup \bar O_{2, \gamma}), 
\end{array}\right.
\end{equation}
\begin{equation}\label{TC-sys-EH2}
\left\{\begin{array}{cl}
\nabla \times E_{c, 2, \delta} = i \omega \tmu H_{c, 2, \delta}  & \mbox{ in } \Omega_3 \setminus \bar O_{3, \gamma}, \\[6pt]
\nabla \times H_{c, 2, \delta} = -  i \omega \teps H_{c, 2, \delta} & \mbox{ in } \Omega_3 \setminus \bar O_{3, \gamma}, 
\end{array}\right.
\end{equation} 
\begin{equation}\label{TC-bdry-EH1}
E_{c, 1, \delta} \times \nu - E_{c, \delta} \times \nu  = H_{c, 1, \delta} \times \nu - H_{c, \delta} \times \nu   =  0 \mbox{ on } \partial \Omega_2, 
\end{equation}
\begin{equation}\label{TC-bdry-EH2}
E_{c, 2, \delta} \times \nu - E_{c, 1, \delta} \times \nu = H_{c, 2, \delta} \times \nu - H_{c, 1, \delta} \times \nu  = 0  \mbox{ on } \partial \Omega_3.  
\end{equation}

Set $\beta = (\alpha + 2)/3$. Note that $(\teps, \tmu) = (\eps^+, \mu^+)$ in $\Omega_3 \setminus \Omega_2$ thanks to the property of doubly complementary media. Applying \Cref{cor-Maxwell-Cl} to $(E_{c, 1, \delta} - E_{c, \delta}, H_{c, 1, \delta} - H_{c, \delta})$ in $\Omega_3 \setminus \Omega_2$ with $\Sigma = \partial \Omega_2$, and to $(E_{c, 2, \delta} - E_{c, 1, \delta}, H_{c, 2, \delta} - H_{c, 1, \delta})$ in $\Omega_3 \setminus \Omega_2$ with $\Sigma = \partial \Omega_3$,
 there exist $\gamma_0, \gamma_2$ with $0< \gamma_0<  \gamma_2/2 < \gamma_2$, depending only on $\alpha$, $(\eps_0, \mu_0)$, $\Omega_1$, $\Omega_2$, $\Gamma_1$, $\Gamma_2$, and $\omega$ such that for 
$$
 \hhgamma = (\gamma_1 + \gamma_2)/2  \quad \mbox{ with }  \gamma_1 = \gamma_2/2, 
$$ 
we have,   with $\cO_{2, \hhgamma} = \Omega_2 \cup O_{2, \hhgamma}$,  by \eqref{TC-sys-EH1} and \eqref{TC-bdry-EH1}, 
\begin{multline*} 
\| (E_{c,1, \delta} \times \nu, H_{c, 1, \delta} \times \nu)  - (E_{c, \delta} \times \nu, H_{c, \delta} \times \nu)  \|_{H^{-1/2}(\dive_\Gamma, \partial \cO_{2, \hhgamma})} \\[6pt]
\le  C \| \big(\delta \cF_*I H_{c, 1, \delta},  \delta \cF_*I E_{c, 1, \delta} \big)  \|_{L^2 \big(\Omega_3 \setminus (\bar \Omega_2 \cup \bar O_{2, \gamma_0}) \big)}^{\beta} \times \\[6pt]
\times  \| (E_{c,1, \delta}, H_{c, 1, \delta} )  - (E_{c, \delta}, H_{c, \delta} )  \|_{L^2 \big(\Omega_3 \setminus (\bar \Omega_2 \cup \bar O_{2, \gamma_0}) \big)}^{1 - \beta}, 
\end{multline*}
and, with  $\cO_{3, \hhgamma} = \Omega_{3} \setminus  \bar O_{3, \hhgamma}$, by \eqref{TC-sys-EH2} and \eqref{TC-bdry-EH2}, 
\begin{multline*}
\| (E_{c, 2, \delta} \times \nu, H_{c, 2, \delta} \times \nu) - (E_{c, 1, \delta} \times \nu, H_{c, 1, \delta} \times \nu)  \|_{H^{-1/2}(\dive_\Gamma, \partial O_{3, \hhgamma} )} \\[6pt]
\le  C  \| \big( \delta \cF_*I H_{c, 1, \delta},  \delta \cF_*I E_{c, 1, \delta} \big)  \|_{L^2\big(\Omega_3 \setminus \bar O_{3, \gamma_0} \big)}^{\beta}\| (E_{c, 2, \delta} , H_{c, 2, \delta}) - (E_{c, 1, \delta}, H_{c, 1, \delta})  \|_{L^2\big(\Omega_3 \setminus \bar O_{3, \gamma_0} \big)}^{1-\beta}. 
\end{multline*}
It follows from \eqref{thm-cloaking-p1} that 
\begin{equation}\label{TC-claim1-cloaking}
\| (E_{c,1, \delta} \times \nu, H_{c, 1, \delta} \times \nu)  - (E_{c, \delta} \times \nu, H_{c, \delta} \times \nu)  \|_{H^{-1/2}(\dive_\Gamma, \partial \cO_{2, \hhgamma})} \le 
C \delta^{\beta} \D (J, \delta)^{1/2}
\end{equation}
and
\begin{equation}\label{TC-claim2-cloaking}
\| (E_{c, 2, \delta} \times \nu, H_{c, 2, \delta} \times \nu) - (E_{c, 1, \delta} \times \nu, H_{c, 1, \delta} \times \nu)  \|_{H^{-1/2}(\dive_\Gamma, \partial O_{3, \hhgamma} )} \le 
C \delta^{\beta} \D (J, \delta)^{1/2}. 
\end{equation}

Set 
\begin{equation*}
D= (\Omega_3 \setminus \Omega_2) \setminus (O_{3, \hhgamma} \cup O_{2, \hhgamma}). 
\end{equation*}
A simple but important part of the proof is the introduction of  $(\tE_{\delta}, \tH_{\delta}) \in [H_{\loc}(\curl, \mR^3 \setminus \partial D)]^2$, known as the removing localized singularity,  as follows:
\begin{equation}\label{TC-main-tEH}
(\tE_{\delta}, \tH_{\delta}) = \left\{\begin{array}{cl}
(E_{c, \delta}, H_{c, \delta})  - \Big[  (E_{c, 1, \delta}, H_{c, 1, \delta})  - (E_{c, 2, \delta}, H_{c, 2, \delta})  \Big] & \mbox{ in } D, \\[6pt]
(E_{c, 2, \delta}, H_{c, 2, \delta}) & \mbox{ in } \Omega_2 \cup O_{2, \hhgamma}, \\[6pt] 
(E_{c, \delta}, H_{c, \delta}) & \mbox{ otherwise}
\end{array}\right.
\end{equation}
(see \Cref{fig-cloaking}). 

It follows from \eqref{TC-sys-EH1}, \eqref{TC-sys-EH2},   the definition of $(\teps, \tmu)$ in \eqref{def-tepsmu}, and the property of doubly complementary media  that  $(\tE_\delta, \tH_\delta) \in [H_{\loc}(\curl, \mR^3)]^2$ is a radiating solution of the system  
\begin{equation*} \left\{
\begin{array}{cl}
\nabla \times \tE_\delta = i \omega \tmu \tH_\delta + \delta  \omega \mathds{1}_{D} \cF_*I H_{c, 1, \delta}  & \mbox{ in } \mR^3 \setminus \partial D, \\[6pt]
\nabla \times \tH_\delta = -  i \omega \teps \tE_\delta  -  \delta \omega  \mathds{1}_{D} \cF_*I E_{c, 1, \delta}  + J & \mbox{ in } \mR^3 \setminus \partial D. 
\end{array}\right. 
\end{equation*}
We derive from the definition of $(\tE, \tH)$ in \eqref{sys-tEH}  that  $(\tE_\delta - \tE, \tH_\delta - \tH) \in [H_{\loc}(\curl, \mR^3 \setminus \partial D)]^2$ is  radiating  and satisfies  
\begin{equation*} \left\{
\begin{array}{cl}
\nabla \times (\tE_\delta - \tE) = i \omega \tmu  (\tH_\delta - \tH) +  \delta \omega \mathds{1}_{D} \cF_*I H_{c, 1, \delta} & \mbox{ in } \mR^3 \setminus \partial D, \\[6pt]
\nabla \times (\tH_\delta - \tH) = -  i \omega \teps (\tE_\delta - \tE)  -  \delta \omega  \mathds{1}_{D} \cF_*I E_{c, 1, \delta} & \mbox{ in } \mR^3 \setminus \partial D.
\end{array}\right. 
\end{equation*}
Since $\teps$ and $\tmu$ are uniformly elliptic, we deduce from \eqref{TC-claim1-cloaking} and \eqref{TC-claim2-cloaking} that 
\begin{equation}\label{main-p1}
\| (\tE_\delta - \tE, \tH_\delta - \tH) \|_{L^2(B_R)} \le C_R  \delta^{\beta} \D(J, \delta)^{1/2}. 
\end{equation}
Since $\beta > 1/2$ and $\D(J, \delta)  \le C \delta^{-1} \| J\|_{L^2(\mR^3)}^2
$,  it follows that  
\begin{equation}\label{main-p2}
 \| (E_{c, \delta}, H_{c, \delta}) \|_{L^2(B_{R_0} \setminus \Omega_3)} \le C \| J\|_{L^2(\mR^3)},  
\end{equation}
and
\begin{equation}\label{main-p3}
\| (\tE_\delta - \tE, \tH_\delta - \tH) \|_{L^2(B_R)}  \le C_R \delta^{\beta - 1/ 2} \| J\|_{L^2(\mR^3)}.  
\end{equation}

Since $\beta > 1/2$, one derives from \eqref{main-p3} that 
$$
\lim_{\delta \to 0} \| (\tE_\delta - \tE, \tH_\delta - \tH) \|_{L^2(B_R)} =0, 
$$
which already yields the cloaking phenomenon since $(E_{c, \delta}, H_{c, \delta}) = (\tE_\delta, \tH_\delta)$ outside $\Omega_3$.  

To reach the convergent rate, we process as in  \cite{Ng-Survey}.  Considering the system $(\tE, \tH)$, multiplying the equation of $\nabla \times \tE = i \omega \mu \tH $ by $\nabla \times \overline{\tE}$,  integrating by parts in $B_R$, letting $R \to + \infty$, and using the radiation condition, one has
\begin{equation*}
\Im \int_{\mR^3} i \omega J \overline{\tE} + \Im I (\tH) = 0. 
\end{equation*}
It follows that 
\begin{multline}\label{St-p1}
\Big|\Im \int_{\mR^3} i \omega J \bar E_{c, \delta} + \Im I (H_{c, \delta}) \Big| = \Big|\Im \int_{\mR^3} i \omega J \bar E_{c, \delta} + \Im I (H_{c, \delta})   - \Im \int_{\mR^3} i \omega J \overline{\tE} - \Im I (\tH)  \Big| \\[6pt]
\le  \Big|\Im \int_{\mR^3} i \omega J \bar E_{c, \delta}   - \Im \int_{\mR^3} i \omega J \overline{\tE}  \Big| +  \Big| \Im I (H_{c, \delta})   - \Im I (\tH)  \Big|. 
\end{multline}
We have 
\begin{equation}\label{St-p2}
\Big|\Im \int_{\mR^3} i \omega J \bar E_{c, \delta}   - \Im \int_{\mR^3} i \omega J \overline{\tE}  \Big| \le C \| J \|_{L^2} \| (E_{c, \delta} - \tE) \|_{L^2(B_{R_0} \setminus \Omega_3)}. 
\end{equation}
Multiplying the equation of $\nabla \times \tE = i \omega \mu \tH $ by  $\nabla \times \overline{\tE}$ and the equation of $\nabla \times E_{c, \delta} = i \omega \mu H_{c, \delta} $ by  $\nabla \times \overline{E_{c, \delta}}$,  integrating by parts in $B_R \setminus B_{R_0}$, letting $R \to + \infty$, and using the radiation condition, we obtain 
$$
 \Big| \Im I (H_{c, \delta})   - \Im I (\tH)  \Big|  = \left| \int_{\partial B_{R_0}} i \omega (\overline{E_{c, \delta}} \times \nu) \cdot H_{c, \delta} -  \int_{\partial B_{R_0}} i \omega (\overline{\tE} \times \nu) \cdot \tH
\right|.
$$
We derive that 
\begin{equation}\label{St-p3}
 \Big| \Im I (H_{c, \delta})   - \Im I (\tH)  \Big| \le C  \| J \|_{L^2} \| (E_{c, \delta} - \tE, H_{c, \delta} - \tH) \|_{L^2(B_{R_0} \setminus \Omega_3)}. 
\end{equation}
Combining \eqref{St-p1}, \eqref{St-p2}, and \eqref{St-p3} yields 
\begin{equation}\label{St-p4}
\Big|\Im \int_{\mR^3} i \omega J \bar E_{c, \delta} + \Im I (H_{c, \delta}) \Big| \le C  \| J \|_{L^2} \| (E_{c, \delta} - \tE, H_{c, \delta} - \tH) \|_{L^2(B_{R_0} \setminus \Omega_3)}. 
\end{equation}

From \eqref{TC-data}, \eqref{main-p3},  and \eqref{St-p4}, we obtain 
\begin{equation*}
\D(J, \delta) \le C \delta^{\beta - 1/2 - 1 }\| J \|_{L^2(\mR^3)}^2. 
\end{equation*}

Repeating this process, one reaches (see \cite[Proof of Theorem 3.1]{Ng-Survey} for related arguments), for $\ell \ge 1$, 
\begin{equation*}
\D(J, \delta) \le C_\ell \delta^{\beta (1 + \cdots + 1/ 2^{\ell-1}) - (1/2 + \cdots + 1/2^\ell) - 1}  \| J \|_{L^2}^2
\end{equation*}
and 
\begin{equation*}
\| (\tE_\delta - \tE, \tH_\delta - \tH) \|_{L^2(B_R)} \le C_\ell \delta^{\beta(1 + \cdots 1/ 2^\ell) - (1/2 + \cdots 1 /2^{\ell+1})} \| J \|_{L^2}. 
\end{equation*}
The conclusion follows by taking $\ell$ sufficiently large. 
 \end{proof}

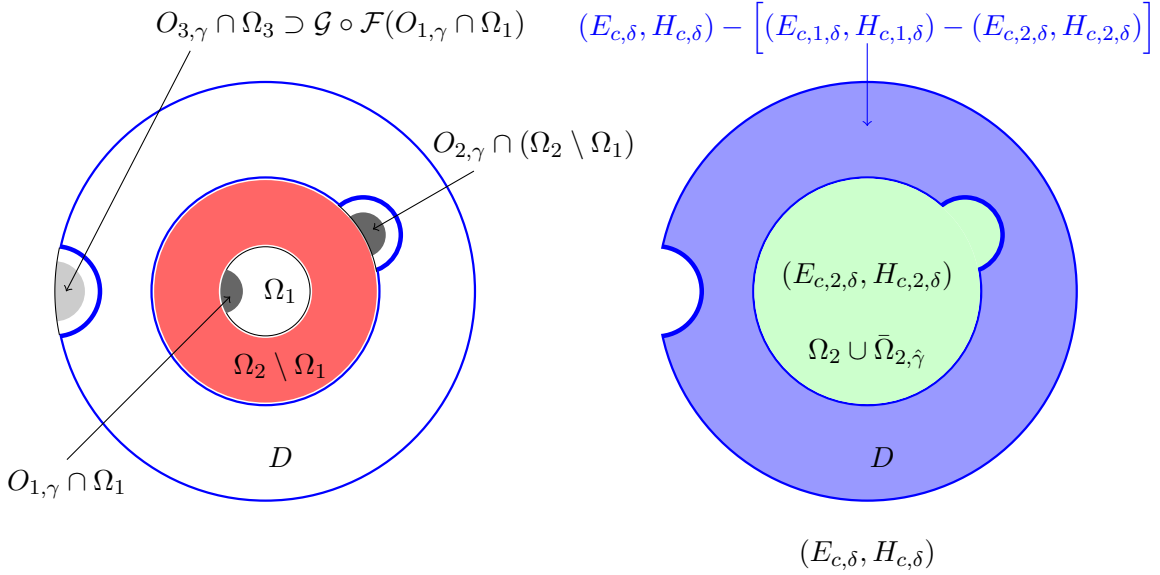
\begin{figure}
\centering
\begin{tikzpicture}[scale=1]

\draw[blue, line width=0.6mm]  (0, 0) circle(2.8);

\fill[white] (-2.8,0) circle (0.6);

\draw[blue, line width=0.6mm] (-2.8,0) circle (0.6);

\fill [white,even odd rule] (0,0) circle[radius=2.8cm] circle[radius=3.5cm];

\begin{scope}
\clip (-2.8, 0) circle(0.4);
\fill[black!20]  (0, 0) circle(2.8);
\end{scope}

\draw [black, domain =-168:-192] plot ({2.8*cos(\x)}, {2.8*sin(\x)});

\draw[blue, line width=0.6mm]  (0, 0) circle(1.5);

\fill[white] ({1.5*cos(30)},{1.5*sin(30)}) circle (0.5);

\draw[blue, line width=0.6mm] ({1.5*cos(30)},{1.5*sin(30)}) circle (0.5);

\fill[white] (0, 0) circle(1.5);

\fill[black!60] ({1.5*cos(30)},{1.5*sin(30)}) circle(0.3);

\fill[white] (0,0) circle (1.4);

\begin{scope}
\clip ({1.5*cos(30)},{1.5*sin(30)}) circle(0.3);
\fill[black!20]  (0, 0) circle(1.5);
\end{scope}

\draw [black, domain =10:50] plot ({1.5*cos(\x)}, {1.5*sin(\x)});

\draw[thick] (0, 0) circle(0.6);

\fill[black!60] (-0.6,0) circle (0.3);

\fill [white,even odd rule] (0,0) circle[radius=0.6cm] circle[radius=0.9cm];

\fill [red!60,even odd rule] (0,0) circle[radius=0.62cm] circle[radius=1.48cm];

\draw (0.2, 0) node{$\Omega_1$};

\draw (0.2, -1.) node{$\Omega_2 \setminus \Omega_1$};

\draw (0.2, -2.2) node{$D$};

\draw[->] (-1, 3.2) -- (-2.65, 0); 

\draw ((1, 3.2) node[above]{$ O_{3, \gamma} \cap \Omega_3 \supset \cG \circ \cF (O_{1, \gamma} \cap \Omega_1)$};

\draw[->] ({3.2*cos(30)} ,{ 3.2*sin(30)}) -- ({1.65*cos(30)} , {1.65*sin(30)}); 

\draw ({0.8+3.2*cos(30)} ,{ 3.2*sin(30)})  node[right,above]{$O_{2, \gamma} \cap (\Omega_2 \setminus \Omega_1)$}; 


\draw[->] ({-3.7 + 1.5*cos(45)} ,{-3.3 +  1.5*sin(45)}) -- ({-1+0.75*cos(45)} ,  {-0.6 + 0.75*sin(45)});  

\draw[] ({-3.7 + 1.5*cos(45)} ,{-3.3 +  1.5*sin(45)})  node[below]{$O_{1, \gamma} \cap \Omega_1$};

\newcommand\z{8.0}

\filldraw[blue!40]  (0+\z, 0) circle(2.8);

\draw[blue, line width=0.6mm]  (0+\z, 0) circle(2.8);

\fill[white] (-2.8+\z,0) circle (0.6);

\draw[blue, line width=0.6mm] (-2.8+\z,0) circle (0.6);

\fill [white,even odd rule] (0+ \z,0) circle[radius=2.8cm] circle[radius=3.5cm];

\draw[blue, line width=0.6mm]  (0+\z, 0) circle(1.5);

\fill[green!20] ({\z+1.5*cos(30)},{1.5*sin(30)}) circle (0.5);

\draw[blue, line width=0.6mm] ({\z+ 1.5*cos(30)},{1.5*sin(30)}) circle (0.5);

\fill[white] (0+\z, 0) circle(1.5);

\fill[green!20] (0+\z,0) circle (1.5);

\draw (0.2+ \z, -2.2) node{$D$};

\draw (0+ \z, 0.2) node{$(E_{c, 2, \delta}, H_{c, 2, \delta})$}; 
\draw (0+ \z, -0.8) node{$\Omega_2 \cup \bar \Omega_{2, \hhgamma}$};

\draw(0+ \z, 3.5) node{\textcolor{blue}{$(E_{c, \delta}, H_{c, \delta})  - \Big[  (E_{c, 1, \delta}, H_{c, 1, \delta})  - (E_{c, 2, \delta}, H_{c, 2, \delta})  \Big] $}};

\draw[->, blue] (0+ \z, 3.3) -- (0+ \z, 2.2);

\draw(0+ \z, -3.5) node{$(E_{c, \delta}, H_{c, \delta})$};  
\end{tikzpicture}
\caption{On the left: the geometry of the cloaking setting,  $O_\gamma = (O_{1, \gamma} \cap \Omega_1) \cup \big( O_{2, \gamma} \cap (\Omega_2 \setminus \Omega_1) \big)$ is the cloaked region,  the plasmonic structure in $\Omega_2 \setminus \Omega_1$ is the red region; On the right: the definition of $(\tE_\delta, \tH_\delta)$.} \label{fig-cloaking}
\end{figure}

\begin{remark} \rm
The removing localized singularity technique introduced in \cite{Ng-Negative-Cloaking, Ng-Superlensing} is
inspired by the idea of renormalizing energy in the theory of the 
Ginzburg-Landau equation \cite{BBH}.  The gluing argument was first suggested in \cite{Ng-Complementary}. This plays an important role in 
our study of negative index materials, see \cite{Ng-Survey} for a survey. 
\end{remark}

The following lemma was used in the proof of \Cref{thm-cloaking}. 

\begin{lemma} \label{lem-stability} Under the assumption of \Cref{thm-cloaking},  we have, for $\gamma < \hat \gamma$, 
\begin{equation}\label{stability-1}
\|(E_{c, \delta}, H_{c, \delta}) \|_{L^2(B_R \setminus D_{\hat \gamma})} \le C_R \D(J, \delta), 
\end{equation}
where $\D(J, \delta)$ is defined by \eqref{TC-data}. Here $C_R$ denotes a positive constant depending only on $\hat \gamma$, $R$, $R_0$, $\eps$, $\mu$, $\Gamma_1$, $\Gamma_2$, $\Omega_1$, and $\Omega_2$. 
\end{lemma}
  
\begin{proof} The proof of this lemma is quite simple as follows. We have, in $\mR^3$,  
\begin{equation*}
\nabla \times ( \mu_{c, \delta}^{-1} \nabla \times E_{c, \delta}) - \omega^2 \eps_{c, \delta} \cE_\delta =   i \omega J. 
\end{equation*}
Multiplying the equation by $\bar E_{c, \delta}$, integrating in $B_R$, and using the fact $\supp J \subset B_{R_0} \setminus \Omega_2$, we obtain, for $R > R_0$, 
\begin{equation*}
\int_{B_R} \langle  \mu_\delta^{-1} \nabla \times E_{c, \delta}, \nabla \times E_{c, \delta} \rangle  + \int_{\partial B_{R}}  \langle i \omega H_{c, \delta} , E_{c, \delta} \times \nu \rangle  - \omega^2 \int_{B_R} \langle \eps_\delta E_{c, \delta}, E_{c, \delta} \rangle =  \int_{B_R}  \langle i \omega J, E_{c, \delta} \rangle.
\end{equation*}
Letting $R \to + \infty$, using the radiation condition,  and considering the imaginary part,  we get
\begin{equation}\label{inE}
 \|(E_{c, \delta}, H_{c, \delta}) \|_{L^2(\Omega_2 \setminus \Omega_1)}^2 \le C  \D(J, \delta).   
\end{equation}
It follows from the trace theory that 
\begin{equation}\label{bdryE}
\| E_{c, \delta} \times \nu \|_{H^{-1/2}(\dive_\Gamma, \partial \Omega_1 \cup \partial \Omega_2)}^2 + \| H_{c, \delta} \times \nu \|_{H^{-1/2}(\dive_\Gamma, \partial \Omega_1 \cup \partial \Omega_2)}^2  \le C \D(J, \delta). 
\end{equation}
A compactness argument involving the unique continuation principle gives, (see e.g. the proof of \cite[Lemma 3]{Ng-Superlensing-Maxwell} for  similar arguments), one has
\begin{equation}\label{inEHo}
\| (E_{c, \delta}, H_{c, \delta})\|_{L^2(B_R \setminus D_{\hat \gamma})}^2 \le C_R \D(J, \delta). 
\end{equation}
The proof is complete. \end{proof}

\begin{remark}  \fontfamily{m} \selectfont In the proof 
\Cref{lem-stability}, the complementary property of $(\eps_0, \mu_0)$ is not required. 
\end{remark}

\section{On superlensing and cloaking using complementary media} \label{sect-discussion}

In this section, we discuss the lensing and cloaking designs using complementary media given in \cite{Ng-Superlensing-Maxwell, Ng-Negative-Cloaking-M}. We show on one hand that  it is necessary to impose additional conditions on the schemes proposed in the physics literature.  On the other hand, we discuss various contexts where a lens can act like a cloak and conversely. 

\subsection{Superlensing using complementary media}\label{sect-CM-Superlensing}

In this section, we analyse the lensing construction given in \cite{Ng-Superlensing-Maxwell} motivated from \cite{Veselago, PendryNegative, Ng-Superlensing}. To magnify $m$-times the region $B_{\tau_0}$ of material parameters $(\eps_O, \mu_O)$ (a pair of uniformly elliptic,  symmetric, matrix-valued functions),  for some $\tau_0> 0$ and $m> 1$, we proposed using two layers. One layer  makes use of complementary media concept
\begin{equation}\label{lens-p1}
\big(\cF^{-1}_*I , \cF^{-1}_*I\big)  \mbox{ in } B_{r_2} \setminus B_{r_1}, 
\end{equation}
and the other layer is given by 
\begin{equation}\label{lens-p2}
\big(m I, m I \big) \mbox{ in } B_{r_1} \setminus B_{r_0}. 
\end{equation}
Here $\cF$ is the Kelvin transform with respect to $\partial B_{r_2}$, and $r_1$ and $r_2$ are required to satisfy $r_1 \ge  m^{1/2} r_0 $ and $r_2 =  m r_0$. The construction in \cite{Ng-Superlensing-Maxwell} is for the case where $r_1 =  m^{1/2} r_0 $,  nevertheless, the case $r_1 \ge  m^{1/2} r_0  $ is its direct consequence.   Other choices for the first and the second layers are possible via the concept of complementary media and were analyzed there.  

Set 
\begin{equation*}
(\eps_\delta, \mu_\delta) = \left\{ \begin{array}{cl} (I, I) & \mbox{ in } \mR^3 \setminus  B_{m r_0}, \\[6pt]
\big(m I, m I \big) & \mbox{ in } B_{r_1} \setminus B_{r_0}, \\[6pt]
 \big(\cF^{-1}_*I + i \delta I , \cF^{-1}_*I + i \delta I\big)  & \mbox{ in } B_{r_2} \setminus B_{r_1}, \\[6pt]
\big(\eps_O, \mu_O \big)  & \mbox{ in } B_{r_0},  
\end{array} \right. 
\end{equation*} 
and 
\begin{equation*}
(\hat \eps, \hat  \mu) = \left\{ \begin{array}{cl} (I, I) & \mbox{ in } \mR^3 \setminus  B_{m r_0}, \\[6pt]
\big(m^{-1}\eps_O(x/ m), m^{-1}\mu_O(x/ m) \big)  & \mbox{ otherwise}. 
\end{array} \right. 
\end{equation*} 

Assume that, with $M = \eps_O$ or $\mu_O$, 
\begin{equation}\label{cond-epsmu-S}
|\nabla M(x)| \le \Lambda, \quad  \Lambda^{-1} |\xi|^2 \le \langle M (x) \xi, \xi \rangle \le \Lambda |\xi|^2  \; \;  \forall \,  \xi \in \mR^d \mbox{ for some } \Lambda \ge 1.   
\end{equation}
Given $J \in [L^{2}(\mR^3)]^3$ with compact support outside $B_{r_3}$ with $r_3 = r_2^2/r_1$, let  $(E_\delta, H_\delta), (\hE, \hH) \in [H_{\loc}(\curl, \mR^3)]^2$ be,  respectively,  the unique radiating solutions  to 
\begin{equation}\label{sys-EH-delta}
\left\{ \begin{array}{lll}
\nabla \times E_\delta &= i \omega \mu_\delta H_\delta & \mbox{ in } \mR^3, \\[6pt]
\nabla \times H_\delta & = - i \omega \eps_\delta E_\delta +  J & \mbox{ in }  \mR^3,  
\end{array} \right.
\end{equation}
and 
\begin{equation}\label{sys-hat-EH}
\left\{ \begin{array}{lll}
\nabla \times \hE &= i \omega \hat  \mu \hH & \mbox{ in } \mR^3, \\[6pt]
\nabla \times \hH & = - i \omega \hat \eps \hE + J & \mbox{ in }  \mR^3.   
\end{array} \right.
\end{equation}
We showed, as $\delta \to 0$,  \cite[Theorem 1]{Ng-Superlensing-Maxwell} that 
$$
(E_\delta, H_\delta) \to (\hE, \hH) \mbox{ in } \mR^3 \setminus B_{r_3}. 
$$
For an observer outside $B_{r_3}$, measuring $(E_\delta, H_\delta)$ using the excitation $J$ gives the same results as measuring $(\hE, \hH)$ using the same excitation.  The object in $B_{\tau_0}$ is magnified $m$-times for such an observer. 
  
The second layer  can be chosen thinner using the analysis in \cite{Ng-Superlensing};  nevertheless, 
the second layer  is necessary. This follows   from the following direct consequence of Theorem~\ref{thm-cloaking}:  

\begin{proposition} \label{pro-lensing} Assume that $(\eps_0, \mu_0)$ is doubly complementary with $\cF$ and $\cG$ being the Kelvin transforms with respect to $\partial B_{r_2}$ and $\partial B_{r_3}$ with $r_3 = r_2^2/ r_1$. Let $\Gamma_1$ be a compact smooth curve on $\partial B_{r_1}$. Set, for $\gamma > 0$,  
\begin{equation*}
O_{1, \gamma} : = \Big\{x \in \mR^3; \dist(x, \Gamma_1) < \gamma \Big\}. 
\end{equation*}
For $\gamma > 0$, let $(\eps_c, \mu_c)$ be a pair of symmetric matrix-valued functions defined in $O_\gamma: = O_{1, \gamma} \cap B_{r_1}$ and define
\begin{equation*}
(\eps_{c, \delta}, \mu_{c, \delta}) = \left\{ \begin{array}{cl} (\eps_c, \mu_c) & \mbox{ in } O_{\gamma}, \\[6pt]
(\eps_\delta, \mu_\delta) & \mbox{ otherwise}. 
\end{array}\right. 
\end{equation*} 
Let $0< \delta < 1$,  $J \in [L^2(\mR^3)]^3$ with $\supp J \subset B_{R_0} \setminus B_{r_3}$,   and let  $(E_{c, \delta}, H_{c, \delta}) \in [H_{\loc}(\curl, \mR^3)]^2$ be  the unique radiating solution of the Maxwell equations 
\begin{equation*}
\left\{\begin{array}{cl}
\nabla \times E_{c, \delta} = i \omega \mu_{c, \delta} H_{c, \delta} &  \mbox{ in } \mR^3, \\[6pt]
\nabla \times H_{c, \delta} = - i \omega \eps_{c, \delta}  E_{c, \delta} + J &  \mbox{ in } \mR^3.  
\end{array} \right.
\end{equation*}
For all $0< \alpha < 1$, there exists $\gamma_0 > 0$ depending only on $\alpha$, $\Gamma_1$,  $r_1$,  and $r_2$ such that for $\gamma \in (0, \gamma_0)$, we have 
\begin{equation*}
\|(E_{c, \delta}, H_{c, \delta}) - (\tE, \tH) \|_{L^2(B_{R} \setminus \Omega_3)} \le C_R \delta^\alpha \| J \|_{L^2},  
\end{equation*}
for some positive constant $C_R$ depending only on $\alpha$, $\Gamma_1$,  $r_1$,  and $r_2$, $\Lambda$, $R_0$, and $R$. 
\end{proposition} 

As a consequence of \Cref{pro-cloaking}, an object inside $B_{r_1}$ located near the  layer $B_{r_2} \setminus B_{r_1}$ is cloaked; the second layer in the lensing construction given in \eqref{lens-p2} is hence necessary to achieve superlensing. 

\subsection{Cloaking using complementary media} \label{sect-CM-Cloaking}

In this section, we analyse the construction of the cloaking device in \cite{Ng-Negative-Cloaking-M} motivated from \cite{Lai1, Ng-Negative-Cloaking}.   Assume that the cloaked region is the annulus  $B_{2r_2} \setminus B_{r_2}$ in $\mR^3$ for some $r_{2}> 0$ in which the medium is characterized by  $(\eps_O, \mu_O)$ (a pair of uniformly elliptic symmetric matrix-valued functions).  The cloaking device proposed in \cite{Ng-Negative-Cloaking-M} then contains two parts. The first one, in $B_{r_2} \setminus  B_{r_1}$, makes use of complementary media to cancel the effect of the cloaked region and the second one, in $B_{r_1}$, is to fill the space which ``disappears" from the cancellation by the homogeneous medium. Concerning the first part, 
instead of $B_{2r_2} \setminus B_{r_2}$, we consider $B_{r_3} \setminus B_{r_2}$ for some $r_3 > 0$ as the cloaked region in which the medium is given by 
\begin{equation}\label{tepsmu}
\big(  \teps_O, \tmu_O \big) = \left\{ \begin{array}{cl} \big(\eps_O, \mu_O\big) & \mbox{ in } B_{2 r_2} \setminus B_{r_2}, \\[6pt]
\big(I, I \big) & \mbox{ in } B_{r_3} \setminus B_{2 r_2}. 
\end{array} \right. 
\end{equation} 
The complementary medium  in $B_{r_2} \setminus B_{r_1}$ is then given by 
\begin{equation}\label{first-layer}
\big(\cF^{-1}_*\teps_O, \cF^{-1}_*\tmu_O \big), 
\end{equation}
where $\cF$ is the Kelvin transform with respect to $\partial B_{r_2}$.  Concerning the second part, the medium in $B_{r_1}$ with $r_1 =  r_2^2/ r_3$ is given by, with $m = r_3^2/ r_2^2 = r_2^2/ r_1^2$, 
\begin{equation}\label{choice-Br1}
\big(m I, m I \big). 
\end{equation}
Set 
\begin{equation}\label{def-eps-mu}
(\eps_\delta, \mu_\delta) = \left\{ \begin{array}{cl} 
\big( \teps_O , \tmu_O \big)& \mbox{ in } B_{r_3} \setminus B_{r_2}, \\[6pt]
\big( F^{-1}_* \teps_O +  i  \delta I, F^{-1}_* \tmu_O  + i\delta I  \big) & \mbox{ in } B_{r_2} \setminus B_{r_1}, \\[6pt]
\big( m   I,  m  I \big) & \mbox{ in } B_{r_1},\\[6pt]
 \big(I,  I \big) & \mbox{ in } \mR^3 \setminus B_{r_3}.  
\end{array} \right. 
\end{equation}
Given  $J \in \big[L^2(\mR^3) \big]^3$ with compact support outside $B_{r_3}$,   let  $(E_\delta, H_\delta), (E, H) \in [H_{\loc}(\curl, \mR^3)]^2$ be respectively the unique outgoing solutions to the Maxwell systems
\begin{equation}\label{eq-EHdelta}
\left\{ \begin{array}{llll}
\nabla \times E_\delta &= & i \omega  \mu_\delta H_\delta & \mbox{ in } \mR^3\\[6pt]
\nabla \times H_\delta & =&  - i \omega  \eps_\delta E_\delta + J & \mbox{ in }  \mR^3, 
\end{array} \right.
\end{equation}
and 
\begin{equation}\label{eq-EH}
\left\{ \begin{array}{llll}
\nabla \times E &= & i \omega   H & \mbox{ in } \mR^3\\[6pt]
\nabla \times H & =&  - i \omega  E + J & \mbox{ in }  \mR^3. 
\end{array} \right.
\end{equation}

Assume that 
\begin{equation}\label{smoothness-2}
\mbox{$(\teps_O, \tmu_O)$  is $C^2$}
\end{equation}
and $r_3/ r_2$ is large enough. We have \cite[Theorem 1.1]{Ng-Negative-Cloaking-M} 
$$
(E_\delta, H_\delta) \to (E, H) \mbox{ in } \mR^3 \setminus B_{r_3} \mbox{ as } \delta \to 0. 
$$ 
The object in $B_{2r_2} \setminus B_{r_2}$ is cloaked. 

\medskip 
We next show that the largeness condition on  $r_3/ r_2$ is necessary. More precisely, we have 

\begin{proposition}\label{pro-cloaking}  Let $\Gamma_3$ be a compact smooth curve on $\partial B_{r_3}$. Set, for $\gamma > 0$,  
\begin{equation*}
O_{3, \gamma} : = \Big\{x \in \mR^3; \dist(x, \Gamma_3) < \gamma \Big\}. 
\end{equation*}
Define, in $B_{r_3} \setminus B_{r_2}$,  
\begin{equation}\label{eps-mu-O}
\big( \teps_O , \tmu_O \big) = \left\{\begin{array}{cl} (\eps_O, \mu_O) &  \mbox{ in } (B_{r_3} \setminus B_{r_2})  \cap O_{3, \gamma}, \\[6pt]
(I, I) & \mbox{ in } (B_{r_3} \setminus B_{r_2}) \setminus O_{3, \gamma}.  
\end{array}\right.
\end{equation}
Let $J \in [L^2(\mR^3)]^3$ with compact support in $B_{R_0} \setminus B_{r_3}$ for some $R_0 > r_3$ and let $(E_\delta, H_\delta)$  be the unique radiating solution of \eqref{eq-EHdelta} in which  $(\eps_\delta, \mu_\delta)$ is given in \eqref{def-eps-mu} with $(\teps_O, \tmu_O)$ defined in \eqref{eps-mu-O}. 
Assume that with $M = \eps_O$ or $\mu_O$, 
\begin{equation}\label{cond-epsmu-S}
|\nabla M(x) \le \Lambda, \quad |\Lambda^{-1} |\xi|^2 \le \langle M (x) \xi, \xi \rangle \le \Lambda |\xi|^2  \; \;  \forall \,  \xi \in \mR^d \mbox{ for some } \Lambda \ge 1.   
\end{equation}
For all $0< \alpha < 1$, there exists $\gamma_0$  depending only  on $\alpha$, $r_1$, $r_2$, and  $\Gamma_3$ such that for $\gamma \le \gamma_0$, we have 
\begin{equation*}
\|(E_\delta, H_\delta) - (\hE, \hH)\|_{L^2(B_R \setminus B_{r_3})} \le C_R \delta^{\alpha} \| J\|_{L^2}, 
\end{equation*}
for some positive constant $C_R$ depending only  on $\alpha$, $\Lambda$, $(\eps_0, \mu_0)$,  $r_2$, $r_3$, $\Gamma_3$, $R_0$, and $R$. 
Here $(\hE, \hH)$ is the unique radiating solution of the equation 
\begin{equation*}
\left\{ \begin{array}{llll}
\nabla \times \hE &= & i \omega  \hmu \hH & \mbox{ in } \mR^3\\[6pt]
\nabla \times \hH & =&  - i \omega \heps \hE + J & \mbox{ in }  \mR^3, 
\end{array} \right.
 \mbox{ where } 
(\heps, \hmu) = \left\{
\begin{array}{cl} (\eps_O, \mu_O) & \mbox{ in } B_{r_3} \cap O_{3, \gamma}, \\[6pt]
(I, I) & \mbox{ otherwise}. 
\end{array}\right. 
\end{equation*}
\end{proposition}

As a consequence of Proposition~\ref{pro-cloaking}, the object $(\eps_c, \mu_c)$ in $B(x_3, r_0) \cap B_{r_3}$ does {\it not disappear}: cloaking is {\it not} achieved.  

\begin{proof} The proof of \Cref{pro-cloaking} is almost the same as the one of \Cref{thm-cloaking}.  Using  the notations in the proof of \Cref{thm-cloaking} with $(E_{c, \delta}, H_{c, \delta}) = (E_\delta, H_\delta)$, $(\eps_{c, \delta}, \mu_{c, \delta}) = (\eps_\delta, \mu_\delta)$, $\Omega_j = B_{r_j}$ for $j = 1, \, 2, \, 3$, and  the convention $\Gamma_1 = \emptyset$ and $O_{1, \gamma} = \emptyset$, one just needs to observe  that   $(\tE_\delta - \hE, \tH_\delta - \hH) \in [H_{\loc}(\curl, \mR^3)]^2$ is  radiating  and satisfies  
\begin{equation*} \left\{
\begin{array}{cl}
\nabla \times (\tE_\delta - \hE) = i \omega \hmu  (\tH_\delta - \hH) -  \delta \omega \mathds{1}_{B_{r_3} \setminus B_{r_2}} \cF_*I H_{c, 1, \delta} & \mbox{ in } \mR^3 \setminus (\partial O_{3, \gamma_1} \cap B_{r_3}), \\[6pt]
\nabla \times (\tH_\delta - \hH) = -  i \omega \heps (\tE_\delta - \hE)  +  \delta \omega  \mathds{1}_{B_{r_3} \setminus B_{r_2}} \cF_*I E_{c, 1, \delta} & \mbox{ in } \mR^3 \setminus (\partial O_{3, \gamma_1} \cap B_{r_3}).
\end{array}\right. 
\end{equation*}
The conclusion then follows as in the proof of \Cref{thm-cloaking}. 
\end{proof}

\appendix

\section{Proof of \Cref{lem-prepare2-0}} \label{ap-lem-prepare2-0}

We have, for $x \in \Y$,  
\begin{align*}
	\Div (M\nabla \testfn) =   p \beta  \testfn \big[p \beta r^{-2 p - 4} - ( p + 2) r^{-p - 4}\big] x \cdot Mx +p \beta r^{-p - 2} \testfn \, \Div (Mx).
		\end{align*}
An integration by parts gives
\begin{equation}\label{1.9}
	 \int_\Y   e^{\beta r^{-p}} (M x \cdot \nabla |w|^2) \,   \Div (M\nabla \testfn) =  P + Q. 
\end{equation}	
Here 
\beq
	P = P_1 + P_2 + P_3 \nonumber
\eeq{} with
\begin{equation*}\left\{
\begin{array}{rl}\dsp P_1 &= \ds- \int_{\Y} p^2 \beta^2 |w|^2 \Div \big[ r^{- 2 p - 4} (x \cdot M x) M x \big], \\[6pt]
\dsp P_2 &= \ds \int_{\Y} p  (p+2)  \beta |w|^2 \Div \big[ r^{- p - 4} (x \cdot M x) M x \big], \\[6pt]
\dsp P_3 &=   \ds\int_{\Y} 2 p \beta r^{-p -2} \Div(M x) w \nabla w \cdot Mx, 
\end{array}\right.
\end{equation*}
and
\begin{equation*}
Q =  \int_{\pY}  p \beta |w|^2   \Big( \big[p \beta r^{- 2p - 4} -( p + 2) r^{-p - 4} \big] x \cdot M x \Big) M x \cdot \nu. 
\end{equation*}
We next estimate $P$ and $Q$. By a straightforward  computation, we have  
\begin{equation*}
 - \dive \big[ r^{- 2 p - 4} (x \cdot M x) M x \big] = (2p + 4) (x \cdot Mx)^2  r^{-2p - 6}  -  r^{-2p - 4}   \Div \big[ (x \cdot Mx) Mx\big ]. 
\end{equation*}
This implies 
\begin{equation}\label{R1}
P_1 =  \int_\Y p^2 \beta^2  \Big( (2p + 4) (x \cdot Mx)^2  r^{-2p - 6}  -  r^{-2p - 4}   \Div \big[ (x \cdot Mx) Mx\big ] \Big) |w|^2.
\end{equation}
Similarly, 
\begin{equation}\label{R2}
P_2 =  -   \int_\Y p(p + 2) \beta   \Big( (p + 4) (x \cdot Mx)^2 r^{-p-6}    -  r^{-p - 4}   \Div \big[ (x \cdot Mx) Mx\big ]  \Big) |w|^2. 
\end{equation}
Using Cauchy's inequality, for $a \in \mR$, 
$$
2 |a \nabla w \cdot M x | \le |a|^2 \langle M \nabla w, \nabla w \rangle + \langle  M x, x \rangle, 
$$
we have
\begin{equation}
		|P_3| \leq \int_{\Y} p^2\beta^2 r^{-2p -4}   |\dive (Mx)|^2 \langle M  x, x \rangle  |w|^2 + \langle M \nabla w, \nabla w \rangle. 
\label{R3}
\end{equation}
Combining \eqref{R1}, \eqref{R2}, and \eqref{R3} yields 
\begin{equation*}
	P \ge   \int_\Y  \Big(p^2 \beta^2 r^{-2p -2}T_1 - \beta p (p+2) r^{- p - 2} T_2 \Big) |w|^2 - \int_\Y   \langle M \nabla w, \nabla w \rangle .
\end{equation*}
Since 
\begin{equation*}
|Q| \le   \int_{\pY}  C  \beta^2 p^2 r^{-2p - 1} |M|^2 |w|^2,  
\end{equation*}
assertion~\eqref{lem-prepare-2-1} follows.  \qed

\section{Proof of \Cref{lem1}} \label{ap-lem1}
Set
\begin{equation*}
	w = e^{\beta r^{-p}} v \quad \mbox{ equivalently }  v = e^{-\beta r^{-p}} w.
\end{equation*}
Since  $\dive \big(M \nabla (g h) \big) = 2 \nabla h \cdot  M\nabla g + h \Div(M \nabla g) + g \Div (M \nabla h)$ ($M$ is symmetric), it follows that 
\begin{align*}
	\dive (M\nabla v) = 2\beta p r^{-p  - 2}  e^{-\beta r^{-p}} x \cdot M\nabla w  +  e^{-\beta r^{-p}} \Div (M\nabla w)  + w \Div (M\nabla  e^{-\beta r^{-p}}).
\end{align*}
Using the inequality $(a + b + c)^2 \geq 2 a (b + c)$,  we obtain
\begin{equation*}
\frac{1}{2 }\big[ \Div (M\nabla v)  \big]^2 \ge 2 |\beta| p r^{-p  - 2}  e^{-\beta r^{-p}} (x \cdot M\nabla w) \Big( e^{-\beta r^{-p}} \Div (M\nabla w)  + w \Div( M\nabla  e^{-\beta r^{-p}})\Big). 
\end{equation*}
This implies 
\begin{multline}\label{p1-Lem3-0}
\int_\Y \frac{r^{p + 2}e^{2\beta r^{-p}}}{2p|\beta|}\big[\Div (M\nabla v)\big]^2 
\geq \int_\Y   2  (x \cdot M\nabla w ) \,  \Div (M\nabla w) \\[6pt]+  \int_\Y  e^{\beta r^{-p}} (Mx \cdot \nabla |w|^2) \,  \Div (M\nabla  e^{-\beta r^{-p}}). 
\end{multline}

Applying \Cref{lem-prepare1} and using \eqref{lem1-as1}, we have 
\begin{equation}\label{p1-Lem3-1}
\int_\Y   2  (x \cdot M\nabla w ) \,  \Div (M\nabla w) \ge   -   \int_{\Y} \Lambda \langle M \nabla w, \nabla w \rangle 
   -  \int_{\pY} C \Lambda^2 r   |\nabla w|^2. 
\end{equation}
Applying  \Cref{lem-prepare2},  we obtain 
\begin{multline}\label{p1-Lem3-2}
 \int_\Y    e^{\beta r^{-p}} (M x \cdot \nabla |w|^2) \,  \Div (M\nabla  e^{-\beta r^{-p}}) \\[6pt]
   \ge  \int_\Y   p^3 \beta^2 \Lambda^{-2} r^{-2p - 2}|w|^2 - \int_\Y   \langle M \nabla w, \nabla w\rangle   - \int_{\pY} C \Lambda^2 \beta^2 p^2 r^{-2p - 1}  |w|^2 . 
\end{multline}
Combining \eqref{p1-Lem3-0}, \eqref{p1-Lem3-1},  and \eqref{p1-Lem3-2} yields 
\begin{multline}\label{p1-Lem3}
\int_\Y \frac{r^{p + 2}e^{2\beta r^{-p}}}{2p|\beta|}\big[\Div (M\nabla v)\big]^2  \\[6pt] 
\ge  \int_\Y   p^3 \beta^2 \Lambda^{-2} r^{-2p - 2}   |w|^2 -  C \Lambda \langle M \nabla w. \nabla w \rangle   - \int_\pY C \Lambda^2 r \Big(  |\nabla w|^2 + \beta^2 p^2 r^{-2p -2} |w|^2 \Big ). 
\end{multline}
Since $w = e^{\beta r^{-p}} v $, 
\begin{equation*}
\nabla w =  e^{ \beta r^{-p}} ( \nabla v - p\beta r^{- p - 2}  v x ), 
\end{equation*}
we derive from \eqref{p1-Lem3}  that, for large $p$, 
\begin{multline*}
\int_\Y \frac{r^{p + 2}e^{2\beta r^{-p}}}{2p|\beta|}\big[\Div (M\nabla v)\big]^2   \ge  \int_\Y   \Lambda^{-2} p^3 \beta^2  r^{-2p - 2} e^{2 \beta r^{-p}}   |v|^2 -   C \Lambda  e^{2 \beta r^{-p}}  \langle M \nabla v, \nabla v \rangle \\[6pt]
 -\int_\pY C \Lambda^2 r e^{2 \beta r^{-p}}  \Big(  |\nabla v|^2 + \beta^2 p^2 r^{-2p -2} |v|^2 \Big).  
\end{multline*}
The conclusion follows.\qed 

\section{Proof of \Cref{lem2}} \label{ap-lem2} We have
\begin{equation}\label{s1}
- \int_\Y e^{2\beta r^{-p}} v \,  \Div (M\nabla v)  
= \int_\Y M\nabla v \cdot \nabla (e^{2\beta r^{-p}} v) - \int_{\pY} e^{2\beta r^{-p}} v M \nabla v \cdot \nu. 
\end{equation} 
It is clear that 
\begin{equation}\label{s2}
 \int_\Y M\nabla v \cdot \nabla (e^{2\beta r^{-p}} v) 
	= \int_\Y  \Big(e^{2\beta r^{-p}} M\nabla v \cdot \nabla v - 2 \beta p  r^{-p - 2}e^{2\beta r^{-p}} v M\nabla v  \cdot x \Big) 
\end{equation} 
and 
\begin{equation}\label{s3}
 \int_{\pY} e^{2\beta r^{-p}} v M \nabla v \cdot \nu \le  \int_{\pY} C |M| e^{2\beta r^{-p}} \Big(|\nabla v|^2 + |v|^2 \Big). 
\end{equation}
Combining \eqref{s1}, \eqref{s2}, and \eqref{s3} yields 
\begin{multline}\label{s4-0}
 \int_\Y e^{2\beta r^{-p}} v \,  \Div (M\nabla v)   + \int_\Y e^{2\beta r^{-p}} M\nabla v \cdot \nabla v \\[6pt] \le 
  \int_\Y   2 \beta p  r^{-p - 2}e^{2\beta r^{-p}} v M\nabla v  \cdot x +  \int_{\pY} C |M| e^{2\beta r^{-p}} \Big(|\nabla v|^2 + |v|^2 \Big). 
\end{multline}

Using Cauchy's inequality, for $a, b \in \mR$, 
$$
2 |a b M \nabla v \cdot  x | \le \frac{1}{2} |a|^2 \langle M \nabla v, \nabla v \rangle +  8 |b|^2 \langle  M x, x \rangle, 
$$
we obtain 
\begin{equation}\label{s4-1}
\int_\Y 2 \beta p  r^{-p - 2}e^{2\beta r^{-p}} v M\nabla v  \cdot x 
\le \int_\Y \frac{1}{2}e^{2\beta r^{-p}} \langle M \nabla v, \nabla v\rangle + 
C |M| p^2 \beta^2  r^{-2p - 2} e^{2\beta r^{-p}} |v|^2. 
\end{equation}
We derive from  \eqref{s4-0} and  \eqref{s4-1} that 
\begin{multline*}
 \int_\Y e^{2\beta r^{-p}} v \,  \Div (M\nabla v) +  \int_\Y \frac{1}{2} e^{2 \beta r^{-p}} \langle M \nabla v, \nabla v \rangle \\
 \le \int_\Y  C |M| \beta^2 p^2 r^{-2p -2} e^{2 \beta r^{-p}}  |v|^2+ \int_{\pY} C |M| e^{2\beta r^{-p}} \Big(|\nabla v|^2 +  |v|^2 \Big), 
\end{multline*}
which is the conclusion. \qed

\section{Proof of \Cref{lem-main}} \label{ap-lem-main}
We have, by Lemma~\ref{lem1},  
\begin{multline}\label{lem-main-1}
 \int_\Y   \Lambda^{-2} p^3 \beta^2  r^{-2p - 2} e^{2 \beta r^{-p}}   |v|^2  \le \int_\Y \frac{1}{2p|\beta|}r^{p + 2}e^{2\beta r^{-p}}\big[\Div (M\nabla v)\big]^2  \\[6pt] 
+  \int_\Y C \Lambda e^{2 \beta r^{-p}}  \langle M \nabla v, \nabla v \rangle  + \int_\pY C \Lambda^2 r e^{2 \beta r^{-p}}(  |\nabla v|^2 + \beta^2 p^2 r^{-2p -2} |v|^2). 
\end{multline}
We also have, by Lemma~\ref{lem2},  
\begin{multline}\label{lem-main-2}
 \int_\Y   \frac{1}{2} e^{2 \beta r^{-p}}  \langle M \nabla v, \nabla v \rangle 
 \le \int_\Y  C \Lambda  \beta^2 p^2 r^{-2p - 2} e^{2 \beta r^{-p}}  |v|^2 \\[6pt]+ \int_{\pY} C \Lambda e^{2\beta r^{-p}} (|\nabla v|^2 +  |v|^2) -  \int_\Y e^{2\beta r^{-p}} v \,  \Div (M\nabla v). 
\end{multline}
Combining \eqref{lem-main-1} and \eqref{lem-main-2} yields 
\begin{multline}\label{lem-main-1-1}
 \int_\Y   \Lambda^{-2} p^3 \beta^2  r^{-2p - 2} e^{2 \beta r^{-p}}   |v|^2 + \int_\Y   \frac{1}{2} e^{2 \beta r^{-p}}  \langle M \nabla v, \nabla v \rangle  \\[6pt]  \le \int_\Y \frac{1}{2p|\beta|}r^{p + 2}e^{2\beta r^{-p}}\big[\Div (M\nabla v)\big]^2  
+ \int_\Y  C \Lambda^2  \beta^2 p^2 r^{-2p - 2} e^{2 \beta r^{-p}}  |v|^2 \\[6pt]  + \int_\pY C \Lambda^2 e^{2 \beta r^{-p}}(  |\nabla v|^2 + \beta^2 p^2 r^{-2p -2} |v|^2) +  C \Lambda \int_\Y e^{2\beta r^{-p}}  |v \,  \Div (M\nabla v) |. 
\end{multline}
Using the fact 
\begin{equation*}
C \Lambda |v \dive (M \nabla v)| \le  p|\beta| |v|^2 r^{- p -2 } + \frac{C^2 \Lambda^2 }{ 4 p |\beta|} |\dive (M \nabla v)|^2 r^{p+2}, 
\end{equation*}
for large $p$, 
$$
C \Lambda^2  \beta^2 p^2 r^{-2p - 2} e^{2 \beta r^{-p}}  |v|^2 \le \frac{1}{4}  \Lambda^{-2} p^3 \beta^2  r^{-2p - 2} e^{2 \beta r^{-p}}   |v|^2, 
$$
and,  for large $p$  and $|\beta| \ge 1$, 
$$
p|\beta|  r^{- p -2 }e^{2\beta r^{-p}} |v|^2 \le  \frac{1}{4}  \Lambda^{-2} p^3 \beta^2  r^{-2p - 2} e^{2 \beta r^{-p}}   |v|^2, 
$$ 
we derive from \eqref{lem-main-1-1} that, for large $p$, 
\begin{multline*}
\int_{\Y} \testfnn \Big(p^3 \beta^2 r^{-2p - 2}  |v|^2 +  \langle M \nabla v, \nabla v \rangle  \Big)  \\[6pt]
 \le \int_\Y \frac{C_\Lambda}{p|\beta|}r^{p + 2}e^{2\beta r^{-p}}\big[\Div (M\nabla v)\big]^2   + \int_\pY C_\Lambda   e^{2 \beta r^{-p}}(  |\nabla v|^2 + \beta^2 p^2 r^{-2p -2} |v|^2). 
\end{multline*}
The proof is complete. \qed

\section{Proof of \Cref{lem-EH}} \label{ap-lem-EH}
By the trace theory, see e.g. \cite{AV96}, there exist $E_f, H_g \in H(\curl, D)$ such that  
$$
E_f \times \nu = f \mbox{ on } \partial D,  \quad H_g \times \nu = g \mbox{ on } \partial D, 
$$
$$
\| E_f \|_{H(\curl, D)} \le C \| f\|_{H^{-1/2}( \dive_\Gamma, \partial D) },
\quad \mbox{ and } \quad   \| H_g \|_{H(\curl, D)} \le C \| g\|_{H^{-1/2}(\dive_\Gamma, \partial D)}. 
$$
By considering the pair $(E - E_f \mathds{1}_{D}, H - H_g \mathds{1}_{D} \big)$, one may assume that $f = g = 0$. This fact is assumed from later on. 

An integration by parts gives 
\begin{equation}\label{lem-EH-uniqueness}
\int_{\Omega} \langle \mu^{-1} \nabla \times E, \nabla \times E \rangle =   \int_{\partial \Omega} i \omega |E \times \nu|^2  + \int_{\Omega} \omega^2 \langle \eps E, E  \rangle + \int_{\Omega} i \omega  \langle J_m, E \rangle   + \langle \mu^{-1}J_e, \nabla \times E \rangle. 
\end{equation}

In the rest, we only  establish  \eqref{lem-EH-cl1} by contradiction since the existence and uniqueness are known.  Assume that there exist sequences $\big(\eps^{(n)} \big)$, $\big( \mu^{(n)} \big)$, $\big( J_e^{(n)} \big)$, $\big(J_m^{(n)} \big) \subset [L^2(\Omega)]^3$, and $\big((E^{(n)}, H^{(n)}) \big) \subset [H(\curl, \Omega)]^2$ such that  \eqref{lem-EH-proM} holds for $(\eps^{(n)}, \mu^{(n)})$, 
\begin{equation*}
\left\{\begin{array}{cl}
\nabla \times H^{(n)} = i \omega \eps^{(n)}  E^{(n)} + J_m^{(n)} & \mbox{ in } \Omega, \\[6pt]
\nabla \times E^{(n)} = - i \omega \mu^{(n)} H^{(n)} + J_e^{(n)} & \mbox{ in } \Omega, \\[6pt]
(H^{(n)} \times \nu) \times \nu - E^{(n)} \times \nu = 0 & \mbox{ on } \partial \Omega, 
\end{array}\right. 
\end{equation*}
$$
n \|(J_e^{(n)}, J_m^{(n)})\|_{L^2(\Omega)} \le \|(E^{(n)}, H^{(n)})\|_{L^2(\Omega)} = 1. 
$$
Using \eqref{lem-EH-uniqueness}, we have 
$$
\| E^{(n)} \times \nu \|_{L^2(\Omega)} \le C. 
$$
Without loss of generality, one may assume that $(E^{(n)} \times \nu)$ converges in $H^{-1/2}(\partial \Omega)$. 
By Ascoli's theorem, one may also assume that $(\eps^{(n)}, \mu^{(n)}) \to (\eps, \mu)$ in $L^\infty(\Omega)$ for some $(\eps, \mu) \in W^{1, \infty}(\Omega)$.  We derive that $\big( E^{(n)} \big)$ is bounded in $H(\curl, \Omega)$ and 
$$
\big( \dive (\eps E^{(n)}) \big)  \mbox{ converges in }  [H^{-1}(\Omega)]^3.  
$$
Applying \cite[Lemma 1]{Ng-Superlensing-Maxwell}, one may assume that $(E^{(n)})$ converges  in $[L^2(\Omega)]^3$. 
Similarly, one may assume that $(H^{(n)})$ converges  in $[L^2(\Omega)]^3$. 

Let $(E, H)$ be the limit  of $(E^{(n)}, H^{(n)})$ in $[L^2(\Omega)]^6$. Then  $(E, H) \in [H(\curl, \Omega)]^2$ and 
\begin{equation*}
\left\{\begin{array}{cl}
\nabla \times H = i \omega \eps E  & \mbox{ in } \Omega, \\[6pt]
\nabla \times E = - i \omega \mu H  & \mbox{ in } \Omega, \\[6pt]
(H \times \nu) \times \nu - E\times \nu = 0 & \mbox{ on } \partial \Omega.  
\end{array}\right. 
\end{equation*}
It follows that  $(E, H) = (0, 0)$ in $\Omega$ by the uniqueness. This contradicts the fact $\|(E, H)\|_{L^2(\Omega)} = \lim_{n \to + \infty} \| (E_n, H_n) \|_{L^2(\Omega)} = 1$.  Therefore, \eqref{lem-EH-cl1} holds. The proof is complete. \qed

\section{Proof of  \Cref{lem-M}} \label{ap-lem-M}

Before giving the proof of \Cref{lem-M}, we recall some properties of the spherical Bessel and Neumann functions and the Bessel and Neumann functions of large order. We first introduce,  for $n \ge 1$, 
\begin{equation}\label{def-jn}
\hat j_n(t) =1 \cdot 3 \cdots (2n + 1) j_n(t) \quad \mbox{ and } \quad  \hat y_n = -  \frac{y_n(t)}{1 \cdot 3 \cdots (2n-1)} ,  
\end{equation}
where $j_n$ and $y_n$   are the spherical Bessel and Neumann functions. Then, see, e.g.  \cite[(2.37) and (2.38)]{CK-Inverse}),   as $n \to + \infty$,  
\begin{equation}\label{jy-n}
\hat j_n(r) = r^n \big[1 + O(1/n) \big] \quad \hat y_n(r) = r^{-n-1} \big[1 + O(1/n) \big].  
\end{equation}
One also has, see, e.g.   \cite[(2.36) and (3.56)]{CK-Inverse},
\begin{equation}\label{W3}
  j_n(r)  y_n'(r) -   j_n'(r) y_n(r) = \frac{1}{ r^2}. 
\end{equation}

In what follows,  for $-n\leq m\leq n, n\in \N$, denote $Y_n^m$ the spherical harmonic function of order $n$ and degree $m$ and set    
$$
U_n^m(\hat x) := \nabla_{\partial B_1}Y_n^m(\hat x) \quad \mbox{ and } \quad V_n^m(\hat x) := \hat x \times U_n^m(\hat x) \mbox{ for } \hat x \in \partial B_1.
$$
We recall that $Y_n^m(\hat x)\hat x$, $U_n^m(\hat x)$, and $V_n^m(\hat x)$ for $-n\leq m\leq n, n\in \N$ form an orthonormal basis of $[L^2(\partial B_1)]^3$.  

\begin{proof}[Proof of \Cref{lem-M}]Without loss of generality, one may assume that $\omega =1$.  One then can represent ${E, H}$ in $B_{R_3} \setminus B_{R_1}$ as follows,  see, e.g. \cite{KH15}, with $r = |x|$ and $\hat x = x/ |x|$, 
\begin{align*}
E(x) =  &    \sum_{n=1}^\infty \sum_{|m| \le n} \sqrt{n (n +1)}  \frac{\alpha_{1, n}^m \hj_{n}(r) + \alpha_{2, n}^m \hy_n(r)}{r} Y_n^m(\hat x) \hat x  \\[6pt]
& + \sum_{n=1}^\infty \sum_{|m| \le n}
 \frac{\left(r \big[\alpha_{1, n}^m  \hj_n(r) + \alpha_{2, n}^m \hy_n (r) \big] \right)'}{r} U_n^m (\hat x) \\[6pt]
& + \sum_{n=1}^\infty \sum_{|m| \le n}  \big[\beta_{1, n}^m \hj_n(r) + \beta_{2, n}^m \hy_n(r) \big] V_n^m(\hat x)
\end{align*}
and 
\begin{align*}
H (x)=  &   i \sum_{n=1}^\infty \sum_{|m| \le n} \sqrt{n (n +1)}  \frac{\beta_{1, n}^m \hj_{n}(r) + \beta_{2, n}^m \hy_n(r)}{r} Y_n^m(\hat x) \hat x  \\[6pt]
& + i \sum_{n=1}^\infty \sum_{|m| \le n}
 \frac{\left(r \big[\beta_{1, n}^m  \hj_n(r) + \beta_{2, n}^m \hy_n (r) \big]  \right)'}{r} U_n^m (\hat x) \\[6pt]
& + i  \sum_{n=1}^\infty \sum_{|m| \le n}  \big[\alpha_{1, n}^m \hj_n(r) + \alpha_{2, n}^m \hy_n(r) \big] V_n^m(\hat x). 
\end{align*}
One can then check that 
\begin{equation*}
\| (E \times \nu, H \times \nu)\|_{H^{-1/2}(\dive_\Gamma, \partial B_{r})}^2 \sim \sum_{n =1}^{\infty} \sum_{|m| \le n} \sum_{j=1}^2 n^3 \big(  |\alpha_{j, n}^{m}|^2 + |\beta_{j, n}^{m}|^2 \big) r^{2 n}. 
\end{equation*}
The conclusion now follows  from the interpolation. 
\end{proof}

\providecommand{\bysame}{\leavevmode\hbox to3em{\hrulefill}\thinspace}
\providecommand{\MR}{\relax\ifhmode\unskip\space\fi MR }
\providecommand{\MRhref}[2]{%
  \href{http://www.ams.org/mathscinet-getitem?mr=#1}{#2}
}
\providecommand{\href}[2]{#2}


\begin{thebibliography}{10}

\bibitem{Agmon}
Shmuel Agmon, \emph{Lectures on elliptic boundary value problems}, Prepared for
  publication by B. Frank Jones, Jr. with the assistance of George W. Batten,
  Jr. Van Nostrand Mathematical Studies, No. 2, D. Van Nostrand Co., Inc.,
  Princeton, N.J.-Toronto-London, 1965. \MR{0178246}

\bibitem{AR09}
Giovanni Alessandrini, Luca Rondi, Edi Rosset, and Sergio Vessella, \emph{The
  stability for the {C}auchy problem for elliptic equations}, Inverse Problems
  \textbf{25} (2009), no.~12, 123004, 47. \MR{2565570}

\bibitem{AV96}
Ana Alonso and Alberto Valli, \emph{Some remarks on the characterization of the
  space of tangential traces of {$H({\rm rot};\Omega)$} and the construction of
  an extension operator}, Manuscripta Math. \textbf{89} (1996), no.~2,
  159--178. \MR{1371994}

\bibitem{A-M13}
Habib Ammari, Giulio Ciraolo, Hyeonbae Kang, Hyundae Lee, and Graeme~W. Milton,
  \emph{Spectral theory of a {N}eumann-{P}oincar\'{e}-type operator and
  analysis of cloaking due to anomalous localized resonance}, Arch. Ration.
  Mech. Anal. \textbf{208} (2013), no.~2, 667--692. \MR{3035988}

\bibitem{Ar57}
Nachman Aronszajn, \emph{A unique continuation theorem for solutions of
  elliptic partial differential equations or inequalities of second order}, J.
  Math. Pures Appl. (9) \textbf{36} (1957), 235--249. \MR{92067}

\bibitem{BGQ}
Guillaume Baffou, Christian Girard, and Romain Quidant, \emph{Mapping heat
  origin in plasmonic structures}, Phys. Rev. Lett. \textbf{104} (2010),
  136805.

\bibitem{BCX12}
John~M. Ball, Yves Capdeboscq, and Basang Tsering-Xiao, \emph{On uniqueness for
  time harmonic anisotropic {M}axwell's equations with piecewise regular
  coefficients}, Math. Models Methods Appl. Sci. \textbf{22} (2012), no.~11,
  1250036, 11. \MR{2974174}

\bibitem{BBH}
Fabrice Bethuel, Ha\"{\i}m Brezis, and Fr\'{e}d\'{e}ric H\'{e}lein,
  \emph{Ginzburg-{L}andau vortices}, Progress in Nonlinear Differential
  Equations and their Applications, vol.~13, Birkh\"{a}user Boston, Inc.,
  Boston, MA, 1994. \MR{1269538}

\bibitem{BCC12}
Anne-Sophie Bonnet-Ben~Dhia, Lucas Chesnel, and Patrick Ciarlet, Jr.,
  \emph{{$T$}-coercivity for scalar interface problems between dielectrics and
  metamaterials}, ESAIM Math. Model. Numer. Anal. \textbf{46} (2012), no.~6,
  1363--1387. \MR{2996331}

\bibitem{BK05}
Jean Bourgain and Carlos~E. Kenig, \emph{On localization in the continuous
  {A}nderson-{B}ernoulli model in higher dimension}, Invent. Math. \textbf{161}
  (2005), no.~2, 389--426. \MR{2180453}

\bibitem{Carleman}
Torsten Carleman, \emph{Sur un probl\`eme d'unicit\'{e} pur les syst\`emes
  d'\'{e}quations aux d\'{e}riv\'{e}es partielles \`a deux variables
  ind\'{e}pendantes}, Ark. Mat., Astr. Fys. \textbf{26} (1939), no.~17, 9.
  \MR{0000334}

\bibitem{CK-Inverse}
David Colton and Rainer Kress, \emph{Inverse acoustic and electromagnetic
  scattering theory}, second ed., Applied Mathematical Sciences, vol.~93,
  Springer-Verlag, Berlin, 1998. \MR{1635980}

\bibitem{Coron07}
Jean-Michel Coron, \emph{Control and nonlinearity}, Mathematical Surveys and
  Monographs, vol. 136, American Mathematical Society, Providence, RI, 2007.
  \MR{2302744}

\bibitem{CS85}
Martin Costabel and Ernst Stephan, \emph{A direct boundary integral equation
  method for transmission problems}, J. Math. Anal. Appl. \textbf{106} (1985),
  no.~2, 367--413. \MR{782799}

\bibitem{Lai-Cloak}
Jian-Wen Dong, Hui~Huo Zheng, Yun Lai, He-Zhou Wang, and C.~T. Chan,
  \emph{Metamaterial slab as a lens, a cloak, or an intermediate}, Phys. Rev. B
  \textbf{83} (2011), 115124.

\bibitem{FI96}
Andrei~V. Fursikov and O.~Yu. Imanuvilov, \emph{Controllability of evolution
  equations}, Lecture Notes Series, vol.~34, Seoul National University,
  Research Institute of Mathematics, Global Analysis Research Center, Seoul,
  1996. \MR{1406566}

\bibitem{GL86}
Nicola Garofalo and Fang-Hua Lin, \emph{Monotonicity properties of variational
  integrals, {$A_p$} weights and unique continuation}, Indiana Univ. Math. J.
  \textbf{35} (1986), no.~2, 245--268. \MR{833393}

\bibitem{GL87}
Nicola Garofalo and Fang-Hua Lin, \emph{Unique continuation for elliptic operators: a
  geometric-variational approach}, Comm. Pure Appl. Math. \textbf{40} (1987),
  no.~3, 347--366. \MR{882069}

\bibitem{Hadamard}
Jacques Hadamard, \emph{{Sur les fonction enti\`eres}}, Bull. Soc. Math. France
  \textbf{24} (1896), 94--96.

\bibitem{HorIII}
Lars H\"{o}rmander, \emph{The analysis of linear partial differential
  operators. {III}}, Grundlehren der Mathematischen Wissenschaften [Fundamental
  Principles of Mathematical Sciences], vol. 274, Springer-Verlag, Berlin,
  1985, Pseudodifferential operators. \MR{781536}

\bibitem{Jain}
Prashant~K Jain, Kyeong~Seok Lee, Ivan~H El-Sayed, and Mostafa~A El-Sayed,
  \emph{{Calculated absorption and scattering properties of gold nanoparticles
  of different size, shape, and composition: Applications in biomedical imaging
  and biomedicine}}, J. Phys. Chem. B \textbf{110} (2006), 7238--7248.

\bibitem{JK85}
David Jerison and Carlos~E. Kenig, \emph{Unique continuation and absence of
  positive eigenvalues for {S}chr\"{o}dinger operators}, Ann. of Math. (2)
  \textbf{121} (1985), no.~3, 463--494, With an appendix by E. M. Stein.
  \MR{794370}

\bibitem{KSW15}
Carlos Kenig, Luis Silvestre, and Jenn-Nan Wang, \emph{{On Landis' conjecture
  in the plane}}, Comm. Partial Differential Equations \textbf{40} (2015),
  no.~4, 766--789. \MR{3299355}

\bibitem{KRS87}
Carlos~E. Kenig, Alberto Ruiz, and Christopher~D. Sogge, \emph{Uniform
  {S}obolev inequalities and unique continuation for second order constant
  coefficient differential operators}, Duke Math. J. \textbf{55} (1987), no.~2,
  329--347. \MR{894584}

\bibitem{KSU07}
Carlos~E. Kenig, Johannes Sj\"{o}strand, and Gunther Uhlmann, \emph{The
  {C}alder\'{o}n problem with partial data}, Ann. of Math. (2) \textbf{165}
  (2007), no.~2, 567--591. \MR{2299741}

\bibitem{KH15}
Andreas Kirsch and Frank Hettlich, \emph{The mathematical theory of
  time-harmonic {M}axwell's equations}, Applied Mathematical Sciences, vol.
  190, Springer, Cham, 2015, Expansion-, integral-, and variational methods.
  \MR{3288313}

\bibitem{KT01}
Herbert Koch and Daniel Tataru, \emph{Carleman estimates and unique
  continuation for second-order elliptic equations with nonsmooth
  coefficients}, Comm. Pure Appl. Math. \textbf{54} (2001), no.~3, 339--360.
  \MR{1809741}

\bibitem{KLSW14}
Robert~V. Kohn, Jianfeng Lu, Ben Schweizer, and Michael~I. Weinstein, \emph{A
  variational perspective on cloaking by anomalous localized resonance}, Comm.
  Math. Phys. \textbf{328} (2014), no.~1, 1--27. \MR{3196978}

\bibitem{KOVW10}
Robert~V. Kohn, Daniel Onofrei, Michael~S. Vogelius, and Michael~I. Weinstein,
  \emph{Cloaking via change of variables for the {H}elmholtz equation}, Comm.
  Pure Appl. Math. \textbf{63} (2010), no.~8, 973--1016. \MR{2642383}

\bibitem{Lai1}
Yun Lai, Huanyang Chen, Zhao-Qing Zhang, and CT. Chan, \emph{{Complementary
  media invisibility cloak that cloaks objects at a distance outside the
  cloaking shell}}, Phys. Rev. Lett. \textbf{102} (2009), 093901.

\bibitem{Landis}
Evgenii~M. Landis, \emph{Some questions in the qualitative theory of
  second-order elliptic equations (case of several independent variables)},
  Uspehi Mat. Nauk \textbf{18} (1963), no.~1 (109), 3--62. \MR{0150437}

\bibitem{RL12}
J\'{e}r\^{o}me Le~Rousseau and Gilles Lebeau, \emph{On {C}arleman estimates for
  elliptic and parabolic operators. {A}pplications to unique continuation and
  control of parabolic equations}, ESAIM Control Optim. Calc. Var. \textbf{18}
  (2012), no.~3, 712--747. \MR{3041662}

\bibitem{LR95}
Gilles Lebeau and Luc Robbiano, \emph{Contr\^{o}le exact de l'\'{e}quation de
  la chaleur}, Comm. Partial Differential Equations \textbf{20} (1995),
  no.~1-2, 335--356. \MR{1312710}

\bibitem{Leis}
Rolf Leis, \emph{Initial-boundary value problems in mathematical physics}, B.
  G. Teubner, Stuttgart; John Wiley \& Sons, Ltd., Chichester, 1986.
  \MR{841971}

\bibitem{Meshkov91}
V.~Z. Meshkov, \emph{On the possible rate of decrease at infinity of the
  solutions of second-order partial differential equations}, Mat. Sb.
  \textbf{182} (1991), no.~3, 364--383. \MR{1110071}

\bibitem{MN06}
Graeme~W. Milton and Nicolae-Alexandru~P. Nicorovici, \emph{On the cloaking
  effects associated with anomalous localized resonance}, Proc. R. Soc. Lond.
  Ser. A Math. Phys. Eng. Sci. \textbf{462} (2006), no.~2074, 3027--3059.
  \MR{2263683}

\bibitem{Ng-Complementary}
Hoai-Minh Nguyen, \emph{Asymptotic behavior of solutions to the {H}elmholtz
  equations with sign changing coefficients}, Trans. Amer. Math. Soc.
  \textbf{367} (2015), no.~9, 6581--6595. \MR{3356948}

\bibitem{Ng-CALR}
Hoai-Minh Nguyen, \emph{Cloaking via anomalous localized resonance for doubly
  complementary media in the quasistatic regime}, J. Eur. Math. Soc. (JEMS)
  \textbf{17} (2015), no.~6, 1327--1365. \MR{3353803}

\bibitem{Ng-Superlensing}
Hoai-Minh Nguyen, \emph{Superlensing using complementary media}, Ann. Inst. H.
  Poincar\'{e} Anal. Non Lin\'{e}aire \textbf{32} (2015), no.~2, 471--484.
  \MR{3325246}

\bibitem{Ng-Negative-Cloaking}
Hoai-Minh Nguyen, \emph{Cloaking using complementary media in the quasistatic regime},
  Ann. Inst. H. Poincar\'{e} Anal. Non Lin\'{e}aire \textbf{33} (2016), no.~6,
  1509--1518. \MR{3569240}

\bibitem{Ng-WP}
Hoai-Minh Nguyen, \emph{Limiting absorption principle and well-posedness for the
  {H}elmholtz equation with sign changing coefficients}, J. Math. Pures Appl.
  (9) \textbf{106} (2016), no.~2, 342--374. \MR{3515306}

\bibitem{Ng-CALR-O}
Hoai-Minh Nguyen, \emph{Cloaking an arbitrary object via anomalous localized resonance:
  the cloak is independent of the object}, SIAM J. Math. Anal. \textbf{49}
  (2017), no.~4, 3208--3232. \MR{3689138}

\bibitem{Ng-Superlensing-Maxwell}
Hoai-Minh Nguyen, \emph{Superlensing using complementary media and reflecting
  complementary media for electromagnetic waves}, Adv. Nonlinear Anal.
  \textbf{7} (2018), no.~4, 449--467. \MR{3871415}

\bibitem{Ng-Negative-Cloaking-M}
Hoai-Minh Nguyen, \emph{Cloaking using complementary media for electromagnetic waves},
  ESAIM Control Optim. Calc. Var. \textbf{25} (2019), Art. 29, 19. \MR{3990650}

\bibitem{Ng-CALR-F}
Hoai-Minh Nguyen, \emph{Cloaking via anomalous localized resonance for doubly
  complementary media in the finite frequency regime}, J. Anal. Math.
  \textbf{138} (2019), no.~1, 157--184. \MR{3996036}

\bibitem{Ng-CALR-M}
Hoai-Minh Nguyen, \emph{The invisibility via anomalous localized resonance of a source
  for electromagnetic waves}, Res. Math. Sci. \textbf{6} (2019), no.~4, Paper
  No. 32, 22. \MR{4011564}

\bibitem{Ng-Survey}
Hoai-Minh Nguyen, \emph{Negative index materials: some mathematical perspectives}, Acta
  Math. Vietnam. \textbf{44} (2019), no.~2, 325--349. \MR{3947973}

\bibitem{MinhLoc2}
Hoai-Minh Nguyen and Loc~Hoang Nguyen, \emph{Cloaking using complementary media
  for the {H}elmholtz equation and a three spheres inequality for second order
  elliptic equations}, Trans. Amer. Math. Soc. Ser. B \textbf{2} (2015),
  93--112. \MR{3418646}

\bibitem{NgSil}
Hoai-Minh Nguyen and Swarnendu Sil, \emph{Limiting {A}bsorption {P}rinciple and
  {W}ell-{P}osedness for the {T}ime-{H}armonic {M}axwell {E}quations with
  {A}nisotropic {S}ign-{C}hanging {C}oefficients}, Comm. Math. Phys.
  \textbf{379} (2020), no.~1, 145--176. \MR{4152269}

\bibitem{Ng-Vinoles}
Hoai-Minh Nguyen and Valentin Vinoles, \emph{Electromagnetic wave propagation
  in media consisting of dispersive metamaterials}, C. R. Math. Acad. Sci.
  Paris \textbf{356} (2018), no.~7, 757--775. \MR{3811748}

\bibitem{NgV-A}
Hoai-Minh Nguyen and Michael~S. Vogelius, \emph{Full range scattering estimates
  and their application to cloaking}, Arch. Ration. Mech. Anal. \textbf{203}
  (2012), no.~3, 769--807. \MR{2928133}

\bibitem{Tu}
Tu~Nguyen and Jenn-Nan Wang, \emph{Quantitative uniqueness estimate for the
  {M}axwell system with {L}ipschitz anisotropic media}, Proc. Amer. Math. Soc.
  \textbf{140} (2012), no.~2, 595--605. \MR{2846328}

\bibitem{NMM94}
Nicolae-Alexandru~P. Nicorovici, Ross~C. McPhedran, and Graeme~W. Milton,
  \emph{{Optical and dielectric properties of partially resonant composites}},
  Phys. Rev. B \textbf{49} (1994), 8479--8482.

\bibitem{PendryNegative}
John~B. Pendry, \emph{Negative refraction makes a perfect lens}, Phys. Rev.
  Lett. \textbf{85} (2000), 3966--3969.

\bibitem{Protter60}
Murray~H. Protter, \emph{Unique continuation for elliptic equations}, Trans.
  Amer. Math. Soc. \textbf{95} (1960), 81--91. \MR{113030}

\bibitem{SSS01}
Richard~A Shelby, David~R Smith, and Seldon Schultz, \emph{Experimental
  verification of a negative index of refraction}, Science \textbf{292} (2001),
  no.~5514, 77--79.

\bibitem{Veselago}
Victor~G. Veselago, \emph{{The electrodynamics of substances with
  simultaneously negative values of $\eps$ and $\mu$}}, Usp. Fiz. Nauk
  \textbf{92} (1964), 517--526.

\bibitem{Y09}
Masahiro Yamamoto, \emph{Carleman estimates for parabolic equations and
  applications}, Inverse Problems \textbf{25} (2009), no.~12, 123013, 75.
  \MR{3460049}

\end{thebibliography}
\end{document}